\documentclass{article}
\usepackage{amsmath,amssymb,theorem}
\usepackage{verbatim}

\usepackage{color}
\usepackage{makeidx}

\theorembodyfont{\upshape}

 \makeatletter
 \@addtoreset{equation}{section}
 \makeatother

 \newtheorem{ittheorem}{Theorem}
 \newtheorem{itlemma}{Lemma}
 \newtheorem{itproposition}{Proposition}
 \newtheorem{itdefinition}{Definition}

 \newtheorem{itremark}{Remark}
 \newtheorem{itclaim}{Claim}
 \newtheorem{itcorollary}{\bf Corollary}

 \newenvironment{theorem}{\addtocounter{equation}{1}
 \begin{ittheorem}}{\end{ittheorem}}

 \newenvironment{lemma}{\addtocounter{equation}{1}
 \begin{itlemma}}{\end{itlemma}}

 \newenvironment{proposition}{\addtocounter{equation}{1}
 \begin{itproposition}}{\end{itproposition}}

 \newenvironment{definition}{\addtocounter{equation}{1}
 \begin{itdefinition}}{\end{itdefinition}}

 \newenvironment{remark}{\addtocounter{equation}{1}
 \begin{itremark}}{\end{itremark}}

 \newenvironment{claim}{\addtocounter{equation}{1}
 \begin{itclaim}}{\end{itclaim}}

 \newenvironment{proof}{\noindent {\bf Proof.\,}
 }{\hspace*{\fill}$\qed$\medskip}

 \newenvironment{corollary}{\addtocounter{equation}{1}
 \begin{itcorollary}}{\end{itcorollary}}

 \newcommand{\be}[1]{\begin{eqnarray*}\label{#1}}
 \newcommand{\ee}{\end{eqnarray*}}

 \newcommand{\bl}[1]{\begin{lemma}\label{#1}}
 \newcommand{\el}{\end{lemma}}

 \newcommand{\br}[1]{\begin{remark}\label{#1}}
 \newcommand{\er}{\end{remark}}

 \newcommand{\bt}[1]{\begin{theorem}\label{#1}}
 \newcommand{\et}{\end{theorem}}

 \newcommand{\bd}[1]{\begin{definition}\label{#1}}
 \newcommand{\ed}{\end{definition}}

 \newcommand{\bcl}[1]{\begin{claim}\label{#1}}
 \newcommand{\ecl}{\end{claim}}

 \newcommand{\bp}[1]{\begin{proposition}\label{#1}}
 \newcommand{\ep}{\end{proposition}}

 \newcommand{\bc}[1]{\begin{corollary}\label{#1}}
 \newcommand{\ec}{\end{corollary}}

 \newcommand{\bpr}{\begin{proof}}
 \newcommand{\epr}{\end{proof}}

 \newcommand{\bi}{\begin{itemize}}
 \newcommand{\ei}{\end{itemize}}

 \newcommand{\ben}{\begin{enumerate}}
 \newcommand{\een}{\end{enumerate}}

\def\un{{1\cdots 1}}

\def\zll{\{\,0,\dots,l-1\,\}}
\def\ul{\{\,1,\dots,\ell\,\}}
\def\zm{\{\,0,\dots,m\,\}}
\def\um{\{\,1,\dots,m\,\}}

\def\zu{\{\,0,1\,\}}
\def\zul{\{\,0,1\,\}^\ell}
\def\zulc{\Big(\{\,0,1\,\}^\ell\Big)^2}
\def\zulm{\{\,0,1\,\}^\ell}

\def\zul{\{\,0,1\,\}^\ell}

\def\tde{\tau_\delta}
\def\zulm{\big(\{\,0,1\,\}^\ell\big)^m} 
\def\uro{\smash{{U}^{\!\!\!\!\raise5pt\hbox{$\scriptstyle o$}}}}

\def\ulmo{\{\,1,\dots,\ell-1\,\}}
\def\um{\{\,1,\dots,m\,\}}

\def\bp{{\overline{p}}}

\def\bp{{\overline{p}}}

\def\sel{\text{\rm sel}}
\def\switch{\text{\rm switch}}
\def\rank{\text{\rm rank}}

 \def \ba {\begin{array}}
 \def \ea {\end{array}}

 \def \qed {{\heartsuit\hfill}}
 
 \def \R {{\mathbb R}}
 
 \def \N {{\mathbb N}}

 \def \cS {{\cal S}}
 \def \cE {{\cal E}}

\def\ve{\varepsilon}
 \def \cC {{\cal C}}

 \def \cB {{\cal B}}
 \def \cM {{\cal M}}

 \def \cP {{\cal P}}

\def \qed {{\square\hfill}}

 \def \m {{\mu}}

 \def\cB{{\cal B}} \def\cC{{\cal C}} 
\def\cE{{\cal E}}   
\def\cI{{\cal I}}   
\def\cM{{\cal M}}   \def\cP{{\cal P}}
  \def\cS{{\cal S}}

 \def \cMH {{\cal M}_H}

\def \qed {{\square\hfill}}

\def\R{{\mathbb R}}

\def\N{{\mathbb N}}

\def\eqref#1{(\ref{#1})}

\def\card{\text{card}\,}

\makeindex

\begin{document}

\title{The quasispecies regime for\\ the simple genetic algorithm\\
with ranking selection}

 \author{
Rapha\"el Cerf
\\
DMA, 
{\'E}cole Normale Sup\'erieure
}

\maketitle



\begin{abstract}
\noindent
We study the simple genetic algorithm with a 
ranking selection mechanism (linear ranking or tournament).
We denote by~$\ell$ the length of
the chromosomes, by $m$
the population size, by $p_C$ the  
crossover probability and by $p_M$ 
the mutation probability.
We introduce 
a parameter~$\sigma$, called the selection drift, which measures
the selection intensity of the fittest chromosome.
We show that the dynamics of the genetic algorithm depend in
a critical way on the parameter
$$\pi
\,=\,\sigma(1-p_C)(1-p_M)^\ell\,.$$
%
If $\pi<1$, then the genetic algorithm operates
in a 
disordered regime: an advantageous mutant disappears with probability
larger than $1-1/m^\beta$, where $\beta$ is a positive exponent.
If $\pi>1$, then the genetic algorithm operates in a 
quasispecies regime: an advantageous mutant invades a positive
fraction of the population with probability larger than a constant~$p^*$
(which does not depend on~$m$).
We estimate next
the probability of the occurrence of a catastrophe
(the whole population falls
below a fitness level which was previously reached by a 
positive fraction of the population).
The asymptotic results suggest the following rules:

\noindent
$\bullet$ 
$\pi=\sigma(1-p_C)(1-p_M)^\ell$
should be slightly larger than~$1$; 

\noindent
$\bullet$ 
$p_M$ should be of order $1/\ell$;

\noindent
$\bullet$ 
$m$ should be larger than $\ell\ln\ell$;

\noindent
$\bullet$ 
the running time should be of exponential order in $m$.

\noindent
The first condition requires that
$ \ell p_M +p_C< \ln\sigma$. 
These conclusions must be taken with great care: 
they come from 
an asymptotic regime, and it is a formidable task
to understand the relevance of this regime for a real--world problem.
At least, we hope that these conclusions provide interesting
guidelines for the practical implementation of the 
simple genetic algorithm.
\end{abstract}



\section{Introduction}
Genetic algorithms
are search procedures based on the genetic mechanisms which guide natural evolution: 
selection, crossover and mutation.
The most cited initial references on genetic 
algorithms are the beautiful books of Holland \cite{HO}, who tried to
initiate a theoretical analysis of these processes, and of Goldberg \cite{GO},
who made a very attractive exposition of these algorithms.
The success of genetic algorithms over the years has been amazing.
They have been used to attack 
optimization problems of every possible kind.
Numerous variants, extensions and generalizations of the basic
genetic algorithms have been developed.
The literature on genetic algorithms is now so huge that it is beyond
my ability to compile a decent reasonable review.
Unfortunately, the theoretical understanding of the mechanisms at work in
genetic algorithms is still far from satisfactory.

We study here the
simple genetic algorithm with a ranking selection mechanism.
Ranking selection means that the selection mechanism depends only on the
ranking of the chromosomes according to the fitness function.
We consider mainly two popular selection mechanisms: linear ranking selection and
tournament selection.
The simple genetic algorithm operates on binary strings of length~$\ell$,
called the chromosomes.
The population size is denoted by~$m$.
We use the standard single point crossover and the crossover probability
is denoted by~$p_C$.
We use independent parallel mutation at each bit
and the mutation probability
is denoted by~$p_M$.
We introduce 
a parameter~$\sigma$, called the selection drift, which measures
the selection intensity of the fittest chromosome.
For linear ranking selection with parameters $\eta^-,\eta^+$, the 
selection drift~$\sigma$ is equal to $\eta^+$.
For tournament selection with parameter~$t$, the
selection drift~$\sigma$ is equal to $t$.
We show that the dynamics of the simple genetic algorithm depend in
a critical way on the parameter
$$\pi
\,=\,\sigma(1-p_C)(1-p_M)^\ell\,.$$
%
Heuristically, the parameter~$\pi$ might be interpreted as the mean
number of offsprings of the fittest chromosome from one generation
to the next. 
We prove that:
\medskip

\noindent
$\bullet$ If $\pi<1$, then the genetic algorithm operates
in a 
disordered regime: an advantageous mutant disappears with probability
larger than $1-1/m^\beta$, where $\beta$ is a positive exponent.
\medskip

\noindent
$\bullet$ If $\pi>1$, then the genetic algorithm operates in a 
quasispecies regime: an advantageous mutant invades a positive
fraction of the population with probability larger than a constant~$p^*$
(which does not depend on~$m$).
\medskip

\noindent
These results hold in the limit of large populations, when~$m$ grows
to~$\infty$.
One is naturally led to think that the parameters of the genetic algorithm
should be adjusted so that $\pi$ is larger than~$1$.
Yet we think that the most interesting regime is when~$\pi$ is 
only slightly
larger than~$1$.
Indeed, in order to increase~$\pi$, the mutation and 
crossover probabilities~$p_M$ and~$p_C$ should be decreased and this would
slow down the exploration of the space.
However an efficient search procedure should realize a delicate balance between
the exploration mechanism and the selection mechanism. 
This general idea is present 
in numerous works dealing with random optimization 
\cite{OCH,RPJ,NP3}.
Another reason is that we wish to avoid the premature convergence
of the genetic algorithm, i.e., an excessive concentration of the
population on the current best chromosome.
This problem has been encountered in practice and it is discussed
in several works on genetic algorithms (see for instance \cite{PC,XG,ZN}).
It seems therefore more judicious to choose 
``large" values of~$p_M$ and~$p_C$ compatible with the condition~$\pi>1$.
This means that the mutation probability should be of 
order $1/\ell$; more precisely, the condition $\pi>1$ implies that
$$ \ell p_M +p_C \,<\, \ln\sigma \,.$$
In particular, the crossover probability cannot be too large
in order to avoid the disordered regime.
It has already been observed in the practice
of genetic algorithms that it is sensible to take the mutation
probability $p_M$ of the same order than $1/\ell$.

Another outcome of our study concerns 
the occurrence of a catastrophe and the influence of
the population size.
Loosely speaking, a catastrophe occurs if the whole population falls
below a fitness level which was previously reached by a 
positive fraction of the population.
A straightforward strategy to prevent the occurrence of a catastrophe
is to use ``elitism", i.e., to retain automatically the best chromosome
from one generation to another.
Yet, in the quasispecies regime,
the simple genetic algorithm is robust enough to avoid
catastrophes for a very long time.
We prove that, when $\pi>1$, a catastrophe occurs typically after
a number of generations which is of exponential order in~$m$,
in fact of order $\exp(c^*m)$, where $c^*$ is a constant
depending on $\pi$ only.
With a small population size, the danger is that
a catastrophe might occur before the 
genetic algorithm
succeeds in finding an advantageous mutant. 
Thus the genetic algorithm will work efficiently only if the
population size is sufficiently large.
Suppose that the population is stuck on a local maximum.
The typical time to 
discover an advantageous mutant is of order $p_M^{-\Delta}$, where 
$0\leq\Delta\leq\ell$ and
$\Delta$ depends on the current population 
(rigorous estimates were derived in 
\cite{CEFR}).
We wish to ensure that,
with high probability, this discovery
will occur before a catastrophe.
So we should have
$$m\,\gg\,\frac{\Delta}{c^*}\ln \frac{1}{p_M}\,.$$
A natural strategy would be to take~$m$ very large.
Yet there is another practical constraint:
we wish to minimize the number of evaluations
of the fitness function.
Therefore we aim for the smallest
population size compatible with the desired goal. 
To take into account these contradictory constraints,
we suggest that the parameters $\sigma,p_C,p_M,\ell,m$ should 
be adjusted according
to the following rules:
\medskip

\noindent
$\bullet$ $\pi=\sigma(1-p_C)(1-p_M)^\ell$
should be slightly larger than~$1$; 
\medskip

\noindent
$\bullet$ $p_M$ should be of order $1/\ell$;
\medskip

\noindent
$\bullet$ $m$ should be larger than $\ell\ln\ell$;
\medskip

\noindent
$\bullet$ 
the running time should be of exponential order in $m$.
\medskip

\noindent
These conclusions must be taken with great care: they come
from 
an asymptotic regime, and it is a formidable task
to understand the relevance of this regime for a real--world problem.
At least, we hope that these conclusions provide interesting
guidelines for the practical implementation of the genetic algorithm.

We provide also sufficient 
conditions on the fitness function~$f$ ensuring that,
starting from
any population, the hitting time of 
the optimal chromosomes is polynomial
in~$\ell$.
We close with a condition ensuring the concentration of the invariant
probability measure of the genetic algorithm
on populations containing optimal chromosomes.
This condition reads:
$$\pi
\,=\,\sigma(1-p_C)(1-p_M)^\ell\,,\quad
p_M\,\geq\,\frac{c^*}{\ell}\,,\quad
m\geq m_0\,,\quad
m\geq c^*\ell\ln\ell\,.
$$
Here $c^*$ and $m_0$ are constants
depending on $\pi$ only.
This result is certainly less relevant in practice than the previous
ones, however it demonstrates the asymptotic validity of the
genetic algorithm and it is reassuring to know that
the genetic algorithm works in this asymptotic regime.
The other
good news is that this condition holds uniformly with respect to the
fitness function.
Hence a population of size of order $\ell\ln\ell$ is enough to search
an arbitrary function on the space $\zul$.

By the way,
the results obtained here vindicate a conjectural picture
outlined in \cite{CGA}.
Namely,
the genetic algorithm running on a fitness landscape is
a finite population model, approximating an infinite population model.
This infinite model presents several phase transitions, depending on the
geometry of the fitness landscape.
In a way, there is a phase transition associated to each local maximum.

The results presented here have been derived with the help of ideas coming
from the quasispecies theory.
In 1971,
Manfred
Eigen analyzed a simple system of replicating molecules and demonstrated
the existence of a critical mutation rate, called
the error threshold \cite{EI1}.
This fundamental result led to the notion of
quasi\-spe\-cies developed by Eigen, McCaskill
and Schuster \cite{ECS1}.
If the mutation rate exceeds the error threshold,
then, at equilibrium, the population is completely random.
If the mutation rate is below the error threshold, 
then, at equilibrium, the population contains a positive fraction of the
Master sequence (the most fit macromolecule) and a cloud of mutants which
are quite close to the Master sequence.
This specific distribution of individuals is called a quasispecies.

Several researchers have already argued that the notion of error threshold 
plays a role
in the dynamics of a genetic algorithm.
This is far from obvious, because Eigen's model is formulated for an infinite
population model. However there is evidence that a similar phenomenon occurs in finite
populations as well, and also in genetic algorithms.
In her PhD thesis \cite{OCH}, Ochoa demonstrated the occurrence of error thresholds in genetic
algorithms over a wide range of problems and landscapes. This very interesting
work is published in a series of conference papers
\cite{OCH4,OCH5,OCH8,OCH6,OCH3,OCH2,OCH7}.
One of the most interesting and inspiring 
works on the theory of genetic algorithms
that
I have read over the last years is the series of papers by
van Nimwegen, Crutchfield and Mitchell
\cite{NP1,NP2,NP3,NCM1,NCM2}.
In these papers, the authors perform a theoretical and experimental 
study of a genetic algorithm on a specific
class of fitness functions. Their analysis 
relies on techniques from mathematical
population genetics, molecular evolution theory and statistical physics.
Among the fundamental ingredients
guiding the analysis are the quasi\-species model, the
error threshold and metastability.
In the last work of the series \cite{NP3},
van Nimwegen
and 
Crutchfield 
describe an entire search effort surface
and they introduce a generalized error threshold in the space of the
population size and the
mutation probability delimiting a set of parameters where
the genetic algorithm proceeds efficiently.
In a genetic algorithm, the crossover operator complicates the dynamics
and either it shifts the critical points or it creates new ones.
This phenomenon has been observed independently
by Rogers, Pr\"ugel--Bennett and Jennings
\cite{RPJ} and by Nilsson Jacobi and Nordahl \cite{NN}.

A version of the quasispecies theory was recently worked out
in the context of a classical model of population genetics, namely
the Wright--Fisher model \cite{CE4,DA1}. 
The Wright--Fisher model
corresponds exactly to a genetic algorithm without crossover.
Let us explain briefly why quasispecies theory is relevant to
understand the dynamics of a genetic algorithm.
Typically, on a complicated landscape,
the evolution of the genetic algorithm proceeds by jumps.
The population stays for a long time exploring the space around 
the current best fit chromosome, until it discovers a better chromosome.
If this discovery time is very long, the process reaches a local
equilibrium and the distribution of the population looks like
a quasispecies.
Our goal here is to estimate the persistence time of this quasispecies,
i.e., the time the quasispecies stays alive until it is destroyed
by a catastrophe.
The persistence time depends in a complicated way on the structure of the
fitness landscape. The persistence time depends also on the population size.
If the population size is large, the genetic algorithm
will be able to keep an interesting quasipecies
alive for
a long time, long enough until a new advantageous mutant is discovered
and creates a new quasispecies.
We shall obtain a simple lower bound on the persistence time of the fittest
chromosomes by comparing the genetic algorithm with a family of simpler 
processes, which are amenable to rigorous mathematical analysis.

A very interesting conclusion of \cite{NP3} 
is the existence of a critical population size
below which it is practically impossible to reach the global optimum.
A similar conclusion was obtained in the simpler framework of 
the generalized
simulated annealing \cite{CE0}: 
within a specific asymptotic regime of low mutations
and high selection pressure, the convergence to the global maximum could
be guaranteed only above a critical population size. 
The approach presented here confirms this prediction. The genetic algorithm
is very unlikely to reach the global optimum
if the population size is too small.
If the population size is too large, the genetic algorithm will evolve
slowly and will require too many evaluations of the fitness function.
The optimal population size seems to depend strongly on the optimization
problem. In any case
a population size of order $c^*\ell\ln\ell$ 
should be enough to search the space of chromosomes of length~$\ell$
(here again $c^*$ is a constant
depending on $\pi$ only).

Technically, we study the genetic algorithm within the framework
of the theory of Markov chains. It has been noted numerous times
in the literature that a genetic algorithm is conveniently modelled as
a Markov chain. Several researchers have studied genetic algorithms
in this context, here is a selection of works belonging to this line
of research: 
\cite{BGK,BE,EAH,GS,LO,LGA,MALAWE,MIWA,NAKA,NIVO,SL1,LS,XZZ,ZN,SU}. 
Unfortunately, the transition matrix is very complicated
and the resulting formulas are intractable.
Our strategy is to consider a specific asymptotic regime.
Twenty years ago, in \cite{CE0}, an asymptotic regime corresponding to the
simulated annealing was studied. In this regime, the space $\zul$
and the population size were kept fixed, while the mutation probability
was sent to~$0$ and the selection strength to~$\infty$. 
It was then possible to analyze precisely the asymptotic dynamics
of the population.
Several other interesting results have been obtained in this regime,
in particular, the understanding of the asymptotic dynamics helped to
design potentially more efficient variants of the genetic algorithm
\cite{RITR,OF1,OF2,DEMI2}. 
Although this regime made it possible
to derive rigorous convergence
results, it turned out to be
irrelevant in practice, because it is certainly not
the  correct regime of parameters to run efficiently a genetic algorithm.
Another interesting approach 
based on the Feynman--Kac formula
was developed 
by Del Moral and Miclo
\cite{DEMI0,DEMI1,DEMI4}. 
Several other works have considered other asymptotic approximations 
on specific fitness landscapes \cite{BEBI,BE,DO,MAPI}.
Here we consider the asymptotic regime corresponding to the quasispecies theory
in a finite population, namely:
\medskip

\noindent
$\bullet$ The size $m$ of the population goes to $\infty$.
\medskip

\noindent
$\bullet$ The length $\ell$ of the chromosomes is large.
\medskip

\noindent
$\bullet$ The mutation probability is of order $1/\ell$.

\bigskip
\noindent
We are able to derive various estimates in this specific asymptotic
regime. We hope that these results 
will be relevant in practice. 
The proofs use various tools from the theory of Markov chains:
coupling techniques, Galton--Watson processes, 
large deviations estimates, Poisson approximation.
The study of the auxiliary chain rests on several ideas which were
developed in
the Freidlin--Wentzell theory and in the analysis of
the simulated annealing.

The main results are stated in the section~\ref{mainre}.
The simple genetic algorithm is briefly explained in the next section
and it is formally described 
in section~\ref{secmod}.
The remaining sections are devoted to the proofs.
In section~\ref{seccoupl}, we build a coupling for the genetic algorithms
starting with different populations.
In section~\ref{secdis}, we develop stochastic bounds to study
the disordered regime.
In section~\ref{secqua}, we build an auxiliary chain and we study its
dynamics. This auxiliary chain keeps track of the evolution
of the quasispecies within the genetic algorithm.
Section~\ref{finproof} contains the final proofs of the main theorems.
Several classical results from probability theory are gathered
in the appendix.

\section{The simple genetic algorithm}
We provide here a brief description of the simple genetic algorithm.
A more formal definition is given in section~\ref{secmod}.
The goal of the simple genetic algorithm is to find the global maxima
of a fitness function~$f$ defined on $\zul$ with values
in $\mathbb R$. The genetic algorithm starts with a population
of $m$ points of $\zul$, called the chromosomes, and it repeats
the following fundamental cycle in order to build the generation
$n+1$ from the generation~$n$:
\bigskip

\noindent
{\bf Repeat}
\medskip

\noindent
\qquad
$\bullet$ Select two chromosomes from the generation~$n$
\medskip

\noindent
\qquad
$\bullet$ Perform the crossover
\medskip

\noindent
\qquad
$\bullet$ Perform the mutation
\medskip

\noindent
\qquad
$\bullet$ Put the two resulting chromosomes in generation~$n+1$
\medskip

\noindent
{{\bf Until} there are $m$ chromosomes in generation~$n+1$}
\bigskip

\noindent
When building the generation~$n+1$ from the generation~$n$, 
all the random choices are performed independently. Therefore,
the above algorithmic description is equivalent to the parallel
version described in section~\ref{secmod}.
Let us describe more precisely the selection, 
crossover and mutation steps.
\bigskip

\noindent
{\bf Selection.} We use ranking selection, meaning that
the chromosomes are ordered according to their fitness, and they
are selected with the help of a probability distribution
which depends only on their rank. 
In case there is a tie between several chromosomes, we rank
them randomly (with the uniform distribution over 
all possible choices).
We consider mainly two popular 
selection mechanisms: linear ranking selection and
tournament selection. 
The linear ranking selection depends
on two parameters $\eta^-,\eta^+$ satisfying
$0\leq\eta^-\leq\eta^+$, $\eta^-+\eta^+=2$ and we have
$$P\Big(
\begin{matrix}
\text{selecting the $(m-i+1)$--th}\\
\text{best chromosome}
\end{matrix}
\Big)
\,=\,\frac{1}{m}\Big(\eta^-+(\eta^+-\eta^-)\frac{i-1}{m-1}\Big)\,.$$
The tournament selection depends
on an integer parameter $t$ satisfying $2\leq t\leq m$ and we have
$$P\Big(
\begin{matrix}
\text{selecting the $(m-i+1)$--th}\\
\text{best chromosome}
\end{matrix}
\Big)
\,=\,
\frac{1}{m^t}
\big(i^t-(i-1)^t\big)\,.
$$
We introduce 
a parameter~$\sigma$, called the selection drift.
For the linear ranking selection,
the selection drift~$\sigma$ is equal to $\eta^+$.
For the tournament selection,
the selection drift~$\sigma$ is equal to $t$.
\bigskip

\noindent
{\bf Crossover.} 
We use the standard single point crossover and the crossover probability
is denoted by~$p_C$:
\medskip

\noindent
$$P\Bigg(
\lower 7pt\hbox{
\vbox{
\hbox{$000\raise 3pt\hbox{\vrule height 0.4pt width 25pt}011$}
\hbox{$100\raise 3pt\hbox{\vrule height 0.4pt width 25pt}110$}
}
\hbox{\kern-2pt\vrule width 0.4pt height 25pt depth 10pt\kern-2pt}
\vbox{
\hbox{$011\raise 3pt\hbox{\vrule height 0.4pt width 25pt}001$}
\hbox{$001\raise 3pt\hbox{\vrule height 0.4pt width 25pt}111$}
}
\raise 7pt\hbox{$\longrightarrow$}\!\!
\hbox{
\vbox{
\hbox{$000\raise 3pt\hbox{\vrule height 0.4pt width 25pt}011$}
\hbox{$100\raise 3pt\hbox{\vrule height 0.4pt width 25pt}110$}
}
\hbox{\kern-2pt\vrule width 0.4pt height 25pt depth 10pt\kern-2pt}
\vbox{
\hbox{$001\raise 3pt\hbox{\vrule height 0.4pt width 25pt}111$}
\hbox{$011\raise 3pt\hbox{\vrule height 0.4pt width 25pt}001$}
}}}\Bigg)\,=\,
\frac{p_C}{\ell-1}\,.$$
\smallskip

\noindent
{\bf Mutation.} 
We use independent parallel mutation at each bit
and the mutation probability
is denoted by~$p_M$.
$$P\Big(0000000\longrightarrow 0101000\Big)\,=\,p_M^2(1-p_M)^5\,.$$

\section{Main results}
\label{mainre}
In this section, we state our main results in the form of six
theorems.
The first two theorems show that the dynamics of the genetic
algorithm depends in a critical way on the value of 
$$\pi \,=\,\sigma(1-p_C)(1-p_M)^\ell\,.$$
If $\pi<1$, the most fit chromosome is very likely to disappear before
$\kappa\ln m$ generations.
If $\pi>1$, the most fit chromosome has a positive probability
(independent of~$m$) to invade a positive fraction of the population.
Heuristically, the parameter~$\pi$ can be interpreted
as the mean
number of offsprings of the fittest chromosome from one generation
to the next. 
\medskip

\noindent
{\bf The disordered regime.}
We consider the fitness function $f$ defined by
$$\forall u\in\zul\qquad
f(u)\,=\,
\begin{cases}
2
&\text{if }u=1\cdots 1\\
1
&\text{otherwise }\\
\end{cases}
$$
This corresponds to the sharp peak landscape.
The chromosome $1\cdots 1$ is called the Master sequence.
We start the genetic algorithm from the population
$$x_0\,=\,
\left(
\begin{matrix}
1&\cdots&1\\
0&\cdots&0\\
\vdots& & \vdots\\
0&\cdots&0\\
\end{matrix}
\right)
$$
\begin{theorem}\label{thspl}
Let $\pi<1$ be fixed.
We suppose that the parameters 
are set so that
$\ell=m$ and
$$\pi
\,=\,\sigma(1-p_C)(1-p_M)^\ell\,.$$
There exist strictly
positive constants $\kappa,\beta,m_0$, which depend on $\pi$
only, such that, for the 
genetic algorithm
starting from $x_0$,
$$\forall m\geq m_0\qquad
P\left(
\begin{matrix}
\text{the Master sequence $1\cdots 1$}\\
\text{disappears from the population}\\
\text{before $\kappa\ln m$ generations}\\
\end{matrix}
\right)
\,\geq\,1-\frac{1}{m^\beta}\,.
$$
\end{theorem}
\noindent {\bf The quasispecies regime.}
We consider an arbitrary fitness function $f$.
We start the genetic algorithm from an arbitrary population~$x_0$.
Let $\lambda_0$ be the maximal fitness in $x_0$, i.e.,
$$\lambda_0\,=\,
\max_{1\leq i\leq m}f\big(x_0(i)\big)\,.
$$
\begin{theorem}\label{thqsr}
Let $\pi>1$ be fixed.
We suppose that the parameters 
are set so that
$$\pi
\,=\,\sigma(1-p_C)(1-p_M)^\ell\,.$$
There exist strictly
positive constants $V^*,p^*$, which depend on $\pi$
only, such that, 
for the 
genetic algorithm
starting from $x_0$,
for any $\ell,m\geq 1$,
$$
P\left(
\begin{matrix}
\text{until the generation $\exp(V^*m)$}\\
\text{the population always contains a chromosome}\\
\text{with fitness larger than or equal to 
$\lambda_0$
}\\
\end{matrix}
\right)
\,\geq\,p^*\,.
$$
\end{theorem}
{\bf The catastrophes.}
We consider next an arbitrary fitness function $f$
and we start the genetic algorithm from an arbitrary position.
For $\lambda\in\mathbb R$ and a population~$x$, we define $N(x,\lambda)$ as the
number of chromosomes in $x$ whose fitness is larger than or equal to~$\lambda$:
$$N(x,\lambda)\,=\,\card\,\{\,i\in\um:f(x(i))\geq\lambda\,\}\,.$$
For $i\in\um$ and $x\in\zul$, we define
$\Lambda(x,i)$ as the fitness of the $i$--th best chromosome in 
the population~$x$:
$$\Lambda(x,i)\,=\,\max\,\big\{\,\lambda
\in{\mathbb R}:N(x,\lambda)\geq i\,\big\}
\,.$$
We denote by $X_n$ the population of the genetic algorithm
after $n$~iterations.
\begin{theorem}\label{thcsr}
Let $\pi>1$ be fixed.
There exist strictly
positive constants $\rho^*,c^*, m_0$, which depend on $\pi$
only, such that:
for any fitness function~$f$, any set of parameters
$\ell,p_C,p_M$ satisfying
$\pi
\,=\,\sigma(1-p_C)(1-p_M)^\ell$,
for any $m\geq m_0$,
for the 
genetic algorithm
starting from an arbitrary population,
we have
$$\displaylines{
P\left(
\forall n\leq\exp(c^*m)\quad
\max_{1\leq i\leq m}f\big(X_n(i)\big)
\,\geq\,
\max_{0\leq s\leq n}\Lambda\big(X_s,\lfloor\rho^*m\rfloor\big)
\right)
\hfill\cr
\hfill
\,\geq\,1-\exp(-c^*m)\,.
}$$
\end{theorem}
We point out that the sequence
$$
\max_{0\leq s\leq n}\Lambda\big(X_s,\lfloor\rho^*m\rfloor\big)
$$
is non--decreasing with respect to the time~$n$.
Thus, with very high probability, until time
$\exp(c^*m)$, the maximal fitness observed in the population
stays above a non--decreasing sequence.
We say that a catastrophe occurs at time~$n$ if
$$\max_{1\leq i\leq m}f\big(X_n(i)\big)
\,<\,
\max_{0\leq s\leq n}\Lambda\big(X_s,\lfloor\rho^*m\rfloor\big)\,.$$
This means that the maximal fitness in generation~$n$
has fallen below a fitness level which had been previously
reached by a fraction~$\rho^*$ of the chromosomes. In other words,
a quasispecies has been destroyed.

\noindent
{\bf Hitting time of optimal chromosomes.}
We denote by $H$ the Hamming distance between two chromosomes:
$$\forall u,v\in
\zul
\qquad
H(u,v)\,=\,\card\,\big\{\,j:
1\leq j\leq\ell,\,u(j)\neq v(j)\,\big\}\,.
$$
For $\lambda\in\mathbb{R}$, we define
$L(\lambda)$ as the set of the points in $\zul$ having a fitness
larger than or equal to~$\lambda$:
$$L(\lambda)\,=\,\big\{\,u\in\zul:f(u)\geq\lambda\,\big\}\,.$$
For $\lambda<\gamma$, we define
$\Delta(\lambda,\gamma)$ as the maximal distance between a point 
of $L(\lambda)$ and the set $L(\gamma)$, i.e.,
$$\Delta(\lambda,\gamma)\,=\,\max_{u\in L(\lambda)}
\,\min_{v\in L(\gamma)} H(u,v)\,.$$
Let $\tau^*$ be the 
hitting time of the set of the populations containing
optimal chromosomes, i.e.,
$$\displaylines{
\tau^*\,=\,\min\,\big\{\,
n\geq 1:\exists\, i\in\um\quad\hfill\cr
\hfill f(X_n(i))\,=\,
\max\,\big\{\,f(u):u\in\zul\,\big\}\,\big\}\,.}$$
We give next a theoretical upper 
bound on the expected value of~$\tau^*$.
\begin{theorem}\label{htopt}
Let $\pi>1$ be fixed.
There exist constants $c^*,\kappa^*,m_0$, which depend
only on $\pi$, such that:
for any set of parameters
$\ell,p_C,p_M,m$ satisfying
$$\pi
\,=\,\sigma(1-p_C)(1-p_M)^\ell\,,\quad
p_M\,\geq\,\frac{c^*}{\ell}\,,\quad
m\geq m_0\,,\quad
m\geq c^*\ell\ln\ell\,,
$$
for the genetic algorithm starting from an arbitrary 
population,
for any increasing sequence 
$\lambda_0<\cdots<\lambda_r$ 
such that
$$\lambda_0\,=\,\min\,\big\{\,f(u):u\in\zul\,\big\}\,,\quad
\lambda_r\,=\,\max\,\big\{\,f(u):u\in\zul\,\big\}\,,$$
we have
$$E(\tau^*)\,\leq\,
2+\kappa^* 
(\ln m)
m^2
\,
\sum_{k=0}^{r-1}
(p_M)^{-\Delta(\lambda_k,\lambda_{k+1})}\,.$$
\end{theorem}
In the next result, we make a strong structural hypothesis on the fitness
landscape and 
we obtain a bound on $\tau^*$ 
which is polynomial in~$\ell$.
\begin{theorem}\label{ctopt}
Let $\gamma,\Delta\geq 1$.
Suppose that the fitness function is such that
there exists
an increasing sequence 
$$\lambda_0=\min_{\zul}f\,<\,\lambda_1\,<\,\cdots\,
<\,\lambda_{r-1}\,<\,
\lambda_r=\max_{\zul}f$$
with 
$r\leq\ell^\gamma$ and satisfying
$$\forall k\in\{\,0,\dots,r-1\,\}\qquad
\Delta(\lambda_k,\lambda_{k+1})\leq\Delta\,.$$
Let $\pi>1$ be fixed.
There exist positive constants $c^*,m_0$, which depend
only on $\pi$, such that:
for any set of parameters
$\ell,p_C,p_M,m$ satisfying
$$\pi
\,=\,\sigma(1-p_C)(1-p_M)^\ell\,,\quad
p_M\,\geq\,\frac{c^*}{\ell}\,,\quad
m\geq m_0\,,\quad
m\geq c^*\Delta\ln\ell\,,
$$
for the genetic algorithm starting from an arbitrary 
population,
$$E(\tau^*)\,\leq\,
2+(\ln m)\,
m^2\ell^{\gamma+\Delta}
\,.$$
\end{theorem}

\noindent
{\bf Asymptotic convergence.}
The bounds on the hitting time of optimal chromosomes yield
simple estimates for the invariant probability measure of
the genetic algorithm. Let us recall that the invariant probability
measure~$\mu$ is given by
$$\forall x,y\in\zulm\qquad
\mu(y)\,=\,\lim_{n\to\infty}
P\big(X_n=y\,\big|\,X_0=x\big)\,.$$
The invariant probability measure~$\mu$ depends on 
the fitness function~$f$ and the parameters $\ell,p_C,p_H,m$,
as well as the selection mechanism.
\begin{theorem}\label{invpt}
Let $\pi>1$ be fixed.
There exist positive constants $c^*,m_0$, which depend
only on $\pi$, such that:
for any set of parameters
$\ell,p_C,p_M,m$ satisfying
$$\pi
\,=\,\sigma(1-p_C)(1-p_M)^\ell\,,\quad
p_M\,\geq\,\frac{c^*}{\ell}\,,\quad
m\geq m_0\,,\quad
m\geq c^*\ell\ln\ell\,,
$$
for any fitness function
$f:\zul\to{\mathbb R}$,
the invariant probability measure~$\m$ of the simple genetic algorithm
satisfies
$$\mu\big(\big\{\,
x:
\max_{1\leq i\leq m}f\big(x(i)\big)
\,=\,
\max_{\zul}f
\,\big\}\big)\,\geq\,1-\exp(-c^*m)\,.
$$
\end{theorem}
This result is certainly less relevant in practice than the previous
ones, since it is extremely difficult to understand the speed of convergence
of the genetic algorithm towards its invariant measure, yet
it demonstrates that the
genetic algorithm is successful in this asymptotic regime.

\section{The model}\label{secmod}
Let $\ell\geq 1$ be an integer. We work on the space
$\zul$
of binary strings of length~$\ell$.
An element of
$\zul$ is called a chromosome.
Generic elements of $\zul$
will be denoted by the letters $u,v,w$.
Let $m\geq 1$ be an even integer.
A population is an $m$--tuple of elements of
$\zul$.
Generic populations will be denoted by the letters 
$x,y,z\index{$x,y,z$}$.
Thus a population $x$ is a vector
$$x\,=\,
\left(
\begin{matrix}
x(1)\\
\vdots\\
x(m)
\end{matrix}
\right)
$$
whose components are chromosomes.
For $i\in\{\,1,\dots,m\,\}$, we denote by
$$x(i,1),\dots,x(i,\ell)$$
the digits of the sequence $x(i)$. This way a population $x$
can be represented as an array
$$x\,=\,
\left(
\begin{matrix}
x(1,1)&\cdots&x(1,\ell)\\
\vdots& & \vdots\\
x(m,1)&\cdots&x(m,\ell)\\
\end{matrix}
\right)
$$
of size $m\times\ell$ of zeroes and ones, the
$i$--th row corresponding to the $i$--th chromosome
of the population.
Let
$f:\zul
\to {\mathbb R}$
be an arbitrary objective function, traditionally called the
fitness function.

Mathematically, a simple genetic algorithm is conveniently modelled
by a Markov chain $(X_n)_{n\in\mathbb N}$ with state space $\zulm$,
the space of the populations of $m$
chromosomes.
The transition mechanism of the simple genetic algorithm can be decomposed
into three steps: selection, crossover and mutation.
We explain separately each step.

\subsection{Selection}
We perform first the selection operation, which consists in
selecting with replacement $m$ chromosomes from the population.
To this end,
we build a selection distribution
$$\sel:
\zulm\times\um\to [0,1]\,.$$
The value $\sel(x,i)$ is the probability of selecting the $i$--th chromosome
in the population~$x$. We consider only ranking selection mechanisms, hence
the value 
$\sel(x,i)$ depends only on the ranking of the chromosomes of the population~$x$
according to their fitness.
We first define a ranking function, which gives the rank of a chromosome in a population.
Let $x\in\zulm$
be a population.
We choose a
permutation 
$\sigma$ 
of $\um$ such that
$$f(x(\sigma(1)))\,\leq\,\cdots
\,\leq\,
f(x(\sigma(m)))\,.$$
The choice of $\sigma$ is not unique in case of ties,
when several chromosomes have the same fitness.
It turns out that the way the permutation $\sigma$ is chosen affects considerably the
behavior of the genetic algorithm.
We choose the permutation $\sigma$ randomly, according to the uniform distribution
over the set of the permutations $\sigma$ which satisfy the above condition.
The choice of $\sigma$ is done 
independently from the other steps of the algorithm, and 
it is performed again at each selection step. 
In particular, each time the process returns to the population $x$, a new permutation
$\sigma$ is drawn independently, and the ordering of the chromosomes will change
accordingly.
We define the rank of the $i$--th chromosome of the population $x$ as
$$\rank(x,i)\,=\,\sigma^{-1}(i)\,.$$
Thus if 
$\rank(x,i)=m$, this implies that $x(i)$ has the largest
 fitness in the population $x$, but the converse is false: when
the fitness function $f$ is not one to one, a chromosome with maximal fitness
might get a ranking smaller than $m$.
Once the ranking function
$\rank(x,i)$ is built, we need a selection distribution $F_m$ on $\um$ to complete the 
definition 
of the selection distribution $\sel(x,i)$.
A selection distribution $F_m$ on $\um$ is simply a probability distribution on $\um$.
We define the selection distribution 
$\sel(x,i)$ by setting
$$\forall i\in\um\qquad
\sel(x,i)\,=\, F_m(\rank(x,i))\,.$$
Throughout the paper, we shall make the following 
hypothesis on~$F_m$.
\smallskip

\noindent
{\bf Hypothesis on $F_m$}.
\label{genhy}
There exists a repartition function $F$ on $[0,1]$ such that
$$\forall s\in[0,1]\qquad
\lim_{m\to\infty}
\sum_{i\leq sm}F_m(i)\,=\,F(s)\,.$$
We suppose that $F$ is continuous on $[0,1]$,
strictly increasing on $[0,1]$, convex on $]0,1[$ 
and left differentiable at~$1$.
The value of its left derivative at $1$ is called the selection drift
and is denoted by $\sigma$ (necessarily $\sigma\geq 1$).
We suppose that the discrete derivative of $F_m$ at $1$
converges to $\sigma$ in the following sense:
$$\displaylines{
\forall \varepsilon>0\quad\exists\,\delta>0\quad
\exists m_0\geq 1\quad
\forall m\geq m_0\quad
\forall i\in\big\{\,1,\dots,\lfloor\delta m\rfloor\,\big\}
\hfill\cr
\Big|F_m(m-i+1)+\cdots+F_m(m)
-\sigma\frac{i}{m}
\Big|
\,\leq\,\varepsilon\sigma\frac{i}{m}\,.}$$
We consider two popular choices for the selection distribution $F_m$.
\smallskip

\noindent
{\bf Linear ranking selection.}
This selection scheme depends on two parameters $\eta^-,\eta^+$ 
which satisfy
$$0\,\leq\,\eta^-\,\leq\,\eta^+\,,\qquad
\eta^-+\eta^+\,\,=2\,.$$
We define the linear ranking selection distribution by
$$\forall i\in\um\qquad
F_m(i)\,=\,\frac{1}{m}\Big(\eta^-+(\eta^+-\eta^-)\frac{i-1}{m-1}\Big)\,.$$
The linear ranking selection distribution satisfies the hypothesis with
$$F(s)\,=\,\eta^-s+\frac{1}{2}(\eta^+-\eta^-)s^2\,,
\qquad \sigma=\eta^+\,.$$
{\bf Tournament selection.}
This selection scheme depends on an integer parameter $t$ satisfying $2\leq t\leq m$.
We define the tournament selection distribution by
$$\forall i\in\um\qquad
F_m(i)\,=\,
\frac{1}{m^t}
\big(i^t-(i-1)^t\big)\,.
$$
The tournament selection distribution satisfies the hypothesis with
$$F(s)\,=\,s^t\,,
\qquad \sigma=t\,.$$
Finally,
we draw independently 
$m$ chromosomes from the population $x$ according to the 
selection distribution $\sel(x,\cdot)$ to obtain the population
after selection.
The stochastic matrix $P_S$ associated to the selection operator is defined
as follows. The probability to select the population $y$ starting from the
population $x$ is 
$$P_S(x,y)\,=\,
\prod_{i=1}^m
\Big(
\sum_{j:x(j)=y(i)}\sel(x,j)\Big)\,.$$

\subsection{Crossover}
After having selected $m$ chromosomes, we perform the crossover operation.
The crossover depends on a parameter $p_C\in[0,1]$ and it acts on pairs of chromosomes.
Let us explain how the crossover operator acts on two chromosomes $u,v$.
With probability $1-p_C$, there is no crossover and the chromosomes $u,v$
are not modified.
With probability $p_C$, there is a crossover between the chromosomes $u,v$.
We choose uniformly at random a cutting position~$k$ in $\{\,1,\dots,\ell-1\,\}$.
A new pair $(u',v')$ of chromosomes is formed, where $u'$ (respectively $v'$)
consists of the first $k$ digits of $u$ (respectively $v$)
and the last $\ell-k$ digits of $v$ (respectively $u$).
\smallskip

\noindent
\centerline{
\hbox{
\vbox{
\hbox{$u=000\raise 3pt\hbox{\vrule height 0.4pt width 25pt}011$}
\hbox{$v=100\raise 3pt\hbox{\vrule height 0.4pt width 25pt}110$}
}
\hbox{\kern-2pt\vrule width 0.4pt height 25pt depth 10pt\kern-2pt}
\vbox{
\hbox{$011\raise 3pt\hbox{\vrule height 0.4pt width 25pt}001$}
\hbox{$001\raise 3pt\hbox{\vrule height 0.4pt width 25pt}111$}
}
\raise 7pt\hbox{$\quad\longrightarrow\quad$}
\hbox{
\vbox{
\hbox{$u'=000\raise 3pt\hbox{\vrule height 0.4pt width 25pt}011$}
\hbox{$v'=100\raise 3pt\hbox{\vrule height 0.4pt width 25pt}110$}
}
\hbox{\kern-2pt\vrule width 0.4pt height 25pt depth 10pt\kern-2pt}
\vbox{
\hbox{$001\raise 3pt\hbox{\vrule height 0.4pt width 25pt}111$}
\hbox{$011\raise 3pt\hbox{\vrule height 0.4pt width 25pt}001$}
}}}}
\hbox{\hskip 52pt cutting position}\newline
\smallskip

\noindent
Mathematically, this mechanism is encoded in a crossover kernel
$$C:
\big(\zul\big)^2
\times
\big(\zul\big)^2
\to [0,1]\,.$$
The value $C\big((u,v),(u',v')\big)$ is the probability that,
by crossover, the pair of chromosomes $(u,v)$ becomes the pair
$(u',v')$. More precisely, 
let
$\varepsilon_1,\dots,\varepsilon_\ell$,
$\eta_1,\dots,\eta_\ell$ belong to $\{\,0,1\,\}$
and let
$k\in
\{\,1,\dots,\ell-1\,\}$. We define
$$\switch\big(k,
{\varepsilon_1\cdots\varepsilon_\ell},
{\eta_1\cdots\eta_\ell}
\big)=
{\varepsilon_1\cdots\varepsilon_k
\eta_{k+1}\cdots\eta_\ell}
\,.$$
The crossover kernel $C$ is then equal to
$$\displaylines{
C\Big(\Big(
\genfrac{}{}{0pt}{0}{u}{v}\Big),
\Big(\genfrac{}{}{0pt}{0}{u'}{v'}
\Big)
\Big)
\,=\,
(1-p_C)1_{(u,v)=(u',v')}
\hfill\cr
+\frac{p_C}{\ell-1}\card\Bigg\{\,k\in
\{\,1,\dots,\ell-1\,\}:
\genfrac{}{}{0pt}{0}{
\switch\big(k,u,v\big)=u'
}{
\switch\big(k,v,u\big)=v'
}
\,\Bigg\}\,.}$$
We apply simultaneously the crossover operator on the $m/2$ consecutive pairs
of chromosomes of a population of size $m$.
The stochastic matrix $P_C$ associated to the crossover operator is defined
as follows.
The probability to obtain the population $z$ after performing the
crossover starting from the population $y$ is
$$P_C(y,z)\,=\,
\prod_{i=1}^{m/2}
C\Big(\Big(
\genfrac{}{}{0pt}{0}{y(2i-1)}{y(2i)}\Big),
\Big(\genfrac{}{}{0pt}{0}{z(2i-1)}{z(2i)}
\Big)\Big)
\,.$$
\subsection{Mutation}
After having performed the crossover, we perform the mutation. 
The mutation depends
on one parameter, the mutation probability $p_M\in[0,1]$, and it
acts on a single chromosome.
Let $u=
\varepsilon_1\cdots\varepsilon_\ell$
be a chromosome.
For each $k\in\ul$, the digit $\varepsilon_k$ is kept unchanged with
probability $1-p_M$ and it mutates to $1-\varepsilon_k$
with probability $p_M$. These changes are done simultaneously and independently.
Mathematically, this mechanism is encoded in a mutation kernel
$$M:
\big(\zul\big)^2
\to [0,1]\,.$$
The value $M(u,v)$ is the probability that, by mutation, the chromosome $u$
becomes the chromosome $v$, and it is given by
$$M(u,v)\,=\, p_M^{H(u,v)}(1-p_M)^{\ell-H(u,v)}\,,$$
where $H(u,v)$ is the Hamming distance between $u$ and $v$, defined by
$$H(u,v)\,=\,\card\,\big\{\,j:
1\leq j\leq\ell,\,u(j)\neq v(j)\,\big\}\,.$$
The stochastic matrix $P_M$ associated to the mutation operator is defined
as follows.
The probability to obtain the population $x'$ after performing the
mutation starting from the population $z$ is
$$P_M(z,x')\,=\,
\prod_{i=1}^m
M(z_i,x'_i)\,.$$
\subsection{Transition matrix of the SGA}
The fundamental cycle of the simple genetic algorithm consists in 
applying successively the
selection, the crossover and the mutation operators on the population.
Mathematically, the simple genetic algorithm is conveniently modelled
by a Markov chain $(X_n)_{n\in\mathbb N}$ with state space $\zulm$,
the space of the populations of $m$
chromosomes.
The transition matrix $P_{SGA}$ of the simple genetic algorithm is defined by
$$\forall x,x'\in\zulm\qquad
P_{SGA}(x,x')\,=\,P\big(X_{n+1}=x'\,|\,X_n=x\,\big)\,.$$
The matrix $P_{SGA}$ is
simply the product 
of the three matrices $P_S,P_C,P_M$,
i.e.,
$P_{SGA}\,=\,P_SP_CP_M$, or
equivalently,
$$\forall x,x'\in\zulm\qquad
P_{SGA}(x,x')\,=\,\sum_{y,z}\,
P_S(x,y)P_C(y,z)P_M(z,x')\,.$$
\section{Coupling for the genetic algorithm}\label{seccoupl}
Throughout the proofs, we rely on various coupling arguments.
We will couple here the simple genetic algorithm starting from any possible
initial population.
We first define separately the maps for coupling the
selection, crossover and mutation.
\smallskip

\noindent
{\bf Selection.}
We define a selection
map
$$\cS:\zulm\times [0,1]\to\um
\index{$\cS_H$} $$
in order to
couple the selection mechanism starting with different populations.
We first build a map
$\cI:[0,1]\to\um$ which gives the rank of the chromosome to choose.
More precisely, for
$s\in[0,1[$, we set
$\cI(s)=i$ where $i$ is the unique index in $\um$ satisfying
$$
F_m(1)+\cdots
+F_m(i-1)
\,\leq\,s\,<\,
F_m(1)+\cdots
+F_m(i)
\,.$$
Let next $x\in\zulm$ and let $s\in[0,1[$. We define 
$\cS(x,s)=j$ where $j$ is the unique index in $\um$ such that
$\rank(x,j)=\cI(s)$.
The map $\cS$ is built in such a way that,
if $U$ is
a random variable 
with uniform law on 
the interval 
$[0,1]$, 
then,
for any population~$x$,
the law of 
$\cS(x,U)$
is given by 
the selection distribution:
$$\forall i\in\um\qquad
P\big(
\cS(x,
U)=i\big)\,=\,
\sel(x,i)\,.$$
%

\par\noindent
{\bf Crossover.}
We define a map 
$$\cC:\zulc\times \zu\times \ulmo\to\zul
\index{$\cMH$}$$
in order to
couple the crossover mechanism starting with different 
pairs of chromosomes. 
Let $u,v\in\zul$
and let $\varepsilon\in\zu$, $k\in\ulmo$. We define
$$\cC \big( u,v,\varepsilon,k \big)
\,=\,
\begin{cases}
\switch\big(k,
{u},{v}
\big)& \text{ if }\varepsilon=1\\
u
& \text{ if }\varepsilon=0
\end{cases}
$$
%
The map $\cC$ is built in such a way that,
if $V,W$ are two independent random variables
with respective laws
the
Bernoulli law with parameter $p_C$
and the
uniform law on $\ulmo$, 
then,
for any chromosomes $u,v$,
the law of  the pair of chromosomes
$\cC(u,v,V,W)$,
$\cC(v,u,V,W)$
is given by 
the crossover kernel $C$: 
$$\displaylines{
\forall u',v'\in\zul\qquad
P
\left(
\begin{matrix}
\cC(u,v,V,W)=u'\\
\cC(v,u,V,W)=v'\\
\end{matrix}
\right)
\,=\,
C\Big(\Big(
\genfrac{}{}{0pt}{0}{u}{v}\Big),
\Big(\genfrac{}{}{0pt}{0}{u'}{v'}
\Big)
\Big)
\,.
}$$

\par\noindent
{\bf Mutation.}
We define a map 
$$\cM:\zul\times \zul\to\zul
\index{$\cMH$}$$
in order to
couple the mutation mechanism starting with different 
chromosomes. 
Let $\varepsilon_1,\dots,\varepsilon_\ell\in\zu$
and let $u_1,\dots,u_\ell\in\zu$.
%
The map $\cM$ is defined by setting
$$\cM(
\varepsilon_1\cdots\varepsilon_\ell,
u_1,\cdots, u_\ell)=
\eta_1\cdots \eta_\ell
\,,$$
where
$$\forall i\in\ul\qquad
\eta_i\,=\,
\begin{cases}
\varepsilon_i &\text{if }u_i=0\\
1-\varepsilon_i &\text{if }u_i=1\\
\end{cases}
$$
%
The map $\cM$ is built in such a way that,
if $U^1,\dots,U^\ell$ are $\ell$ independent random variables
with law
the
Bernoulli law with parameter $p_M$,
then,
for any chromosome $u$,
the law of  the chromosome
$\cM(u,U^1,\dots,U^\ell)$
is given by the line
of the mutation matrix corresponding to~$u$:
$$
\forall v\in\zul\qquad
P\big(\cM(u,U^1,\dots,U^\ell)=v)\big)\,=\,M(u,v)\,.$$
{\bf Coupling for the genetic algorithm.}
We will now combine the maps $\cS,\cC,\cM$ with random inputs in 
order to couple the genetic algorithm with various initial conditions.
We will build
all the processes on a single large probability space.
We consider a probability space $(\Omega,{\mathcal F}, P)$
containing the following collection of independent random variables:
%
%
%
%
%
$$
\begin{array}{llll}
\text{$\bullet$ Uniform on the interval $[0,1]$:}&
S_{n}^{i}\,,&n\geq 1\,,&
1\leq i\leq m\,;\\
\text{$\bullet$ Bernoulli with parameter $p_M$:}&
U_{n}^{i,j}\,,&n\geq 1\,,& 1\leq i\leq m\,,1\leq j\leq\ell\,;\!\\
\text{$\bullet$ Bernoulli with parameter $p_C$:}&
V_{n}^{i}\,,& n\geq 1\,,&
1\leq i\leq m/2\,;\\
\text{$\bullet$ Uniform on $\ulmo$:}&
W_{n}^{i}\,,& n\geq 1\,,&
1\leq i\leq m/2\,.
\end{array}
$$
The variables having subscript~$n$ constitute 
the random input which is used to perform the
$n$--th step of the Markov chains.
For each $n\geq 1$, 
we build a map
$$\Phi_n:
\zulm
\to\zulm
\index{$\Psi_H$}
$$
in order to realize the coupling between the genetic algorithm
with various initial conditions. 
The coupling map $\Phi_n$ is defined by
$$\displaylines{
\forall x\in\zulm\qquad
\Phi_n(x)\,=\,
\hfill\cr
\left(
\begin{matrix}
\cM\big(
\cC\big(
\cS(x,S_n^1), 
\cS(x,S_n^2), 
V_n^1, W_n^1\big),
U_{n}^{1,1},\dots,U_{n}^{1,\ell}\big)
\\
\cM\big(
\cC\big(
\cS(x,S_n^2), 
\cS(x,S_n^1), 
V_n^1, W_n^1\big),
U_{n}^{2,1},\dots,U_{n}^{2,\ell}\big)
\\
\cM\big(
\cC\big(
\cS(x,S_n^3), 
\cS(x,S_n^4), 
V_n^2, W_n^2\big),
U_{n}^{3,1},\dots,U_{n}^{3,\ell}\big)
\\
\cM\big(
\cC\big(
\cS(x,S_n^4), 
\cS(x,S_n^3), 
V_n^2, W_n^2\big),
U_{n}^{4,1},\dots,U_{n}^{4,\ell}\big)
\\
\qquad\vdots\qquad\\
\cM\big(
\cC\big(
\cS(x,S_n^{m-1}), 
\cS(x,S_n^m), 
V_n^{m/2}, W_n^{m/2}\big),
U_{n}^{m-1,1},\dots,U_{n}^{m-1,\ell}\big)
\\
\cM\big(
\cC\big(
\cS(x,S_n^m), 
\cS(x,S_n^{m-1}), 
V_n^{m/2}, W_n^{m/2}\big),
U_{n}^{m,1},\dots,U_{n}^{m,\ell}\big)
\\
\end{matrix}
\right)
}$$
The coupling is then built in a standard way with the help of the 
sequence
$(\Phi_n)_{n\geq 1}$.
Let $x\in\zulm$ be the starting point of the process.
We build the process 
$(X_n)_{n\geq 0} 
\index{$D_n$}$
by setting
$X_0=x$ and
$$\forall n\geq 1\qquad
X_n\,=\,\Phi_n\big(X_{n-1}\big)\,.
$$
A routine check shows that the process
$(X_n)_{n\geq 0}$ 
is a Markov chain starting from $x$
with the adequate transition matrix. This way we have coupled
the genetic algorithm with all possible initial conditions.
\section{The disordered regime}\label{secdis}
We consider the fitness function $f$ defined by
$$\forall u\in\zul\qquad
f(u)\,=\,
\begin{cases}
2
&\text{if }u=1\cdots 1\\
1
&\text{otherwise }\\
\end{cases}
$$
This corresponds to the sharp peak landscape.
The chromosome $1\cdots 1$ is called the Master sequence.
We start the genetic algorithm with the population
$$x_0\,=\,
\left(
\begin{matrix}
1&\cdots&1\\
0&\cdots&0\\
\vdots& & \vdots\\
0&\cdots&0\\
\end{matrix}
\right)
$$
and we wish to estimate the probability of survival of the
Master sequence.
We denote by $(X_n)_{n\in\mathbb N}$ the genetic algorithm
starting from $x_0$.
We shall develop bounds in the sense of stochastic domination.
Let $\pi<1$ be fixed. Throughout the section, we suppose that
$\ell,p_C,p_M$ satisfy
$$\sigma(1-p_C)(1-p_M)^\ell\,=\,\pi\,.$$
\subsection{Genealogy}
To build a chromosome in the generation~$n$, we select two parents
in generation~$n-1$, and we apply the crossover and the
mutation operators.
Thus each chromosome has two parents in the previous generation.
With the coupling construction, the parents of the chromosomes
$X_n(2i-1),X_n(2i)$ are the chromosomes
$\cS(X_{n-1},S_n^{2i-1})$,
$\cS(X_{n-1},S_n^{2i})$.
The genealogy of a chromosome consists of all its ancestors until
time~$0$.
We define auxiliary random variables in order to control 
the progeny of the initial Master sequence.
For $n\geq 1$, $i\in\um$, we set
$M_n(i)=1$ if the Master sequence appears in the genealogy of $X_n(i)$
and $0$ otherwise. We denote by $M_n$ the vector 
$(M_n(1),\dots,M_n(m))$ and
we define also
$$
T_n\,=\,\sum_{i=1}^mM_n(i)\,.$$
The variable $T_n$ is the total number of descendants of the Master
sequence at time $n$.
Let also $N^*_n$ be the number of Master sequences present in the 
population at time~$n$:
$$\forall n\geq 1\qquad
N^*_n\,=\,\card\,\big\{\,i\in\um:X_n(i)=1\cdots 1\,\big\}\,.$$
We shall next compute stochastic bounds on 
$(T_n)_{n\in\mathbb N}$ and
$(N^*_n)_{n\in\mathbb N}$.
The process
$(T_n)_{n\in\mathbb N}$ will be controlled by a supercritical
branching process, while
the process
$(N^*_n)_{n\in\mathbb N}$ will be controlled by a subcritical
branching process.
\subsection{Bound on $T_n$}
By construction, the chromosomes $X_n(2i-1)$ and $X_n(2i)$ have the
same parents, thus $M_n(2i-1)=M_n(2i)$ and
$$\forall n\geq 1\qquad
T_n\,=\,\sum_{i=1}^{m/2} 2M_n(2i)
\,.$$
Conditionally on $X_{n-1},M_{n-1}$, the random variables
$M_n(2i)$, $1\leq i\leq m/2$, are independent and identically
distributed. Let us estimate their parameter:
$$\displaylines{
P\big(M_n(2)=0\,\big|\,X_{n-1},M_{n-1}\,\big)
\,=\,\hfill\cr
P\left(
\begin{matrix}
\text{the selection operator selects
two parents}\\
\text{in $X_{n-1}$ which do not belong to}\\
\text{the progeny of the initial
Master sequence}\\
\end{matrix}
\,\Bigg|\,
X_{n-1},M_{n-1}\,
\right)\,.
}
$$
The number of chromosomes in the progeny of the initial
Master sequence at time $n-1$ is $T_{n-1}$. The lowest value 
for the above conditional probability corresponds to the
situation where all these chromosomes are ranked best during
the selection process, therefore
$$P\big(M_n(2)=0\,\big|\,X_{n-1},M_{n-1}\,\big)
\,\geq\,\Big(F_m(1)+\cdots+F_m(m-T_{n-1})\Big)^2\,.$$
To go further, we need a bound on~$T_{n-1}$.
Thus we define
$$\tau_1\,=\,\inf\,\big\{\,n\geq 1: T_n>m^{1/4}\,\big\}\,$$
and we will study the random variable
$T_n1_{\{\,\tau_1\geq n\,\}}$.
In order to incorporate the event $\{\,\tau_1\geq n\,\}$ in
the previous 
inequality, we condition with respect to the whole history
of the process as follows:
$$\displaylines{
P\big(M_n(2)=1,\,
\tau_1\geq n\,\big|\,X_{n-1},M_{n-1},\dots,X_0,M_0\,\big)\hfill\cr
\,=\,
1_{\{\,T_0\leq m^{1/4},\dots
,T_{n-1}\leq m^{1/4}\,\}}
P\big(M_n(2)=1
\,\big|\,X_{n-1},M_{n-1},\dots,X_0,M_0\,\big)
\cr
\,=\,
1_{\{\,T_0\leq m^{1/4},\dots
,T_{n-1}\leq m^{1/4}\,\}}
P\big(M_n(2)=1
\,\big|\,X_{n-1},M_{n-1}\,\big)
\cr
\,\leq\,
1_{\{\,\tau_1\geq n-1\,\}}
\Bigg(
1-
\Big(F_m(1)+\cdots+F_m(m-T_{n-1})\Big)^2
\Bigg)\,.
}$$
Using the hypothesis 
on $F_m$, we obtain that, for $m$
large enough, there exists $\delta>0$ such that
$$\forall i\in\big\{\,1,\dots,\lfloor\delta m\rfloor\,\big\}
\qquad
F_m(m-i+1)+\cdots+F_m(m)
\,\leq\,2\sigma\frac{i}{m}\,.$$
For $m$ large enough, if 
$\tau_1\geq n$, then 
$T_{n-1}\leq m^{1/4}\leq\lfloor\delta m\rfloor$, whence
\begin{multline*}
P\big(M_n(2)=1,\,
\tau_1\geq n\,\big|\,X_{n-1},M_{n-1},\dots,X_0,M_0\,\big)
\cr
\,\leq\,
1_{\{\,\tau_1\geq n-1\,\}}
\Bigg(
1-\Big(1-2\sigma\frac{T_{n-1}}{m}\Big)^2\Bigg)
\,\leq\,
1_{\{\,\tau_1\geq n-1\,\}}
4\sigma\frac{T_{n-1}}{m}\,.
\end{multline*}
\begin{proposition}\label{bgwa}
Let 
$(Z_n)_{n\in\mathbb N}$
be a Galton--Watson process starting from $Z_0=1$
with reproduction law
$\nu=2\cP(4\sigma)$, i.e., the law~$\nu$ is twice the
Poisson law of parameter~$4\sigma$. We have
$$\forall n\geq 0\qquad
T_n1_{\{\,\tau_1\geq n\,\}}\,\preceq\,Z_n\,.$$
\end{proposition}
\begin{proof}
We recall that $\preceq$ means stochastic domination 
(see appendix~\ref{apstoc}).
We will prove the inequality by induction on $n$. 
For $n=0$, the inequality holds trivially, since
$$T_01_{\{\,\tau_1\geq 0\,\}}\,=\,Z_0\,=\,1\,.$$
Let $n\geq 1$ and suppose 
that the inequality has been proved at rank $n-1$.
The previous computation shows that,
conditionally on
$X_0$,$M_0,\dots,$ $X_{n-1}$,$M_{n-1}$,
the law of 
$T_n1_{\{\,\tau_1\geq n\,\}}$ is stochastically dominated by
the law
$$2\cB\Big(\frac{m}{2},
\frac{4\sigma}{m}
T_{n-1}1_{\{\,\tau_1\geq n-1\,\}}\Big)\,.$$
There exists $t_0>0$ such that, for $0<t<t_0$, we have
$\ln(1-t)\geq -2t$. Therefore, for $m$ large enough so that 
$4\sigma m^{-3/4}<t_0$, we have
$$
\Big(1-
\frac{4\sigma}{m}
T_{n-1}1_{\{\,\tau_1\geq n-1\,\}}\Big)^{m/2}\,\geq\,
\exp\Big(
-
{4\sigma}
T_{n-1}1_{\{\,\tau_1\geq n-1\,\}}\Big)
\,.$$
By lemma~\ref{binopoi}, we conclude from this inequality that
$$2\cB\Big(\frac{m}{2},
\frac{4\sigma}{m}
T_{n-1}1_{\{\,\tau_1\geq n-1\,\}}\Big)\,\preceq\,
2\cP\big(4\sigma
T_{n-1}1_{\{\,\tau_1\geq n-1\,\}}\big)
\,.$$
Therefore, for any non--decreasing function~$\phi$, 
we have
$$\displaylines{
E\Big(\phi\big(
T_{n}1_{\{\,\tau_1\geq n\,\}}\big)
\,\big|\,X_{n-1},M_{n-1},\dots,X_0,M_0\,\Big)
\,\leq\,\hfill\cr
E\Big(\phi\Big(
\sum_{k=1}^{
T_{n-1}1_{\{\,\tau_1\geq n-1\,\} }
}
Y_k
\Big)
\,\big|\,
T_{n-1}1_{\{\,\tau_1\geq n-1\,\} }
\,\Big)
}$$
where the random variables
$(Y_k)_{k\geq 1}$ are independent identically distributed
with 
law
$\nu=2\cP\big(4\sigma\big)$, and they are 
also independent of the processes
$(X_n)_{n\in\mathbb N}$, 
$(M_n)_{n\in\mathbb N}$.
The conditional expectation
is a non--decreasing function of 
$T_{n-1}1_{\{\,\tau_1\geq n-1\,\}}$.
Taking the expectation with respect to
$$X_{n-1},M_{n-1},\dots,X_0,M_0$$ 
and using the induction
hypothesis, we get
$$\displaylines{
E\Big(\phi\big(
T_{n}1_{\{\,\tau_1\geq n\,\}}\big)
\Big)
\,\leq\,
E\Big(\phi\Big(
\sum_{k=1}^{
T_{n-1}1_{\{\,\tau_1\geq n-1\,\} }}
Y_k
\Big)
\Big)
\hfill\cr
\,\leq\,
E\Big(\phi\Big(
\sum_{k=1}^{Z_{n-1}}
Y_k
\Big)
\Big)
\,=\,
E\Big(\phi\big(Z_n\big)\Big)\,.
}$$
This completes the induction step.
\end{proof}

\noindent
In order to exploit proposition~\ref{bgwa}, we shall need a bound 
on~$\tau_1$, which we compute next.
\begin{proposition}\label{boundtau}
There exist $\kappa>0$, $c_1>0$, $m_1\geq 1$, such that
$$\forall m\geq m_1\qquad
P\big(\tau_1<\kappa\ln m\big)\,\leq\,\frac{1}{m^{c_1}}\,.$$
\end{proposition}
\begin{proof}
Let 
$(Z_n)_{n\in\mathbb N}$
be a Galton--Watson process as in
proposition~\ref{bgwa}.
We have, for $k\geq 0$,
\begin{multline*}
P(\tau_1=k)\,\leq\,
P\big(\tau_1\geq k,T_k>m^{1/4}\big)
\,=\,
P\big(
T_{k}1_{\{\,\tau_1\geq k\,\}}
>m^{1/4}\big)
\cr
\,\leq\,
P\big(Z_k>m^{1/4}\big)\,\leq\,
m^{-1/4}E\big(Z_k\big)\,\leq\,
m^{-1/4}(8\sigma)^k\,.
\end{multline*}
We sum this inequality: for $n\geq 1$,
$$P(\tau_1<n)\,\leq\,
m^{-1/4}\sum_{k=0}^{n-1}(8\sigma)^k\,=\,
m^{-1/4}\frac{(8\sigma)^n-1}{8\sigma-1}\,.$$
We choose $n=\kappa\ln m$, where $\kappa$ is
positive and sufficiently small, and we obtain the desired
conclusion.
\end{proof}
\subsection{Bound on $N^*_n$}
By definition, we have
$$\forall n\geq 1\qquad
N^*_n\,=\,\sum_{i=1}^{m/2}\Big(
1_{\{\,X_n(2i-1)=\un\,\}}+
1_{\{\,X_n(2i)=\un\,\}}\Big)
\,.$$
Let us define, for $1\leq i\leq m/2$,
$$B_n(i)\,=\,
1_{\{\,X_n(2i-1)=\un\,\}}+
1_{\{\,X_n(2i)=\un\,\}}\,.$$
Conditionally on $X_{n-1}$, the variables
$B_n(i)$, $1\leq i\leq m/2$, are independent and identically
distributed. 
A Master sequence appearing in generation~$n$ is either in
the progeny of the initial Master sequence, or it has been
created through numerous mutations and crossover
from $0\cdots 0$. The probability of the first scenario will
be controlled with the help of $T_{n-1}$ (the size of the progeny
of the initial Master sequence in generation~$n-1$).
The second scenario is very unlikely unless $n$ is large.
To control its probability, we introduce
the time~$\tau_2$, when a mutant, not belonging to the progeny
of the initial Master sequence, is at distance less than
$\ell-\sqrt{\ell}$ from the Master sequence.
Let us recall that $H(u,v)$ is the Hamming distance between the
chromosomes $u,v$.
We set
$$\displaylines{
\tau_2\,=\,\inf\,\big\{\,n\geq 1: \exists\,i\in\um\quad\hfill\cr
\hfill
H(X_n(i),\un)\leq \ell-\sqrt{\ell},\,M_n(i)=0\,\big\}\,.}$$
We recall that
$$\tau_1\,=\,\inf\,\big\{\,n\geq 1: T_n>m^{1/4}\,\big\}\,.$$
We set also
$$\tau_0\,=\,\inf\,\big\{\,n\geq 1: N^*_n=0\,\big\}\,.$$
We shall compute a bound on $N^*_n$ until time 
$\tau=\min(\tau_0,\tau_1,\tau_2)$.
\begin{proposition}\label{boundN}
Let $\pi<1$. We suppose that
$$\sigma(1-p_C)(1-p_M)^\ell\,=\,\pi\,.$$
We suppose in addition that $\ell=m$.
For $m$ large enough, the process
$$(N^*_n
1_{\{\,\tau\geq n\,\}}
)_{n\in\mathbb N}$$
is stochastically dominated by
a subcritical
Galton--Watson process.
\end{proposition}
\begin{proof}
We shall estimate the law of $B_n(1)$. The proof is tedious
because there are several cases to consider.
The chromosomes $X_n(1),X_n(2)$ are obtained after 
applying the crossover and the mutation operators on
the chromosomes of the population $X_{n-1}$ having indices
$$I_1\,=\,\cS(X_{n-1},S_n^1)\,,\qquad
I_2\,=\,\cS(X_{n-1},S_n^2)\,.$$
We denote by $Y_1,Y_2$ the chromosomes obtained after crossover
from the chromosomes $X_{n-1}(I_1), X_{n-1}(I_2)$, i.e.,
\begin{align*}
Y_1\,&=\,
\cC\big(
X_{n-1}(I_1),
X_{n-1}(I_2),
V_n^{1}, W_n^{1}\big)\,,\cr
Y_2\,&=\,
\cC\big(
X_{n-1}(I_2),
X_{n-1}(I_1),
V_n^{1}, W_n^{1}\big)\,.
\end{align*}
We suppose also that
$\tau_2\geq n$.
Let $\lambda>0$ be such that 
$\pi/\sigma\geq \exp(-\lambda)$.
We have then
$$(1-p_M)^\ell\,=\,\frac{\pi}{\sigma(1-p_C)}\,\geq\,
\frac{\pi}{\sigma}\,\geq\,
\exp(-\lambda)\,.$$
Notice that $\lambda$ depends only on $\pi/\sigma$, and not on $\ell$ or $p_M$.
By lemma~\ref{binopoi}, the binomial law $\cB(\ell,p_M)$ is then
stochastically dominated by the Poisson law $\cP(\lambda)$. 
We will use repeatedly the bound given in lemma~\ref{poisstail}: 
$$\forall t\geq\lambda\qquad
P\big(U_{n}^{1,1}+\dots+U_{n}^{1,\ell}\geq t\big)\,\leq\,
\Big(\frac{\lambda e}{t}\Big)^t\,.$$
When using this bound, the value of~$t$ will be a function of~$\ell$.
We will always take $\ell$ large enough, so that the
value of $t$ will
be larger than $\lambda$.
We examine several cases, depending on whether the initial Master sequence
belongs to the genealogy of the chromosomes
$X_{n-1}(I_1)$,
$X_{n-1}(I_2)$:
\medskip

\noindent
$\bullet\quad  M_{n-1}(I_1)=M_{n-1}(I_2)=0$.
Since $\tau_2\geq n$, then the chromosomes
$X_{n-1}(I_1)$ and 
$X_{n-1}(I_2)$ have strictly less than $\sqrt{\ell}$ ones,
therefore the chromosomes $Y_1,Y_2$ obtained after crossover
have strictly less than $2\sqrt{\ell}$ ones.
Thus,
for $\ell$ large enough and $*=1$ or $*=2$,
\begin{multline*}
P\Big(
\cM\big(Y_*, U_{n}^{1,1},\dots,U_{n}^{1,\ell}\big)=\un
\Big)\cr
\,\leq\,
P\big(U_{n}^{*,1}+\dots+U_{n}^{*,\ell}>\ell-2\sqrt{\ell}\big)\,\leq\,
\Big(\frac{\lambda e}{\ell-2\sqrt{\ell}}\Big)^{\ell-2\sqrt{\ell}}\,.
\end{multline*}
Therefore
$$
P\left(
\begin{matrix}
B_n(1)>0,\,\tau_2\geq n\\
M_n(I_1)=M_n(I_2)=0
\end{matrix}
\,\Big|\,
\begin{matrix}
X_{n-1}\\
M_{n-1}
\end{matrix}
\right)
\,\leq\,
2\Big(\frac{\lambda e}{\ell-2\sqrt{\ell}}\Big)^{\ell-2\sqrt{\ell}}\,.
$$
\medskip

\noindent
$\bullet\quad  M_{n-1}(I_1)=1$, $M_{n-1}(I_2)=0$.
We estimate first the probability that $B_n(1)=2$.
Since $\tau_2\geq n$, then the chromosome
$X_{n-1}(I_2)$ has strictly less than $\sqrt{\ell}$ ones,
therefore at least one of the chromosomes $Y_1,Y_2$ obtained after crossover
has strictly less than $(\ell+\sqrt{\ell})/2$ ones.
Suppose for instance that it is the case for $Y_1$.
It is then very unlikely that a Master sequence
is created from $Y_1$ with the help of the mutations.
Indeed,
for $\ell$ large enough,
we have 
\begin{multline*}
P\Big(
\cM\big(Y_1, U_{n}^{1,1},\dots,U_{n}^{1,\ell}\big)=\un
,\,
H(Y_1,\un)>
(\ell-\sqrt{\ell})/2
\Big)\cr
\,\leq\,
P\big(U_{n}^{1,1}+\dots+U_{n}^{1,\ell}>
(\ell-\sqrt{\ell})/2 \big)\,\leq\,
\Big(\frac{2\lambda e}{\ell-\sqrt{\ell}}\Big)^{(\ell-\sqrt{\ell})/2}\,.
\end{multline*}
Therefore
$$\displaylines{
P\left(
\begin{matrix}
B_n(1)=2,\,\tau_2\geq n\\
M_n(I_1)=1,M_n(I_2)=0
\end{matrix}
\,\Big|\,
\begin{matrix}
X_{n-1}\\
M_{n-1}
\end{matrix}
\right)
\,\leq\,
2\Big(\frac{2\lambda e}{\ell-\sqrt{\ell}}\Big)^{(\ell-\sqrt{\ell})/2}\,.
}$$
We estimate next the probability than $B_n(1)=1$.
Suppose that a crossover occurs between
$X_{n-1}(I_1), X_{n-1}(I_2)$, i.e.,
that $V_n^1=1$.
Since $\tau_2\geq n$, then the chromosome
$X_{n-1}(I_2)$ has strictly less than $\sqrt{\ell}$ ones.
After crossover,
the probability that one of the two resulting chromosomes $Y_1,Y_2$
has at least $\ell-\sqrt{\ell}$ ones is less than $4/\sqrt{\ell}$.
Indeed, this can happen only if, either on the left of
the cutting site, or on its right, there are at most $\sqrt{\ell}$ zeroes. 
The most favorable
situation is when all the ones in~$Y_2$ are at the end or at the
beginning of~$Y_2$, in which case we have 
$2\sqrt{\ell}$ cutting sites which lead to the desired result.
Now, if a chromosome $u$ is such that
$H(u,\un)>  \sqrt{\ell}$, then
\begin{multline*}
P\Big(
\cM\big(u, U_{n}^{1,1},\dots,U_{n}^{1,\ell}\big)=\un
\Big)\cr
\,\leq\,
P\big(U_{n}^{1,1}+\dots+U_{n}^{1,\ell}>
\sqrt{\ell} \big)\,\leq\,
\Big(\frac{\lambda e}{\sqrt{\ell}}\Big)^{\sqrt{\ell}}\,.
\end{multline*}
From the previous discussion, we conclude that
$$\displaylines{
P\left(
\begin{matrix}
B_n(1)=1,\,V_n^1=1,\,\tau_2\geq n\\
M_n(I_1)=1,M_n(I_2)=0
\end{matrix}
\,\Big|\,
\begin{matrix}
X_{n-1}\\
M_{n-1}
\end{matrix}
\right)
\,\leq\,\hfill\cr\hfill
\Big(
\frac{4}{\sqrt{\ell}}\,+\,
2\Big(\frac{\lambda e}{\sqrt{\ell}}\Big)^{\sqrt{\ell}}\Big)
\Big(F_m(m)-F_m(m-T_{n-1})\Big)\,.
}$$
We consider now the case where 
no crossover occurs between
the chromosomes
$X_{n-1}(I_1), X_{n-1}(I_2)$, i.e.,
we have
$$V_n^1=0\,,\quad
Y_1=X_{n-1}(I_1)\,,\quad Y_2=X_{n-1}(I_2)\,.$$
We write
$$\displaylines{
P\big(B_n(1)=1,\,M_n(I_1)=1, M_n(I_2)=0, V_n^1=0
\,,\tau_2\geq n
\,\big|\,X_{n-1},\,M_{n-1}\big)
\,\leq\,
\hfill\cr
P\left(
\begin{matrix}
\cM\big(X_{n-1}(I_1), U_{n}^{1,1},\dots,U_{n}^{1,\ell}\big)=\un\cr
M_n(I_1)=1 
,\,V_n^1=0
,\,\tau_2\geq n\cr
\end{matrix}
\,\,\Bigg|\,\,
\begin{matrix}
X_{n-1}\cr
M_{n-1}
\end{matrix}
\right)
\hfill
\cr
\hfill
\,+\,
P\left(
\begin{matrix}
\cM\big(X_{n-1}(I_2), U_{n}^{1,1},\dots,U_{n}^{1,\ell}\big)=\un\cr
M_n(I_2)=0 
,\,V_n^1=0
,\,\tau_2\geq n\cr
\end{matrix}
\,\,\Bigg|\,\,
\begin{matrix}
X_{n-1}\cr
M_{n-1}
\end{matrix}
\right)
\,.
}$$
We estimate the last term.
Since $\tau_2\geq n$, then the chromosome
$Y_2=X_{n-1}(I_2)$ has strictly less than $\sqrt{\ell}$ ones.
As before, 
for $\ell$ large enough,
we have 
\begin{multline*}
P\left(
\begin{matrix}
\cM\big(X_{n-1}(I_2), U_{n}^{1,1},\dots,U_{n}^{1,\ell}\big)=\un\cr
M_n(I_2)=0 
,\,V_n^1=0
,\,\tau_2\geq n
\end{matrix}
\,\,\Bigg|\,\,
\begin{matrix}
X_{n-1}\cr
M_{n-1}
\end{matrix}
\right)
\cr
\,\leq\,
P\big(U_{n}^{2,1}+\dots+U_{n}^{2,\ell}>
\ell-\sqrt{\ell} \big)\,\leq\,
\Big(\frac{\lambda e}{\ell-\sqrt{\ell}}\Big)^{\ell-\sqrt{\ell}}\,.
\end{multline*}
Thus it is very unlikely that a Master sequence
is created from 
$Y_2=X_{n-1}(I_2)$ 
with the help of the mutations.
The most likely scenario is that the Master sequence comes from 
$Y_1= X_{n-1}(I_1)$. We estimate the probability of this scenario, 
and to do so,
we distinguish further two cases, according to whether $X_n(I_1)$
is a Master sequence or not:
$$\displaylines{
P\left(
\begin{matrix}
\cM\big(X_{n-1}(I_1), U_{n}^{1,1},\dots,U_{n}^{1,\ell}\big)=\un\cr
M_n(I_1)=1 
,\,V_n^1=0
,\,\tau_2\geq n
\end{matrix}
\,\,\Bigg|\,\,
\begin{matrix}
X_{n-1}\cr
M_{n-1}
\end{matrix}
\right)\,=\,\hfill\cr
P\left(
\begin{matrix}
\cM\big(X_{n-1}(I_1), U_{n}^{1,1},\dots,U_{n}^{1,\ell}\big)=\un\cr
X_{n-1}(I_1)=\un,\,
M_n(I_1)=1 
,\,V_n^1=0
,\,\tau_2\geq n
\end{matrix}
\,\,\Bigg|\,\,
\begin{matrix}
X_{n-1}\cr
M_{n-1}
\end{matrix}
\right)\cr
\,+\,
P\left(
\begin{matrix}
\cM\big(X_{n-1}(I_1), U_{n}^{1,1},\dots,U_{n}^{1,\ell}\big)=\un\cr
X_{n-1}(I_1)\neq\un,\,
M_n(I_1)=1 
,\,V_n^1=0
,\,\tau_2\geq n
\end{matrix}
\,\,\Bigg|\,\,
\begin{matrix}
X_{n-1}\cr
M_{n-1}
\end{matrix}
\right)\,.
}$$
To estimate these probabilities, we make an intermediate conditioning and
we obtain
$$\displaylines{
P\left(
\begin{matrix}
\cM\big(X_{n-1}(I_1), U_{n}^{1,1},\dots,U_{n}^{1,\ell}\big)=\un\cr
X_{n-1}(I_1)=\un,\,
M_n(I_1)=1 
,\,V_n^1=0
,\,\tau_2\geq n
\end{matrix}
\,\,\Bigg|\,\,
\begin{matrix}
X_{n-1}\cr
M_{n-1}
\end{matrix}
\right)\,\leq\,\hfill\cr
P\Big(
\cM\big(\un, U_{n}^{1,1},\dots,U_{n}^{1,\ell}\big)=\un\Big)
\hfill\cr
\hfill\times
P\left(
\begin{matrix}
X_{n-1}(I_1)=\un\cr
M_n(I_1)=1 
,\,V_n^1=0
,\,\tau_2\geq n
\end{matrix}
\,\,\Bigg|\,\,
\begin{matrix}
X_{n-1}\cr
M_{n-1}
\end{matrix}
\right)\cr
\,\leq\,
(1-p_M)^\ell
(1-p_C)
\Big(F_m(m)-F_m(m-N_{n-1}^*)\Big)
\,.
}
$$
Indeed, the number of Master sequences present in $X_{n-1}$
is $N_{n-1}^*$ and
the probability of selecting a Master sequence in $X_{n-1}$
is at most
$(F_m(m)-F_m(m-N_{n-1}^*))$.
In a similar way,
$$\displaylines{
P\left(
\begin{matrix}
\cM\big(X_{n-1}(I_1), U_{n}^{1,1},\dots,U_{n}^{1,\ell}\big)=\un\cr
X_{n-1}(I_1)\neq\un,\,
M_n(I_1)=1 
,\,V_n^1=0
,\,\tau_2\geq n
\end{matrix}
\,\,\Bigg|\,\,
\begin{matrix}
X_{n-1}\cr
M_{n-1}
\end{matrix}
\right)\,\leq\,\hfill\cr
\sum_{u\neq\un}
P\Big(
\cM\big(u, U_{n}^{1,1},\dots,U_{n}^{1,\ell}\big)=\un\Big)
P\left(
\begin{matrix}
X_{n-1}(I_1)=u\cr
M_n(I_1)=1 
,\,\tau_2\geq n
\end{matrix}
\,\,\Bigg|\,\,
\begin{matrix}
X_{n-1}\cr
M_{n-1}
\end{matrix}
\right)\cr
\,\leq\,
p_M
P\left(
\begin{matrix}
X_{n-1}(I_1)\neq \un\cr
M_n(I_1)=1 
,\,\tau_2\geq n
\end{matrix}
\,\,\Bigg|\,\,
\begin{matrix}
X_{n-1}\cr
M_{n-1}
\end{matrix}
\right)
\,\leq\,
p_M
\Big(F_m(m)-F_m(m-T_{n-1})\Big)
\,.  }
$$
Putting together the previous inequalities, we obtain
$$\displaylines{
P\big(B_n(1)=1,\,M_n(I_1)=1, M_n(I_2)=0,\,\tau_2\geq n
\,\big|\,X_{n-1},\,M_{n-1}\big)
\hfill\cr
\,\leq\,
\Big(
\frac{4}{\sqrt{\ell}}\,+\,
2\Big(\frac{\lambda e}{\sqrt{\ell}}\Big)^{\sqrt{\ell}}
\Big)
\Big(F_m(m)-F_m(m-T_{n-1})\Big)
\,+\,
\Big(\frac{\lambda e}{\ell-\sqrt{\ell}}\Big)^{\ell-\sqrt{\ell}}
\,+\,\hfill
\cr
(1-p_M)^\ell
(1-p_C)
\Big(F_m(m)-F_m(m-N_{n-1}^*)\Big)
\,+\,
p_M
\Big(F_m(m)-F_m(m-T_{n-1})\Big)
\,.
}$$

\noindent
$\bullet\quad  M_{n-1}(I_1)=M_{n-1}(I_2)=1$.
In this case, we have
$$\displaylines{
P\big(M_n(I_1)=1, M_n(I_2)=1\,,\tau_2\geq n
\,\big|\,X_{n-1},\,M_{n-1}\big)
\hfill\cr
\hfill
\,\leq\,
\Big(F_m(m)-F_m(m-T_{n-1})\Big)^2
\,.
}$$
In conclusion, we obtain the following inequalities:
$$\displaylines{
P\big(B_n(1)=2\,,\tau_2\geq n
\,\big|\,X_{n-1},\,M_{n-1}\big)
\,\leq\,
\hfill\cr
\hfill
2\Big(\frac{\lambda e}{\ell-2\sqrt{\ell}}\Big)^{\ell-2\sqrt{\ell}}\,+\,
4\Big(\frac{2\lambda e}{\ell-\sqrt{\ell}}\Big)^{(\ell-\sqrt{\ell})/2}
\,+\,
\Big(F_m(m)-F_m(m-T_{n-1})\Big)^2
\,,\cr
P\big(B_n(1)=1\,,\tau_2\geq n
\,\big|\,X_{n-1},\,M_{n-1}\big)
\,\leq\,\hfill\cr
2\Big(\frac{\lambda e}{\ell-2\sqrt{\ell}}\Big)^{\ell-2\sqrt{\ell}}\,+\,
\Big(
\frac{8}{\sqrt{\ell}}\,+\,
4\Big(\frac{\lambda e}{\sqrt{\ell}}\Big)^{\sqrt{\ell}}\Big)
\Big(F_m(m)-F_m(m-T_{n-1})\Big)
\hfill\cr
\,+\,
2\Big(\frac{\lambda e}{\ell-\sqrt{\ell}}\Big)^{\ell-\sqrt{\ell}}
\,+\,
2(1-p_M)^\ell
(1-p_C)
\Big(F_m(m)-F_m(m-N_{n-1}^*)\Big)
\cr
\hfill
\,+\,
2p_M
\Big(F_m(m)-F_m(m-T_{n-1})\Big)
\,+\,
\Big(F_m(m)-F_m(m-T_{n-1})\Big)^2
\,.
}$$
In order to incorporate the event $\{\,\tau\geq n\,\}$ in
these inequalities, we condition with respect to the whole history
of the process as follows: for $*=1$ or $*=2$,
$$\displaylines{
P\big(B_n(1)=*,\,
\tau\geq n\,\big|\,X_{n-1},M_{n-1},\dots,X_0,M_0\,\big)\hfill\cr
\,\leq\,
1_{\{\,\tau\geq n\,\}}
P\big(B_n(1)=*
\,\big|\,X_{n-1},M_{n-1}\,\big)\,.
}$$
Let $\varepsilon>0$ be such that
$$(1+5\varepsilon)\sigma(1-p_C)(1-p_M)^\ell\,=\,
(1+5\varepsilon)
\pi\,<\,1\,.$$
We use next the hypothesis on the selection function:
there exists $c>0$ such that,
for $m$ large enough, 
$$1_{\{\,\tau\geq n\,\}}
\Big(F_m(m)-F_m(m-T_{n-1})\Big)
\,\leq\,
1_{\{\,\tau\geq n\,\}}
\frac{c }{m}
T_{n-1}
\,\leq\,
\frac{c }{m^{3/4}}
\,.$$
Moreover, 
for $m$ large enough, 
$$1_{\{\,\tau\geq n\,\}}
\Big(F_m(m)-F_m(m-N_{n-1}^*)\Big)
\,\leq\,
1_{\{\,\tau\geq n\,\}}
\frac{ \sigma(1+\varepsilon) }{m}
N_{n-1}^*
\,.$$
Thus there exists a constant $c>0$ such that,
for $m,\ell$ large enough,
$$\displaylines{
1_{\{\,\tau\geq n\,\}}
P\big(B_n(1)=2
\,\big|\,X_{n-1},\,M_{n-1}\big)
\,\leq\,
\frac{1 }{\ell^2}
+
\frac{c^2 }{m^{3/2}}
\,,\hfill\cr
1_{\{\,\tau\geq n\,\}}
P\big(B_n(1)=1
\,\big|\,X_{n-1},\,M_{n-1}\big)
\,\leq\,
\frac{1 }{\ell^2}
+
\frac{8 }{\sqrt\ell}
\frac{c }{m^{3/4}}
\hfill\cr
+
\frac{2}
{m}
 \sigma(1+\varepsilon) 
1_{\{\,\tau\geq n\,\}}
N_{n-1}^*
(1-p_M)^\ell
(1-p_C)
\,+\,
2p_M
\frac{c }{m^{3/4}}
\,+\,
\frac{c^2 }{m^{3/2}}
\,.
}$$
We rewrite the previous inequalities in the case $\ell=m$.
First, we have, for a positive constant $c$,
$$P\big(B_n(1)=2\,,\tau\geq n
\,\big|\,X_{n-1},M_{n-1},\dots,X_0,M_0\,\big)
\,\leq\,
\frac{c}{m^{3/2}}
\,.$$
Moreover
$\sigma(1-p_M)^m\,\geq\,\pi$,
whence
$$p_M\,\leq\,-\frac{1}{m}\ln(\pi/\sigma)\,.$$
For $m$ large enough, we have therefore
$$\frac{1 }{m^2}
+
\frac{8 }{\sqrt m}
\frac{c }{m^{3/4}}
\,+\,
2p_M
\frac{c }{m^{3/4}}
\,+\,
\frac{c^2 }{m^{3/2}}
\,\leq\,
\frac{2} {m} \pi\varepsilon \,,
$$
and it follows that
$$\displaylines{
P\big(B_n(1)=1\,,\tau\geq n
\,\big|\,X_{n-1},M_{n-1},\dots,X_0,M_0\,\big)
\,\leq\,
\hfill\cr
\hfill\frac{2}
{m}
 \pi(1+2\varepsilon) 
1_{\{\,\tau\geq n\,\}}
N_{n-1}^*
\,.
}$$
Coming back to the initial equality for~$N_n^*$, 
we conclude that, for $m$ large enough,
the law of
$N^*_{n}
1_{\{\,\tau\geq n\,\}}$ is stochastically dominated by the sum of two
independent binomial random variables as follows:
$$
N^*_{n}
1_{\{\,\tau\geq n\,\}}\,\preceq\,
\cB\Big(
\frac{m}{2},
\frac{2
}{m}
 \pi(1+2\varepsilon) 
1_{\{\,\tau\geq n\,\}}
N^*_{n-1}
\Big)
+
2\cB\Big(
\frac{m}{2},
\frac{c}{m^{3/2}}
\Big)\,.
$$
For $m$ large, these two binomial laws
are in turn stochastically
dominated by two Poisson laws.
More precisely, 
for $m$ large enough,
$$\displaylines{
\Big(1-
\frac{2}{m}
 \pi(1+2\varepsilon) 
1_{\{\,\tau\geq n\,\}}
N^*_{n-1}
\Big)^{m/2}\geq
\exp\Big(-
 \pi(1+3\varepsilon) 
1_{\{\,\tau\geq n\,\}}
N^*_{n-1}\Big)
\,,\cr
\big(1-
{c}{m^{-3/2}}
\big)^{m/2}\geq\exp(-\varepsilon)\,.}$$
Lemma~\ref{binopoi} yields then that
$$
N^*_{n}
1_{\{\,\tau\geq n\,\}}\,\preceq\,
\cP\big(
 \pi(1+3\varepsilon) 
N^*_{n-1}
1_{\{\,\tau\geq n\,\}}
\big)
+
2\cP(
\varepsilon
)\,.
$$
The point is that we have got rid of the variable $m$ in the upper bound, so we are
now in position to compare 
$N^*_{n}
1_{\{\,\tau\geq n\,\}}$
with a Galton--Watson process.
Let 
$(Y'_n)_{n\geq 1}$
be a sequence of i.i.d. random variables with law
$\cP(
 \pi(1+3\varepsilon) 
)$, 
let 
$(Y''_n)_{n\geq 1}$
be a sequence of i.i.d. random variables with law
$\cP(\varepsilon)$, both sequences being independent.
The previous stochastic inequality can be rewritten as
$$
N^*_{n}
1_{\{\,\tau\geq n\,\}}\,\preceq\,
\Big(
\sum_{k\geq 1}^{
N^*_{n-1}
1_{\{\,\tau\geq n\,\}}}
Y'_k
\Big)
+2Y''_1
\,.
$$
This implies further that
$$
N^*_{n}
1_{\{\,\tau\geq n\,\}}\,\preceq\,
\sum_{k\geq 1}^{
N^*_{n-1}
1_{\{\,\tau\geq n-1\,\}}}
\big(
Y'_k+2Y''_k
\big)\,.
\qquad(\star)
$$
Let $\nu^*$ be the law of 
$Y'_1+2Y''_1$  and let
$(Z^*_n)_{n\geq 0}$ be a Galton--Watson process starting from $Z_0=1$
with reproduction law~$\nu^*$. We prove finally that,
for $m$ large enough, 
$$\forall n\geq 0\qquad
N^*_{n}
1_{\{\,\tau\geq n\,\}}\,\preceq\,
Z^*_n\,.$$
We suppose that $m$ is large enough so that the stochastic 
inequality~$(\star)$ holds and we proceed by induction on~$n$.
For $n=0$, we have
$$
N^*_{0}
1_{\{\,\tau\geq 0\,\}}\,=\,1\,\leq\,
Z^*_0\,=\,1\,.$$
Let $n\geq 1$ and 
suppose that the inequality holds at rank $n-1$.
Inequality~$(\star)$ yields
$$
N^*_{n}
1_{\{\,\tau\geq n\,\}}\,\preceq\,
\sum_{k\geq 1}^{
N^*_{n-1}
1_{\{\,\tau\geq n-1\,\}}}
\kern-7pt
\big(
Y'_k+2Y''_k
\big)
\,\preceq\,
\sum_{k\geq 1}^{
Z^*_{n-1}
}
\big(
Y'_k+2Y''_k
\big)
\,=\,
Z^*_{n}
\,.
$$
Thus the inequality holds at rank~$n$ and the induction is
completed.
Moreover we have
$$E(\nu^*)\,=\,E
\big(
Y'_1+2Y''_1
\big)\,=\,
 \pi(1+5\varepsilon) \,<\,1\,.
$$
Thus the 
Galton--Watson process 
$(Z^*_n)_{n\geq 0}$ is subcritical and
this completes the proof of
proposition~\ref{boundN}.
\end{proof}

\noindent
We close this section with a bound on~$\tau_2$, which will be useful
when applying proposition~\ref{boundN}.
\begin{lemma}\label{bounddau}
For $m\geq 2$ and for $\ell$ large enough, we have
$$P\Big(\tau_2\leq
\frac{1}
{5}
{\ln\ell}
\Big)\,\leq\,
1-\exp(-m\exp\big(-\ell^{1/4}\big)\Big)\,.
$$
\end{lemma}
\begin{proof}
If $\tau_2<n$, then, before time $n$, a chromosome has been
created with at least $\sqrt{\ell}$ ones, and whose genealogy
does not contain the initial Master sequence.
We shall compute an upper bound on the number of ones
appearing in the genealogy of such a chromosome at generation~$n$.
Let us define
$$\forall n\geq 1\qquad
D_n\,=\,
\max\,\big\{\,
\ell-H(X_n(i),\un):1\leq i\leq m,\,M_n(i)=0\,\big\}\,.$$
The quantity $D_n$ is the maximum number of ones in 
a chromosome
of the generation~$n$, which does not belong to the progeny
of the initial Master sequence.
These ones must have been created by mutation.
Let us consider 
a chromosome
of the generation~$n+1$, which does not belong to the progeny
of the initial Master sequence.
The number of ones in each of its two parents was at most~$D_n$.
After crossover between these two parents, the number of ones
was at most~$2D_n$.
After mutation, the number of ones
was at most
$$D_{n+1}\,\leq\, 2D_n+
\max\,\Big\{
\sum_{j=1}^\ell U_n^{i,j}:1\leq i\leq m\Big\}\,.$$
We first control the last term.
Let $n\geq 1$ and let us define the event
$\cE(n)$ by
$$\cE(n)\,=\,
\Big\{\,
\forall i\in\um\quad
\forall k\in\{\,1,\dots,n\,\}\quad
\sum_{j=1}^\ell U_k^{i,j}\,\leq\,\ell^{1/4}\,\Big\}\,.$$
We have
$$\displaylines{
P\big(\cE(n)\big)
\,=\,\Big(1-
P\Big(
\sum_{j=1}^\ell U_1^{1,j}>\ell^{1/4}\Big)\Big)^{mn}\,.
}$$
The law of the sum
$\sum_{j=1}^\ell U_1^{1,j}$
is the Bernoulli law $\cB(\ell,p_M)$.
Let $\lambda>0$ be such that 
$\pi/\sigma\geq \exp(-\lambda)$.
We have then
$$(1-p_M)^\ell\,=\,\frac{\pi}{\sigma(1-p_C)}\,\geq\,
\frac{\pi}{\sigma}\,\geq\,
\exp(-\lambda)\,.$$
By lemma~\ref{binopoi}, the binomial law $\cB(\ell,p_M)$ is 
stochastically dominated by the Poisson law $\cP(\lambda)$. 
Using the bound given in lemma~\ref{poisstail}, we obtain that,
for $\ell^{1/4}>\lambda$, 
$$
P\big(\cE(n)\big)
\,\geq\,\Big(1-
\Big(\frac{\lambda e}{\ell^{1/4}}\Big)^{\ell^{1/4}}\Big)^{mn}\,,$$
whence, for $\ell$ large enough,
$$
P\big(\cE(n)\big)
\,\geq\,\exp\Big(-mn\exp\big(-\ell^{1/4}\big)\Big)\,.$$
Suppose that the event $\cE(n)$ occurs. We have then
$$\forall k\in\{\,0,\dots,n-1\,\}\qquad
D_{k+1}\,\leq\, 2D_k+\ell^{1/4}\,.$$
Dividing by $2^{k+1}$ and summing from $k=0$ to $n-1$, we get
$$D_n\,\leq\,2^n\sum_{k=0}^{n-1}
\frac{\ell^{1/4}}{2^{k+1}}
\,\leq\,
2^n
{\ell^{1/4}}
\,.$$
Therefore, if $2^n<
{\ell^{1/4}}$ and if the event $\cE(n)$ occurs, then
$\tau_2>n$. Taking $n=(\ln\ell)/5$, 
we obtain the estimate stated in the lemma.
\end{proof}

\subsection{Proof of theorem~\ref{thspl}}
We complete here the proof of
theorem~\ref{thspl}.
The hypothesis of 
theorem~\ref{thspl} allows to apply
proposition~\ref{boundN}. Thus there exists a subcritical
Galton--Watson process 
$(Z^*_n)_{n\geq 0}$, with reproduction law~$\nu^*$,
which dominates stochastically
the process $(N^*_{n}
1_{\{\,\tau\geq n\,\}})_{n\geq 0}$.
A standard result on Galton--Watson processes 
(see for instance~\cite{AN})
ensures the existence
of a positive constant $c^*$, which depends only on the law~$\nu^*$,
such that
$$\forall n\geq 1\qquad P\big(Z_n^*>0\big)\,\leq\,\exp(-cn)\,.$$ 
Let $\kappa,c_1$ be as in proposition~\ref{boundtau}. 
We suppose that $\kappa<1/5$, so that we can use the estimate
of
lemma~\ref{bounddau}. 
We have then
$$\displaylines{
P\big(\tau_0> \kappa\ln m\big)\,\leq\,
\hfill\cr
P\big(
\tau_0> \kappa\ln m,\,
\tau< \kappa\ln m\big)\,+\,
P\big(
N_{\lfloor\kappa\ln m\rfloor}^*>0,\,
\tau\geq \kappa\ln m\big)\cr
\,\leq\,
P\big(\tau_1<\kappa\ln m\big)\,+\,
P\big(\tau_2<\kappa\ln m\big)\,+\,
P\big(Z_{\lfloor\kappa\ln m\rfloor}^*>0\big)\cr
\,\leq\,
\frac{1}{m^{c_1}}
\,+\,
1-\exp(-m\exp\big(-m^{1/4}\big)\Big)
\,+\,
\,\exp(-c^* {\lfloor\kappa\ln m\rfloor})\,.
}$$
This inequality yields the estimate stated in theorem~\ref{thspl}.
\section{The quasispecies regime}\label{secqua}
For $\lambda\in\mathbb R$ and a population~$x$, we define $N(x,\lambda)$ as the
number of chromosomes in $x$ whose fitness is larger than or equal to~$\lambda$:
$$N(x,\lambda)\,=\,\card\,\{\,i\in\um:f(x(i))\geq\lambda\,\}\,.$$
We shall couple the processes
$(N(X_n,\lambda))_{n\in\mathbb N}$, $\lambda\in\mathbb R$, 
with a family
of Markov chains 
$\smash{\big(N_n(t,i)\big)_{n\geq t}}$,
$t\in\mathbb N$. 
We study then dynamics of these
Markov chains.

\subsection{The auxiliary chain}
Here, we couple the genetic algorithm with an auxiliary chain.
%
For $n\geq 1$,
let 
$\Gamma_n :\um^2\to\zu$
be the map defined by
$$\forall i,j\in\um\qquad
\Gamma_n(i,j)\,=\,
\begin{cases}
1 &\text{if}\quad \cI(S_n^j)\geq m-i+1\\
0 &\text{otherwise }\\
\end{cases}
$$
Recall that the map $\cI$ is defined together with the selection map $\cS$.
In fact, the map $\Gamma_n(i,j)$ is equal to one if the $j$--th 
chromosome
chosen during the selection at time~$n$ is among the best $i$
chromosomes of generation~$n$.
For each $n\geq 1$,
we define also a map $\Psi_n:\zm\to\zm$ by setting
$$\displaylines{
\forall i\in\zm
\hfill\cr
\Psi_n(i)\,=\,
\sum_{j=1}^{m}
\Big(
\Gamma_n(i,j)\,(1-V_n^{\lceil j/2\rceil})
\,\prod_{k=1}^\ell(1-U_n^{j,k})
\Big)
\,.
}$$
The map $\Psi_n(i)$ counts the number of chromosomes in generation~$n+1$
which have been obtained by
selecting a chromosome among
the best $i$ chromosomes of generation~$n$,
and for which there was no crossover and no subsequent mutation.
For any $j$, 
the map 
$i\mapsto \Gamma_n(i,j)$ is non--decreasing, therefore the map
$i\mapsto \Psi_n(i)$ is also non--decreasing.
For $t\in\N$ and $i\in\um$, let
$$\big(N_n(t,i)\big)_{n\geq t}$$
be the Markov chain starting from~$i$ at time~$t$
and defined by
$$\forall n\geq t\qquad
N_{n+1}(t,i)\,=\,\psi_n
\big(N_{n}(t,i)\big)\,.$$
\begin{proposition}\label{couplnl}
For any $t\in\N$ and $\lambda\in\R$, we have
$$\forall n\geq t\qquad
N(X_n,\lambda)\,\geq\,
N_{n}\big(t,N(X_t,\lambda)\big)\,.$$
\end{proposition}
\begin{proof}
Let us fix $\lambda\in\mathbb R$.
At time $t$, there is equality.
We prove the inequality by induction over~$n$.
Suppose that the inequality holds at time $n\geq t$.
The value $\Psi_n(N(X_n,\lambda))$ is a lower bound on the number of
chromosomes in generation~$n+1$ which are an exact copy of one 
of the
chromosomes of generation~$n$ which have a fitness larger than or 
equal to~$\lambda$.
Therefore
$$N(X_{n+1},\lambda)\,\geq\,
\Psi_n\big(N(X_n,\lambda)\big)\,.$$
Using the inequality at time~$n$ and the monotonicity of~$\Psi_n$, we get
$$\Psi_n\big(N(X_n,\lambda)\big)\,\geq\,
\Psi_n\big(N_n(t,N(X_t,\lambda))\big)\,=\,
N_{n+1}\big(t,N(X_t,\lambda)\big)\,.$$
The induction step is completed.
\end{proof}
\subsection{Transition probabilities of $N_n$}
To alleviate the notation, we suppose that 
the Markov chain
$\smash{\big(N_n(t,i)\big)_{n\geq t}}$
starts 
at time $0$, we remove
$t,i$ from the notation and
we write simply
$(N_n)_{n\geq 0}$.
Let us compute the transition probabilities of
$(N_n)_{n\geq 0}$.
The null state is absorbing for the Markov chain
$(N_n)_{n\geq 0}$. 
By definition, we have
$$\displaylines{
\forall n\geq 1\qquad
N_{n+1}
\,=\,
\sum_{j=1}^{m}
\Big(
\Gamma_n(N_n,j)\,(1-V_n^{\lceil j/2\rceil})
\,\prod_{k=1}^\ell(1-U_n^{j,k})
\Big)
\,.
}$$
The random variable $N_{n+1}$ is a sum of $m$
identically distributed
Bernoulli random variables, whose parameter depends on the
value of $N_n$. Yet these random variables are not independent,
because the
crossover creates a correlation between two 
consecutive chromosomes
(through the variable 
$\smash{V_n^{\lceil j/2\rceil}}$).
In order to get rid of this correlation, we first count 
the number of crossover occurring in generation~$n$, and then
we sum over the indices where no crossover has taken place.
Let $B_n$ be the random variable defined by
$$B_n=
\frac{m}{2}-\sum_{j=1}^{m/2}
V_n^{j}\,.$$
The law of $B_n$ is the binomial law $\cB(m/2,1-p_C)$.
Conditionally on $N_n=i$, the law of $N_{n+1}$ is the same
as the law of the random variable
$$\sum_{k=1}^{2B_n}Y_k^i\,,$$
where the variables
$Y_k^i$, $k\in\N$, $i\in\um$, are Bernoulli i.i.d. random variables
(independent of $B_n$ as well) with parameter
$$\varepsilon_m(i)\,=\,
\big(F_m(m-i+1)+\cdots+F_m(m)\big)
(1-p_M)^\ell\,.$$
Finally, we have
for $i\in\um$ and $j\in\zm$,
$$\displaylines{
P\big(
N_{n+1}=j
\,|\,
N_{n}=i
\,\big)
\,=\,
\hfill\cr
\sum_{b=0}^{m/2}
\begin{pmatrix}
{m/2}\\
{b}
\end{pmatrix}
(1-p_C)^b
p_C^{m/2-b}
\begin{pmatrix}
{2b}\\
{j}
\end{pmatrix}
\varepsilon_m(i)^j
(1-\varepsilon_m(i))^{2b-j}\,.
}$$

\subsection{Large deviations upper bound}
The formula for the transition probabilities
is very complicated, so we will study its asymptotics
as $m$ goes to~$\infty$. The goal of this section is to prove
the large deviations upper bound stated in proposition~\ref{ldup}.
We do not have a genuine large deviations principle for the
chain
$(N_n)_{n\geq 0}$, because there is some freedom left for the 
parameters $p_C,p_M,\ell$. In order to derive a corresponding
lower bound, we would have to fix the limiting value of
$p_C$ and $(1-p_M)^\ell$. However we wish to focus on the role
of the parameter $\pi$, and for our purpose, we need only
a large deviations upper bound under the constraint
$$\pi
\,=\,\sigma(1-p_C)(1-p_M)^\ell\,.$$
For $p\in[0,1]$ and $t\geq 0$, we define
$$
I(p,t)\,=\,
\begin{cases}
\displaystyle t\ln\frac{t}{p}
+(1-t)\ln\frac{1-t}{1-p}& 0<p<1,\, 0\leq t\leq 1\\
\quad\qquad\qquad 0&\text{$t=p=0$ or $t=p=1$} \\
\qquad\qquad +\infty &\text{$(p\in\{0,1\},\, t\neq p)$ 
or $t>1$ or $p>1$} \\
\end{cases}
$$
The function $I(p,\cdot)$ is the rate function governing the large deviations
of the binomial distribution $\cB(n,p)$ with parameters $n$ and $p$.
It is standard that $I$ is lower semicontinuous with respect to~$t$,
yet we will need a little more, as stated in the next lemma.
\begin{lemma}\label{lscbnp}
The map $I(p,t)$
is sequentially lower semicontinuous in 
$p,t$, i.e., 
for any 
$p\in[0,1]$, $t\in {\mathbb R}^+$, any sequences
$(p_n)_{n\geq 1}$,
$(t_n)_{n\geq 1}$ converging towards $p,t$, we have
$$\liminf_{n\to\infty} \,I(p_n,t_n)\,\geq\,I(p,t)\,.$$
\end{lemma}
\begin{proof}
We need only to distinguish a few cases.
For $0<p<1$, the result is straightforward.
If $p=0$ and $t>0$, 
or if $p=1$ and $t<1$, 
we check that
$$\liminf_{n\to\infty} \,I(p_n,t_n)\,=\,+\infty\,.$$
If $p=t=0$ or if $p=t=1$, the inequality holds 
since $I(0,0)=I(1,1)=0$.
\end{proof}

\noindent
We define, for 
$s\in]0,1]$,
$t\in[0,1]$,
$$\displaylines{
V_{1}(s,t)\,=\,
\inf\Big\{\,
\frac{1}{2}I\big(1-p,\beta\big)
+\beta
I\Big(
\frac{
\big(1-F(1-s)\big)
\pi}{\sigma(1-p)},
\frac{t}{\beta}\Big):\hfill\cr
\hfill
0\leq p\leq 1-\frac{\pi}{\sigma},\,t\leq\beta\leq 1\,\Big\}
\,.
}$$
We set also $V_1(0,0)=0$ and $V_1(0,t)=+\infty$ for $t>0$.
\begin{lemma}\label{lscp}
The map 
$V_{1}(s,t)$ is sequentially lower semicontinuous
in $s,t$, i.e.,
for any 
$s,t\in[0,1]$, any sequences
$(s_n)_{n\geq 1}$,
$(t_n)_{n\geq 1}$ converging towards $s,t$, we have
$$\liminf_{n\to\infty} \,V_1(s_n,t_n)\,\geq\,V_1(s,t)\,.$$
\end{lemma}
\begin{proof}
Let $s,t\in[0,1]$.
Let
$(s_n)_{n\geq 1}$,
$(t_n)_{n\geq 1}$ be two sequences in $[0,1]$
which converge towards $s,t$.
For each $n\geq 1$, let
$p_n$ and $\beta_n$
be such that
$$\displaylines{
0\leq {p_n}\leq 1-{\pi}/{\sigma}\,,\quad
t_n\leq {\beta}_n\leq 1\,,\cr
\frac{1}{2}I\big(1-{p_n},\beta_n\big)
+{\beta_n}
I\Big( \frac{ \big(1-F(1-s_n)\big)
\pi}{\sigma(1-{p_n})},
\frac{t_n}{\beta_n}\Big)
\,\leq\,
V_{1}(s_n,t_n)
+\frac{1}{n}\,.}$$
By compactness, up to the extraction of a subsequence, we can
suppose that there exist $\widetilde{p}$, $\widetilde{\beta}$,
$\widetilde{\gamma}$
such that
$$\displaylines{
0\leq \widetilde{p}\leq 1-{\pi}/{\sigma}\,,\quad
t\leq \widetilde{\beta}\leq 1\,,\cr
\lim_{n\to\infty} p_n\,=\,\widetilde{p}\,,\qquad
\lim_{n\to\infty} \beta_n\,=\,\widetilde{\beta}\,,\qquad
\lim_{n\to\infty} \frac{t_n}{\beta_n}\,=\,\widetilde{\gamma}\,.}$$
Using the continuity of $F$ and the lower semicontinuity of $I$,
we obtain 
$$\liminf_{n\to\infty} \,V_1(s_n,t_n)\,\geq\,
\frac{1}{2}I\big(1-\widetilde{p},\widetilde{\beta}\big)
+\widetilde{\beta}
I\Big( \frac{ \big(1-F(1-s)\big)
\pi}{\sigma(1-\widetilde{p})},
\widetilde{\gamma}
\Big)\,.$$
Let us denote by $\Delta$ the righthand quantity.
We distinguish several cases:
\medskip

\noindent
$\bullet\quad  t>0$. We have then $\widetilde{\gamma}=t/\widetilde{\beta}$,
whence $\Delta\geq V_1(s,t)$.

\medskip

\noindent
$\bullet\quad  t=0,\widetilde{\beta}>0$. 
We have then $\widetilde{\gamma}=0$,
whence $\Delta\geq V_1(s,0)$.
\medskip

\noindent
$\bullet\quad  t=0,\widetilde{\beta}=0, s>0$. 
We have then 
$\Delta
\geq
\frac{1}{2}I\big(1-\widetilde{p},0\big)
\geq V_1(s,0)$.
\medskip

\noindent
$\bullet\quad  t=0,\widetilde{\beta}=0, s=0$. Obviously, $\Delta\geq V_1(0,0)=0$.
\medskip

\noindent
In each case, we conclude that
$\Delta\geq V_1(s,t)$. This shows that $V_1$ is lower semicontinuous.
\end{proof}
\begin{proposition}\label{ldup}
For any $s\in[0,1]$, any subset $U$ of $[0,1]$, we have,
for any $n\geq 1$,
 $$\limsup_{m\to\infty}\,\frac{1}{m}\ln
P\big(
N_{n+1}\in mU
\,|\,
N_{n}=\lfloor sm\rfloor
\,\big)
\,\leq\,
-\inf_{t\in U}\,
V_{1}(s,t)\,.
$$
\end{proposition}
\begin{proof}
Let $n\geq 1$, let $i\in\um$ and $j\in\zm$.
From lemma~\ref{cnk}, we obtain that, for any $b\leq m/2$,
$$\displaylines{
\begin{pmatrix}
{m/2}\\
{b}
\end{pmatrix}
(1-p_C)^b
p_C^{m/2-b}
\begin{pmatrix}
{2b}\\
{j}
\end{pmatrix}
\varepsilon_m(i)^j
(1-\varepsilon_m(i))^{2b-j}\,\leq\,
\hfill\cr
\exp\,
\bigg(\,
-m\,
\Big\{
\frac{1}{2}I\Big(1-p_C,\frac{2b}{m}\Big)
+
\frac{2b}{m}I\Big(\varepsilon_m(i),\frac{j}{2b}\Big)\Big\}
+4\ln m + 6\bigg)\,.}
$$
We take the maximum with respect to $b$, we sum and we get
$$\displaylines{
P\big(
N_{n+1}=j
\,|\,
N_{n}=i
\,\big)
\,\leq\,
(m+1)\,\times
\hfill\cr
\exp\,
\bigg(\,
-m\,\min_{0\leq b\leq m/2}
\Big\{
\frac{1}{2}I\Big(1-p_C,\frac{2b}{m}\Big)
+
\frac{2b}{m}I\Big(\varepsilon_m(i),\frac{j}{2b}\Big)\Big\}
+4\ln m + 6\bigg)\,.
}$$
We seek next a large deviations upper bound for the transition
probabilities.
Let $s\in [0,1]$ and let us take $i=\lfloor ms\rfloor$.
We first consider the cases $s=0$ and $s=1$.
For $s=0$, we have 
$\varepsilon_m(0)=0$ and
$$P\big(
N_{n+1}\in mU
\,|\,
N_{n}=0
\,\big)
\,=\,
\begin{cases}
0\quad\text{if}\quad 0\not\in mU\\
1\quad\text{if}\quad 0\in mU\\
\end{cases}
$$
and the inequality stated in the lemma holds.
Suppose that $s=1$. We have then
$$\varepsilon_m(m)
\,=\,(1-p_M)^\ell
\,=\,
\frac{ \pi}{\sigma(1-p_C)}\,.$$
It follows that
$$\displaylines{
P\big(
N_{n+1}\in mU
\,|\,
N_{n}=m
\,\big)
\,\leq\,
(m+1)^2\,
\exp\big(4\ln m + 6\big)
\times
\hfill\cr
\exp\,
\bigg(\,
-m\,
\min_{0\leq b\leq m/2}
\min_{j\in mU}
\Big\{
\frac{1}{2}I\Big(1-p_C,\frac{2b}{m}\Big)
+
\frac{2b}{m}I\Big(
\frac{
\pi}{\sigma(1-p_C)},
\frac{j}{2b}\Big)\Big\}
\bigg)
\hfill\cr
\,\leq\,
(m+1)^2\,
\exp\big(4\ln m + 6\big)
\exp\,
\bigg(\,
-m\,
\min_{t\in U}
V_{1}(1,t)
\bigg)\,.
}
$$
Taking $\ln$, dividing by~$m$ and sending $m$ to $\infty$, we obtain
the desired large deviations upper bound.
From now on, we suppose that $0<s<1$.
We have
$$\lim_{m\to+\infty}
F_m(m- \lfloor sm\rfloor
+1)+\cdots+F_m(m)\,=\,1-F(1-s)$$
whence
$$
\varepsilon_m(\lfloor sm\rfloor)
\sim\,\big(1-F(1-s)\big)
(1-p_M)^\ell\quad\text{as}\quad m\to\infty
\,.$$
Let us set
$$\varepsilon(s)\,=\,
\big(1-F(1-s)\big)
(1-p_M)^\ell\,.$$
For any $u\in[0,1]$, we have
$$\Big|
I\big(\varepsilon_m(\lfloor sm\rfloor),u\big)-
I\big(\varepsilon(s),u\big)\Big|
\,\leq\,
\Big|\ln\frac{\varepsilon(s)}
{\varepsilon_m(\lfloor sm\rfloor)
}
\Big|
\,+\,
\Big|\ln\frac{1-\varepsilon(s)}
{
1-\varepsilon_m(\lfloor sm\rfloor)}
\Big|
\,.$$
In order to bound these terms, we suppose in addition that
$0<s<1$.
Since $F$ is strictly increasing
on $[0,1]$ by hypothesis, then
$0<F(1-s)<1$. 
It follows that there exist $\gamma$ and $m_0(s)$
such that 
for $m\geq m_0(s)$,
$$0\,<\,\gamma\,<\,F_m(m- \lfloor sm\rfloor
+1)+\cdots+F_m(m)\,<\,1-\gamma\,<\,1\,.$$
Let us set
$$\Delta(s,m)\,=\,
\big|
1-F(1-s)-
\big(F_m(m- \lfloor sm\rfloor
+1)+\cdots+F_m(m)\big)
\big|\,.$$
Since, for any $a\leq 1$, any $x\in]0,1[$,
$$\Big|
\frac{\partial}{\partial x}\ln(1-xa)
\Big|
\,=\,
\Big|
\frac{a}{1-xa}
\Big|
\,\leq\,
\frac{1}{1-x}\,,$$
we have, for $m\geq m_0(s)$,
$$\displaylines{
\Big|\ln\frac{1-\varepsilon(s)}
{
1-\varepsilon_m(\lfloor sm\rfloor)}
\Big|\,\leq\,
\frac{1}{\gamma}
\Delta(s,m)
\,.}$$
Similarly, we have
$$\displaylines{
\Big|\ln\frac{\varepsilon(s)}
{
\varepsilon_m(\lfloor sm\rfloor)}
\Big|\,\leq\,
\frac{1}{\gamma}
\Delta(s,m)
\,.}$$
These inequalities hold uniformly with respect to
the value of $(1-p_M)^\ell$.
Let now $s\in]0,1[$ and let $U$ be a subset of $[0,1]$. 
Collecting together the previous inequalities, we have,
for any $m\geq m_0(s)$,
$$\displaylines{
P\big(
N_{n+1}\in mU
\,|\,
N_{n}=\lfloor sm\rfloor
\,\big)
\,\leq\,
(m+1)^2\,
\times
\hfill\cr
\exp\,
\bigg(\,
-m\,
\min_{0\leq b\leq m/2}
\min_{j\in mU}
\Big\{
\frac{1}{2}I\Big(1-p_C,\frac{2b}{m}\Big)
+
\frac{2b}{m}I\Big(
\frac{
\big(1-F(1-s)\big)
\pi}{\sigma(1-p_C)},
\frac{j}{2b}\Big)\Big\}
\hfill\cr
\hfill
+4\ln m + 6
+
\frac{2m}{\gamma}\Delta(s,m)
\bigg)\,.
}
$$
We are now in position to replace the discrete variational problem appearing
in this inequality by a continuous one.
Let $V_{1}(s,t)$ be the function defined before
lemma~\ref{lscp}.
The previous inequality implies that,
for any $m\geq m_0(s)$,
$$\displaylines{
P\big(
N_{n+1}\in mU
\,|\,
N_{n}=\lfloor sm\rfloor
\,\big)
\,\leq\,
\hfill\cr
(m+1)^2\,
\exp\,
\bigg(\,
-m\,
\min_{t\in U}
V_{1}(s,t)
+4\ln m + 6
+
\frac{2m}{\gamma}\Delta(s,m)
\bigg)\,.
}
$$
Taking $\ln$, dividing by~$m$ and sending $m$ to $\infty$, we obtain
the desired large deviations upper bound.
\end{proof}

\noindent
Proceeding in the same way, we can prove a similar large deviations
upper bound for the
$l$--step transition
probabilities.
For $l\geq 1$, we define a function 
$V_{l}$ 
on $[0,1]\times [0,1]$ by
\begin{multline*}
V_{l}(s,t)\,=\,
\inf\,\Big\{\,
\smash{\sum_{k=0}^{l-1}}
V_{1}\big(\rho_k,\rho_{k+1}\big):\rho_0=s
,\,\rho_{l}=t,\,\cr
\rho_k\in [0,1]\text{ for }0\leq k<l
\,\Big\}\,.
\end{multline*}
\begin{corollary}\label{ppgd}
For $l\geq 1$, the $l$--step transition probabilities of 
$(N_n)_{n\geq 0}$ satisfy the following 
large deviations upper bound:
for any $s\in]0,1[$, any subset $U$ of $[0,1]$, we have
 $$\limsup_{m\to\infty}\,\frac{1}{m}\ln
P\big(
N_{n+l}\in mU
\,|\,
N_{n}=\lfloor sm\rfloor
\,\big)
\,\leq\,
-\inf_{t\in U}\,
V_{l}(s,t)\,.
$$
\end{corollary}
\subsection{Dynamics of $N_n$}
Let us examine when the rate function~$V_{1}(s,t)$ vanishes.
Let $\pi>1$ and let $s,t\in[0,1]$.
By lemma~\ref{lscp}, 
the variational problem defining $V_{1}(s,t)$ is well
posed, i.e., there exists $p^*,\beta^*\in [0,1]$ such that
$0\leq p^*\leq 1-{\pi}/{\sigma}$,
$t\leq\beta^*\leq 1$ and
$$V_{1}(s,t)\,=\,
\frac{1}{2}I\big(1-p^*,\beta^*\big)
+\beta^*
I\Big(
\frac{
\big(1-F(1-s)\big)
\pi}{\sigma(1-p^*)},
\frac{t}{\beta^*}\Big)\,.
$$
Thus 
$V_{1}(s,t)$ vanishes if and only if
$$\beta^*=1-p^*\,,\qquad
\frac{\big(1-F(1-s)\big)\pi}{\sigma(1-p^*)}=
\frac{t}{\beta^*}\,,
$$
or equivalently
$$t\,=\,
\big(1-F(1-s)\big)
\frac{
\pi}{\sigma}\,.$$
We define a function $\phi:[0,1]\to[0,1]$ by
$$\forall r\in[0,1]\qquad
\phi(r)\,=\,
\big(1-F(1-r)\big)
\frac{
\pi}{\sigma}\,.$$
The Markov chain
$(N_n)_{n\geq 0}$
can be seen as a random perturbation of the dynamical system
associated to the map $\phi$:
$$z_0\in[0,1]\,,\qquad
\forall n\geq 1\quad z_n\,=\,\phi(z_{n-1})\,.$$
Since $\phi$ is non--decreasing, the sequence $(z_n)_{n\in\mathbb N}$ is monotonous
and it converges to a fixed point of $\phi$.
We have supposed 
that $F$ is convex, so that $\phi$ is concave.
Moreover we have $\phi(0)=0$, $\phi(1)=\pi/\sigma<1$
and $\phi'(0)=\pi$, therefore:

\noindent
$\bullet$ 
If $\pi<1$,
then
the function $\phi$ admits only one fixed point, $0$, and
$(z_n)_{n\in\mathbb N}$ converges to $0$;

\noindent
$\bullet$ 
If $\pi> 1$,
the function $\phi$ admits two fixed points, 
$0$ 
and $\rho^*(\pi)$,
and
$(z_n)_{n\in\mathbb N}$ converges to $\rho^*(\pi)$
whenever $z_0>0$.

\noindent
We can even compute $\rho^*(\pi)$ for linear ranking
and tournament selection:

\noindent
{\bf Linear ranking selection.} In this case, we have
$$\rho^*(\pi)\,=\,\frac{2\eta^+}{\eta^+-\eta^-}\Big(1-\frac{1}{\pi}\Big)\,.$$
{\bf Tournament selection.}
The non null fixed point is the solution of 
$$1+\rho^*(\pi)+\cdots+
\rho^*(\pi)^{t-1}
\,=\,\frac{\sigma}{\pi}\,.$$
In the case where $t=2$, we obtain
$$\rho^*(\pi)\,=\,\frac{\sigma}{\pi}-1\,.$$
\noindent
The natural strategy 
to study the Markov chain
$(N_n)_{n\geq 0}$
is to use the 
Freidlin--Wentzell theory \cite{FW}. 
The crucial quantity to analyze the dynamics is the following 
cost function~$V$.
We define, for $s,t\in [0,1]$,
$$\displaylines{
V(s,t)\,=\,
\inf_{l\geq 1}\,V_{l}(s,t)
\,=\,
\hfill\cr
\inf_{l\geq 1}\,
\inf\,\Big\{\,
\sum_{k=0}^{l-1}
V_{1}\big(\rho_k,\rho_{k+1}\big):\rho_0=s
,\,\rho_{l}=t,\,
\rho_k\in [0,1]\text{ for }0\leq k<l
\,\Big\}\,.
}$$
\begin{lemma}\label{vpro}
Suppose that $\pi>1$.
For 
$s\in ]0,1[$, 
$t\in [0,1]$, 
we have $V(s,t)=0$ if and only if

\noindent
$\bullet$ 
either $s=t=0$,

\noindent
$\bullet$ or $s>0$, $t=\rho^*(\pi)$,

\noindent
$\bullet$ 
or there exists $l\geq 1$ such that $t=\phi^l(s)$.
%
\end{lemma}
\begin{proof}
Throughout the proof, we write simply $\rho^*$ instead of $\rho^*(\pi)$.
Let $s,t\in [0,1]$ be such that $V(s,t)=0$. 
For each $n\geq 1$, let
$\smash{(\rho^n_0,\dots,\rho^n_{l(n)})}$ be a sequence 
of length $l(n)$ in $[0,1]$ such that
$$\rho^n_0=s,\,\rho^n_{l(n)}=t,\,\quad
\sum_{k=0}^{l(n)-1}
V_{1}\big(\rho^n_k,\rho^n_{k+1}\big)
\,\leq\,\frac{1}{n}\,.$$
If $s=0$, then necessarily
$\smash{\rho^n_1=\dots=\rho^n_{l(n)}=0}$ and $t=0$.
From now on, we suppose that $s>0$.
We consider two cases.
If the sequence $(l(n))_{n\geq 1}$ is bounded, then
we can extract a subsequence
$$\big(\rho^{\phi(n)}_0,\dots,
\rho^{\phi(n)}_{l({\phi(n)})}\big)$$ 
such that
$l({\phi(n)})=l$ does not depend on $n$ and
for any $k\in \zll$, the following limit exists:
$$
\lim_{n\to\infty}\,
\rho^{\phi(n)}_k
\,=\, \rho_k\,.
$$
The map $V_{1}$ being lower semicontinuous, we have then
$$\forall k\in\zll \qquad 
V_{1}\big(\rho_k,\rho_{k+1}\big)\,=\,0\,,$$
whence
$$\forall k\in\{\,0,\dots,l\,\} \qquad 
\rho_k\,=\,\phi^k(\rho_0)\,.$$
Since in addition
$\rho_0=s$ and $\rho_l=t$, we conclude that
$t=\phi^{l}(s)$. 
Suppose next that the sequence $(l(n))_{n\geq 1}$ is not bounded.
Our goal is to show that $t=\rho^*$.
Using Cantor's diagonal procedure, we can extract a subsequence
$$\big(\rho^{\phi(n)}_0,\dots,
\rho^{\phi(n)}_{l({\phi(n)})}\big)$$ 
such that,
for any $k\geq 0$, the following limit exists:
$$
\lim_{n\to\infty}\,
\rho^{\phi(n)}_k
\,=\, \rho_k\,.
$$
The map $V_{1}$ being lower semicontinuous, we have then
$$\forall k\geq 0\qquad 
V_{1}\big(\rho_k,\rho_{k+1}\big)\,=\,0\,,$$
whence
$$\forall k\geq 0\qquad 
\rho_k\,=\,\phi^k(\rho_0)\,.$$
We have also
$\smash{V_{1}\big(\rho^*,\rho^*\big)}\,=\,0$.
By lemma~\ref{lscp}, there exist
$p^*,\beta^*$ such that
$0\leq {p^*}\leq 1-{\pi}/{\sigma}$,
$\rho^*\leq {\beta^*}\leq 1$ and
$$V_{1}(\rho^*,\rho^*)\,=\,
\frac{1}{2}I\big(1-p^*,\beta^*\big)
+\beta^*
I\Big(
\frac{
\big(1-F(1-\rho^*)\big)
\pi}{\sigma(1-p^*)},
\frac{\rho^*}{\beta^*}\Big)\,.
$$
Since $\rho^*$ is in $]0,1[$, certainly we have 
$\beta^*>0$.
Let $\ve>0$.
The map
$$t\mapsto
\beta^*
I\Big(
\frac{
\big(1-F(1-\rho^*)\big)
\pi}{\sigma(1-p^*)},
\frac{t}{\beta^*}\Big)
$$
is continuous at $\rho^*$, thus
there exists a neighborhood $U$ of
$\rho^*$
such that
$$\forall \rho\in U\qquad
V_{1}
(\rho^*,\rho)\,\leq\,
\,\ve\,.$$
Since $s>0$,
the sequence $(\phi^n(s))_{n\in\mathbb N}$  converges towards
$\rho^*$
and
$\phi^h(s)\in U$ for some $h\geq 1$.
In particular,
$$\lim_{n\to\infty}\,
\rho^{\phi(n)}_h
\,=\, 
\phi^h(s)\,\in\, U
\,,$$
so that, for $n$ large enough,
$\rho^{\phi(n)}_h$ is in $U$
and
$$
V(\rho^*,t)\,\leq\,
V_{1}
\big(\rho^*,
\rho^{\phi(n)}_h
\big)+
V\big(
\rho^{\phi(n)}_h,t
\big)\,\leq\,\ve+
\frac{1}{n}\,.$$
Letting successively $n$ go to $\infty$ and $\ve$ 
go to $0$
we obtain that
$V(\rho^*,t)=0$.
Let $\delta\in\, ]0,\rho^*/2[$ and let 
$U=\,]\rho^*-\delta,\rho^*+\delta[$. 
Let $\alpha$ be the infimum
$$\alpha\,=\,\inf\big\{\,
V_{1}
\big(\rho_0,\rho_{1}\big):
\rho_0\in
\overline{U},
\rho_{1}\not\in U
\,\big\}
\,.$$
Since 
$V_{1}$
is lower semicontinuous on
the compact set
$\overline{U}\times 
\big([0,1]\setminus U\big)$, then
$$\exists (\rho^*_0,\rho^*_{1}\big)\in
\overline{U}\times 
\big([0,1]\setminus U\big)\qquad
\alpha\,=\,
V_{1}
\big(\rho^*_0,\rho^*_{1}\big)\,.$$
The function $\phi$ is non--decreasing and continuous, therefore
$$\phi\big(\overline{U}\big)\,=\,
\phi\big(
[\rho^*-\delta,\rho^*+\delta]
\big)\,=\,
\big[
\phi(\rho^*-\delta),\phi(\rho^*+\delta)\big]\,.$$
Since $\rho^*$ is the unique fixed point of $\phi$ in
$]0,1]$, then 
$\phi(\rho)>\rho$ for $\rho\in]0,\rho^*[$ and
$\phi(\rho)<\rho$ for $\rho\in]\rho^*,1[$.
Therefore we have
$$\rho^*-\delta\,<\,
\phi(\rho^*-\delta)\,\leq\,
\phi(\rho^*+\delta)\,<\,
\rho^*+\delta\,.$$
Thus
$\phi(\overline U)\,\subset\, U$
and
necessarily
$\rho^*_{1}\neq  \phi(\rho^*_0)$ and $\alpha>0$.
It follows that any sequence
$(\rho_0,\dots,\rho_l)$ such that
$$
\rho_0\in U \,,\qquad
\sum_{k=0}^{l-1}
V_1\big(\rho_k,\rho_{k+1}\big)\,<\,\alpha$$
is trapped in $U$. As a consequence, a point $t$
satisfying 
$V(\rho^*,t)=0$ must belong to 
$U=
]\rho^*-\delta,\rho^*+\delta[$.
This is true for 
any $\delta>0$, hence for any neighborhood of $\rho^*$,
thus $t=\rho^*$.
\end{proof}

\subsection{Creation of a quasispecies}
Our goal in this section is to prove a lower bound
for the probability of the creation of a quasispecies around
the current best fit chromosome in the population.
The delicate situation is when there is only one chromosome
in the population which has the best fitness. This chromosome 
might be destroyed or it might invade a positive fraction
of the population.
We will obtain a lower bound on the fixation probability
by estimating the probability that the progeny
of the best fit chromosome grows geometrically.
The key estimate is stated in the next proposition.
\begin{proposition}\label{crqa}
Let $\pi>1$ be fixed.
There exist 
$$\delta_0>0\,,\quad\rho>1\,,\quad c_0>0\,,\quad m_0\geq 1\,,$$
which depend
on $\pi$ only, such that:
for any set of parameters
$\ell,p_C,p_M$ satisfying
$\pi
\,=\,\sigma(1-p_C)(1-p_M)^\ell$,
we have
$$\displaylines{
\forall m\geq m_0\quad
\forall i\in\big\{\,1,\dots,\lfloor\delta_0m\rfloor\,\big\}\quad
P\big(\,
N_{n+1}\leq\rho i\,\big|\,N_n=i\,\big)
\,\leq\,\exp(-c_0i)\,.}$$
\end{proposition}
\begin{proof}
We recall that,
conditionally on $N_n=i$, the law of $N_{n+1}$ is the same
as the law of the random variable
$$\sum_{k=1}^{2B_n}Y_k^i\,,$$
where 
the law of $B_n$ is the binomial law $\cB(m/2,1-p_C)$,
the variables
$Y_k^i$, $k\in\N$, $i\in\um$, are Bernoulli i.i.d. random variables
with parameter
$$\varepsilon_m(i)\,=\,
\big(F_m(m-i+1)+\cdots+F_m(m)\big)
(1-p_M)^\ell\,.$$
Let $\varepsilon>0$ be such that $\pi(1-\varepsilon)^2>1$ and let
$$l(m,\varepsilon)\,=\,
\Big\lfloor\frac{m}{2}(1-p_C)(1-\varepsilon)\Big\rfloor +1+
\frac{m}{4}(1-p_C)\varepsilon
\,.$$
For $m$ large enough, we have
$$l(m,\varepsilon)\,<\,
\frac{m}{2}(1-p_C)
\,.$$
Let $\rho>1$.
We write
\begin{align*}
P\big(\,
N_{n+1}<\rho i\,\big|\,N_n=i\,\big)
&\,=\,
P\Big(\sum_{k=1}^{2B_n}Y_k^i<\rho i\Big)\cr
&\,\leq\,
P\big(B_n\leq l(m,\varepsilon)\big)+
P\Big(\sum_{k=1}^{2l(m,\varepsilon)}Y_k^i<\rho i\Big)
\,.
\end{align*}
We control the first probability with the help of Hoeffding's inequality
(see appendix~\ref{fixed}).
The expected value of $B_n$ is
${m}(1-p_C)/2> l(m,\varepsilon)$, thus
$$P\big(B_n\leq l(m,\varepsilon)\big)
\,\leq\,
\exp\Big(
-\frac{2}{m}\Big(
\frac{m}{2}(1-p_C)-l(m,\varepsilon)\Big)^2\Big)
\,.
$$
Recall that $1-p_C>1/\sigma$. 
For $m$ large enough, we have
$$\frac{m}{2}(1-p_C)-l(m,\varepsilon)\,\geq\,
\frac{m}{2}(1-p_C)\frac{\varepsilon}{2}-1\,
\,\geq\,\frac{m\varepsilon}{4\sigma}-1\,
\,\geq\,\frac{m\varepsilon}{8\sigma}
\,.$$
It follows that
$$
P\big(B_n\leq l(m,\varepsilon)\big)
\,\leq\,
\exp\Big(
-\frac{m}{32}
\frac{
\varepsilon^2}{
\sigma^2}\Big)
\,.
$$
To control the second probability, we decompose
the sum into~$i$ blocks and we use the
Tchebytcheff exponential inequality. Each block
 follows a binomial law, and we bound the
Cram\'er transform of each block by
the Cram\'er transform of a Poisson law having
the same mean.
More precisely, we choose for the block size
$$
b\,=\,\Big\lfloor\frac{\displaystyle 2l(m,\varepsilon)
-\frac{m}{4}(1-p_C)\varepsilon
}{i}+1\Big\rfloor\,,$$
and we define
the sum associated to each block of size~$b$:
$$\forall j\in\{\,1,\dots,i\,\}\qquad
Y_j'\,=\,
\sum_{k=b(j-1)+1}^{bj}
Y_k^i\,.$$
Notice that 
$Y'_1$ follows the binomial law 
with parameters $b,\varepsilon_m(i)$.
We will next estimate from below
the product
$b\varepsilon_m(i)$. 
By the choice of~$b$ and $l$, we have
$$\displaylines{
b\,\geq\,\frac{1}{i}\Big(2l(m,\varepsilon)
-\frac{m}{4}(1-p_C)\varepsilon
\Big)\,,\cr
l(m,\varepsilon)\,\geq\,
\frac{m}{2}(1-p_C)
\Big(1-\frac{\varepsilon}{2}\Big)\,,
}$$
whence
$$b\,\geq\,\frac{m}{i}
(1-p_C)
\big(1-{\varepsilon}
\big)\,.$$
Let $\delta_0>0$ be such that
$$\delta_0\,<\,\frac{1}{4}(1-p_C)\varepsilon\,.$$
Let $m_0\geq 1$ be associated to $\varepsilon$
as in the hypothesis on $F_m$ (see section~\ref{genhy}).
We have, 
for $m\geq m_0$ and 
$i\in\big\{\,1,\dots,\lfloor\delta_0 m\rfloor\,\big\}$,
$$
\varepsilon_m(i)
\,\geq\,
\sigma(1-\varepsilon)
\frac{i}{m}
(1-p_M)^\ell\,
$$
and we conclude from the previous inequalities that
$$
b\varepsilon_m(i)
\,\geq\,
(1-p_C)(1-\varepsilon)^2
\sigma
(1-p_M)^\ell\,=\,\pi
(1-\varepsilon)^2
\,.
$$
We choose $\rho$ such that
$1<\rho<
\pi(1-\varepsilon)^2$, this implies in particular that
$$\rho\,<\,E(Y'_1)
\,=\,b\varepsilon_m(i)\,.$$
We have also that
\begin{align*}
bi\,&\leq\,
2l(m,\varepsilon)
-\frac{m}{4}(1-p_C)\varepsilon+i
\cr
&\,\leq\,
2l(m,\varepsilon)
-\frac{m}{4}(1-p_C)\varepsilon+\delta_0 m
\,\leq\,
2l(m,\varepsilon)\,.
\end{align*}
We have then, using Tchebytcheff exponential inequality
(see appendix~\ref{fixed}):
$$\displaylines{
P\Big(\sum_{k=1}^{2l(m,\varepsilon)}Y_k^i\leq\rho i\Big)
\,\leq\,
P\Big(\sum_{k=1}^{bi}Y_k^i\leq\rho i\Big)\hfill
\cr
\,\leq\,
P\Big(
\sum_{j=1}^{i}
Y_j' \leq\rho i\Big)
\,\leq\,
P\Big(
\sum_{j=1}^{i}
-Y_j' \geq-\rho i\Big)\,\leq\,
\exp\Big(-i\Lambda^*_{-Y'_1}(-\rho)\Big)
\,,
}$$
where
$\smash{\Lambda^*_{-Y'_1}}$ is the Cram\'er transform of $-Y'_1$.
Let $Y''_1$ be a random variable following the Poisson law of parameter
$b\varepsilon_m(i)$.
By lemma~\ref{bnpp}, we have
$$\Lambda^*_{-Y'_1}(-\rho)\,\geq\,
\Lambda^*_{-Y''_1}(-\rho)\,=\,
\rho\ln\Big(\frac{\rho}{b\varepsilon_m(i)}\Big)-\rho+
{b\varepsilon_m(i)}\,.$$
The map
$$\lambda\mapsto
\rho\ln\Big(\frac{\rho}{\lambda}\Big)-\rho+
{\lambda}\,$$
is non--decreasing on $[\rho,+\infty[$, thus
$$\Lambda^*_{-Y''_1}(-\rho)\,\geq\,
\rho\ln\Big(\frac{\rho}{
\pi(1-\varepsilon)^2
}\Big)-\rho+
{\pi(1-\varepsilon)^2}
\,.$$
Let us denote by $c_0$ the righthand quantity. Then $c_0$
is positive and it depends only on $\rho,\pi$ and $\varepsilon$.
Finally, we have
for $m\geq m_0$, 
$i\in\big\{\,1,\dots,\lfloor\delta_0 m\rfloor\,\big\}$,
$$P\Big(\sum_{k=1}^{2l(m,\varepsilon)}Y_k^i\leq\rho i\Big)
\,\leq\,\exp(-c_0i)\,$$
whence
$$
P\big(\,
N_{n+1}\leq\rho i\,\big|\,N_n=i\,\big)
\,\leq\,
\exp\Big( -\frac{m}{32}
\frac{ \varepsilon^2}{ \sigma^2}\Big)
\,+\,\exp(-c_0i)
\,.
$$
Let $\eta\in]0,1[$ be small enough so that
$$\exists\,m_1\quad\forall m\geq m_1
\qquad
\exp\Big( -\frac{m}{32}
\frac{ \varepsilon^2}{ \sigma^2}\Big)
\,\leq\,
\exp\Big( -\eta\frac{mc_0}{2}\Big)
\Big(1-
\exp\Big( -\eta\frac{c_0}{2}\Big)\Big)\,.$$
For $m\geq\max(m_0,m_1)$ and
$i\in\big\{\,1,\dots,\lfloor\delta_0 m\rfloor\,\big\}$,
we have
\begin{multline*}
P\big(\,
N_{n+1}\leq\rho i\,\big|\,N_n=i\,\big)\cr
\,\leq\,
\exp\Big( -\eta\frac{ic_0}{2}\Big)
\Big(1-
\exp\Big( -\eta\frac{c_0}{2}\Big)\Big)
+\exp\big( -\eta{ic_0}\big)\cr
\,\leq\,
\exp\big( -\eta\frac{ic_0}{2}\big)
\end{multline*}
and this inequality yields the claim of the proposition.
\end{proof}

\noindent
We define
$$\tau_0\,=\,\inf\,\big\{\,n\geq 1: N_n=0\,\big\}\,.
$$
For $\delta>0$, 
let $T({\delta})$ be the first time the process
$(N_n)_{n\geq 0}$
becomes larger than~$\delta m$:
$$T({\delta})\,=\,\inf\,\{\,n\geq 0: N_n\geq\delta m\,\}\,.$$
\begin{proposition}\label{sfj}
There exist $\delta_0>0$ and $p_0>0$ which depend only on $\pi$
such that
$$
\forall m\geq 1\qquad
P\big(T({\delta_0})\leq\kappa \ln m,\,\tau_0>T({\delta_0})
\,|\,N_0=1\big)\,\geq\,
p_0\,.$$
\end{proposition}
\begin{proof}
Let $T_k$ be the first time the process
$(N_n)_{n\geq 0}$
hits~$k$:
$$T_k\,=\,\inf\,\{\,n\geq 0: N_n=k\,\}\,.$$
Let $\delta_0$, $\rho>1$, $c_0$, $m_0$
be as given in
proposition~\ref{crqa}.
We suppose that 
the process $(N_n)_{n\geq 0}$
starts from $N_0=1$.
Let $\cE$ be the event:
$$\cE\,=\,\big\{\,
\forall k\in\big\{\,1,\dots,\lfloor\delta_0 m\rfloor\,\big\}\quad
N_{T_k+1}>\rho N_{T_k}\,\big\}\,.$$
We claim that,
on the event~$\cE$, we have
$$\forall n\leq T({\delta_0})\qquad
N_{n+1}>\rho N_{n}\,.$$
Let us prove this inequality by induction on~$n$.
We have $T_0=0$, so that
$N_{1}>\rho N_{0}$
 and the inequality is true
for $n=0$. 
Suppose that the inequality has been proved until rank 
$n< T({\delta_0})$, so that
$$\forall k\leq n\qquad
N_{k+1}>\rho N_{k}\,.$$
This implies in particular that
$$N_0\,<\,N_1\,<\,\dots\,<\,N_n\,<\,m\delta_0\,.$$
Suppose that $N_n=i$. The above inequality implies that
$T_i=n$.
Therefore
$$N_{T_i+1}=N_{n+1}\,>\,\rho N_n$$
and the inequality still holds at rank $n+1$.
Iterating the inequality until time $T({\delta_0})-1$,
we see that
$$N_{T({\delta_0})-1}\,>\,\rho^{T({\delta_0})-1}\,.$$
Moreover
$N_{T({\delta_0})-1}\,\leq\,m\delta_0$, thus
$$T({\delta_0})\,\leq\,1+\frac{\ln(m\delta_0)}{\ln\rho}\,.$$
Let $m_1\geq 1$ and $\kappa>0$ such that
$$\forall m\geq m_1\qquad
1+\frac{\ln(m\delta_0)}{\ln\rho}\,\leq\,\kappa\ln m\,.$$
The constants $m_1,\kappa$ depend only on $\delta_0$ and $\rho$, and
we have 
$$P\big(T({\delta_0})\leq\kappa \ln m,\,\tau_0>T({\delta_0})
\,|\,N_0=1\big)\,\geq\,
P(\cE)\,.$$
By lemma~\ref{indepday}, the random variables
$N_{T_k+1}$, $k\leq\delta_0 m$, are independent. 
To be precise, we cannot directly apply 
lemma~\ref{indepday}, because the
Markov chain
$(N_n)_{n\geq 0}$
has an absorbing state at~$0$ and therefore it is not irreducible.
So we consider the modified
Markov chain
$(\smash{\widetilde{N}}_n)_{n\geq 0}$ which has the same transition
probabilities as
$(N_n)_{n\geq 0}$, except that we set the transition probability from $0$
to $1$ to be $1$. The event we wish to estimate in the lemma
has the same probability for both processes. Indeed, we require that
$T({\delta_0})\leq\kappa \ln m$ and $\tau_0>T({\delta_0})$, 
so that
the processes do not visit $0$ before
$T({\delta_0})$.
Using proposition~\ref{crqa}, we obtain, for $m$ larger than
$m_0$ and $m_1$,
\begin{align*}
P(\cE)\,&\geq\,
\prod_{k=1}^{\lfloor \delta_0 m\rfloor}
P\big(N_{T_k+1}>\rho N_{T_k})\cr
&\,=\,
\prod_{k=1}^{\lfloor \delta_0 m\rfloor}
\Big(1-P\big(\,
N_{1}\leq\rho k\,\big|\,N_0=k\,\big)\Big)\cr
&\,\geq\,
\prod_{k=1}^{\lfloor \delta_0 m\rfloor}
\Big(1- \exp(-c_0k) \Big)
\,\geq\,
\prod_{k=1}^{\infty}
\Big(1- \exp(-c_0k) \Big)
\,.
\end{align*}
The last infinite product is converging. 
Let us denote
its value by $p_1$. 
Let also
$$p_2\,=\,
\min\,\Big\{\,
P\big(T({\delta_0})\leq\kappa \ln m,\,\tau_0>T({\delta_0})\,|\,N_0=1\big)
:m\leq \max(m_0,m_1)\,\Big\}
.$$
The value $p_2$ is positive
and the inequality stated in the proposition
holds with $p_0=\min(p_1,p_2)$.
\end{proof}
\begin{lemma}
\label{lbd}
Let $\pi>1$ be fixed.
For any $\delta>0$, 
there exist $h\geq 1$, $c>0$, $m_0\geq 1$, which
depend only on $\delta$ and $\pi$,
such that:
for any set of parameters
$\ell,p_C,p_M$ satisfying
$\pi
\,=\,\sigma(1-p_C)(1-p_M)^\ell$,
we have,
for any $m\geq m_0$,
$$
P\big(N_1>0,
\dots,N_{h-1}>0,\,
N_h
>m(\rho^*-\delta)
\,|\,N_0=
{\lfloor m\delta\rfloor}
\big)\,\geq\,
1-\exp(-cm)
\,.$$
\end{lemma}
\begin{proof}
Let $\delta>0$. 
The sequence $(\phi^n(\delta))_{n\in\mathbb N}$  converges to $\rho^*$,
thus
there exists $h\geq 1$ such that
$\phi^h(\delta)>\rho^*-\delta$.
By continuity of the map $\phi$, there exist
$\rho_0,\rho_1,\dots,\rho_h>0$ such that 
$\rho_0=\delta$, 
$\rho_h>\rho^*-\delta$ and
$$\forall k\in\{\,1,\dots,h\,\}\qquad
\phi(\rho_{k-1})\,>\,\rho_k\,.$$
Now, 
\begin{multline*}
P\big(N_1>0,
\dots,N_{h-1}>0,\,
N_h
>m(\rho^*-\delta)
\,|\,N_0=
{\lfloor m\delta\rfloor}
\big)\,\geq\,\cr
P\Big(
\forall k\in\{\,1,\dots,h\,\}\quad
N_k\geq 
m\rho_{k}
\,|\,N_0=
{\lfloor m\delta\rfloor}
\Big)
\,.
\end{multline*}
Passing to the complementary event, we have
$$\displaylines{
P\big(\exists k\in\{\,1,\dots,h-1\,\}\quad
N_k=0
\quad
\text{or}
\quad
N_h
\leq m(\rho^*-\delta)
\,|\,N_0=
{\lfloor m\delta\rfloor}
\big)
\cr
\,\leq\,
P\Big(
\exists k\in\{\,1,\dots,h\,\}\quad
N_k< 
m\rho_{k}
\,|\,N_0=
{\lfloor m\delta\rfloor}
\Big)\cr
\,\leq\,
\sum_{1\leq k\leq h}
P\Big(
N_{1} \geq
m\rho_{1},\dots,
N_{k-1} \geq
m\rho_{k-1},\,
N_k< 
m\rho_{k}
\,|\,N_0=
{\lfloor m\delta\rfloor}
\Big)\cr
\,\leq\,
\sum_{1\leq k\leq h}
\sum_{i\geq
m\rho_{k-1}
}
P\Big(
N_{k-1} =i,\,
N_k< 
m\rho_{k}
\,|\,N_0=
{\lfloor m\delta\rfloor}
\Big)\cr
\,\leq\,
\sum_{1\leq k\leq h}
\sum_{i\geq
m\rho_{k-1}
}
P\Big(
N_k< 
m\rho_{k}
\,|\,
N_{k-1} =i
\Big)\,
P\Big(
N_{k-1} =i
\,|\,N_0=
{\lfloor m\delta\rfloor}
\Big)\,
\cr
\,\leq\,
\sum_{1\leq k\leq h}
P\Big(
N_1< 
m\rho_{k}
\,|\,
N_{0} =
{\lfloor 
m\rho_{k-1}
\rfloor}
\Big)
\,.
}$$
The large deviations upper bound for the transition probabilities of 
the Markov chain
$(N_n)_{n\geq 0}$ 
stated in
proposition~\ref{ldup}
implies that
\begin{multline*}
\forall k\in\{\,1,\dots,h\,\}\hfill\cr
 \limsup_{m\to\infty}\,\frac{1}{m}\ln
P\Big(
N_1< 
m\rho_{k}
\,|\,
N_{0} =
{\lfloor 
m\rho_{k-1}
\rfloor}
\Big)\cr
\,\leq\,
-\inf\,\Big\{\,
V_1\big(\rho_{k-1},t
 \big)
:t\leq
\rho_{k}
\,\Big\}
\,<\,0\,.
\end{multline*}
Since $h$ is fixed, we conclude that
$$
 \limsup_{m\to\infty}\,\frac{1}{m}\ln
P\bigg(
\begin{matrix}
\exists k\in\{\,1,\dots,h-1\,\}\quad
N_k=0\\
\text{or}
\quad
N_h
\leq m(\rho^*-\delta)
\end{matrix}
\,\Big|\,N_0=
{\lfloor m\delta\rfloor}
\bigg)\,<\,0
$$
and this yields the desired estimate.
\end{proof}

\noindent
With the estimate of lemma~\ref{lbd}, we show that the process
is very unlikely to stay a long time in
$[m\delta,
m(\rho^*-\delta)]$.
\begin{corollary}
\label{cbd}
Let $\pi>1$ be fixed.
For any $\delta>0$, 
there exist $h\geq 1$, $c>0$, $m_0\geq 1$, which
depend only on $\delta$ and $\pi$,
such that:
for any set of parameters
$\ell,p_C,p_M$ satisfying
$\pi
\,=\,\sigma(1-p_C)(1-p_M)^\ell$,
we have,
for any $m\geq m_0$,
$$\displaylines{
\forall k\in [m\delta,
m(\rho^*-\delta)]\quad\forall n\geq 0\qquad\hfill\cr
P\Big(
m\delta\leq
N_t\leq
m(\rho^*-\delta)
\text{ for }0\leq t\leq n
\,|\,N_0=k
\Big)\,\leq\,
\exp\Big(-cm\Big\lfloor
\frac{n}{h}
\Big\rfloor\Big)
\,.
}$$
\end{corollary}
\begin{proof}
Let 
$k\in [m\delta,
m(\rho^*-\delta)]$.
Let $\delta>0$ and 
let $h\geq 1$ and $c>0$ be associated
to $\delta$ as in
lemma~\ref{lbd}. 
We divide the interval $\{\,0,\dots,n\,\}$ into subintervals of length $h$ and
we use repeatedly the estimate of lemma~\ref{lbd}. 
Let $i\geq 0$. We write
$$\displaylines{
P\big( m\delta\leq  N_t\leq m(\rho^*-\delta)
\text{ for }0\leq t\leq (i+1)h
\,|\,N_0=k \big)
\, =\,
\hfill
\cr
\sum_{\delta m\leq j
\leq (\rho^*-\delta)m
}
\kern-7pt
P\big( m\delta\leq  N_t\leq m(\rho^*-\delta)
\text{ for }0\leq t\leq (i+1)h,\,N_{ih}=j
\,|\,N_0=k \big)
\cr
\,=\,\sum_{\delta m\leq j
\leq (\rho^*-\delta)m
}
\kern-7pt
P\big( m\delta\leq  N_t\leq m(\rho^*-\delta)
\text{ for }0\leq t\leq ih,\,N_{ih}=j
\,|\,N_0=k \big)
\cr
\phantom{\sum_{\delta m\leq j
\leq (\rho^*-\delta)m
}}
\hfill\times
P\big( m\delta\leq  N_t\leq m(\rho^*-\delta)
\text{ for }ih\leq t\leq (i+1)h
\,|\,
N_{ih}=j
\big)\cr
\,\leq\,
\sum_{\delta m\leq j
\leq (\rho^*-\delta)m
}
\kern-7pt
P\big( m\delta\leq  N_t\leq m(\rho^*-\delta)
\text{ for }0\leq t\leq ih,\,N_{ih}=j
\,|\,N_0=k \big)
\cr
\phantom{\sum_{\delta m\leq j
\leq (\rho^*-\delta)m
}}
\hfill\times
P\big(N_h
\leq m(\rho^*-\delta)
\,|\,N_0=
{\lfloor m\delta\rfloor}
\big)\cr
\,\leq\,
P\big( m\delta\leq  N_t\leq m(\rho^*-\delta)
\text{ for }0\leq t\leq ih
\,|\,N_0=k \big)
\exp(-cm)
\,.
}$$
Iterating this inequality, we obtain
$$\forall i\geq 0\qquad\!
P\big( m\delta\leq  N_t\leq m(\rho^*-\delta)
\text{ for }0\leq t\leq ih
\,|\,N_0=k \big)
\,\leq\,
\exp(-cmi)
\,.
$$
The claim of the corollary follows by applying this inequality with
$i$ equal to the integer part of $n/h$.
\end{proof}

\subsection{The catastrophe}
We have computed the relevant estimates to reach the neighborhood
of $\rho^*$. Our next goal is to study the hitting 
time $\tau_0$ starting from a neighborhood of $\rho^*$.
Since we need only a lower bound, we shall study the hitting time
of a neighborhood of $0$.
For $\delta>0$, we define
$$\tau_\delta\,=\,\inf\,\big\{\,n\geq 0: N_{n}<m\delta\,\big\}\,.$$
\begin{proposition}\label{cghy}
Let $\pi>1$ be fixed.
For any $\delta>0$, 
there exists $m_0\geq 1$, which
depend only on $\delta$ and $\pi$,
such that:
for any set of parameters
$\ell,p_C,p_M$ satisfying
$\pi
\,=\,\sigma(1-p_C)(1-p_M)^\ell$,
we have
$$\displaylines{
\forall m\geq m_0\qquad
\forall i \geq \lfloor(\rho^*-\delta)m\rfloor\qquad
\forall n\geq 1\hfill\cr
P\big(\tde\leq n
\,|\,N_0=i
\big)
\,\leq\, n
\exp\big(-mV(\rho^*-\delta,\delta)+m\delta\big)
\,.}$$
\end{proposition}
\begin{proof}
Let $i \geq\lfloor (\rho^*-\delta)m\rfloor$. 
The strategy consists in looking 
at the portion of the
trajectory starting at the last visit to the neighborhood of
$\rho^*$ before reaching the neighborhood of~$0$.
Accordingly,
we define
$$S\,=\,\max\,\big\{\,n\leq \tau_\delta: N_{n}> (\rho^*-\delta)m\,\big\}\,.$$
Notice that $S$ is not a Markov time.
We write, for $n,k\geq 1$,
$$\displaylines{
P\big(\tde\leq n
\,|\,N_0=i
\big)\,=\,
\sum_{1\leq s<t\leq n}
P\big(
\tde=t,\,S=s
\,|\,
N_{0} =i\,\big)\cr
\,=\,
\sum_{
\genfrac{}{}{0pt}{1}
{1\leq s<t\leq n}
{s<t\leq s+k }}
P\big(
\tde=t,\,S=s
\,|\,
N_{0} =i\,\big)+
\!
\sum_{
\genfrac{}{}{0pt}{1}
{1\leq s< n}
{s+k<t\leq n }}
\!
P\big(
\tde=t,\,S=s
\,|\,
N_{0} =i\,\big)
\,.
}$$
Let $h\geq 1$ and $c>0$ be associated
to $\delta$ as in
corollary~\ref{cbd}. 
For
${1\leq s< n}$ and $t>s+k$,
\begin{multline*}
P\big(
\tde=t,\,S=s
\,|\,
N_{0} =i\,\big)\cr
=
\sum_{m\delta\leq j
\leq (\rho^*-\delta)m
}
P\big(
\tde=t,\,S=s,\,N_{s+1}=j
\,|\,
N_{0} =i\,\big)
\cr
\leq\sum_{m\delta\leq j
\leq (\rho^*-\delta)m
}
P\Big(
\begin{matrix}
\delta m\leq
N_r\leq
(\rho^*-\delta)m\\
 \text{ for }s+1\leq r\leq t-1
\end{matrix}
\,\bigg|\,
N_{s+1}=j
\,\Big)
\cr
\,\leq\,m
\exp\Big(-cm\Big\lfloor
\frac{t-s-2}{h}
\Big\rfloor\Big)
\,,
\end{multline*}
whence
$$\sum_{
\genfrac{}{}{0pt}{1}
{1\leq s< n}
{s+k<t\leq n }}
P\big(
\tde=t,\,S=s
\,|\,
N_{0} =i\,\big)
\,\leq\,
n 
\sum_{t\geq k}
m\exp\Big(-cm\Big\lfloor
\frac{t-1}{h}
\Big\rfloor\Big)
\,.
$$
For
${1\leq s<t\leq n}$
 and $t\leq s+k$,
\begin{multline*}
P\big(
\tde=t,\,S=s
\,|\,
N_{0} =i\,\big)
\cr
\leq
\sum_{j
> (\rho^*-\delta)m
}
P\big(
\tde=t,\,S=s,\,N_s=j
\,|\,
N_{0} =i\,\big)
\cr
\leq
\sum_{j
> (\rho^*-\delta)m
}
P\big(
N_t<\delta m
\,|\,
N_{s} =j\,\big)
\cr
\,\leq\, m
P\big(
N_{t-s}<\delta m
\,|\,
N_{0} =
{\lfloor 
(\rho^*-\delta)m
\rfloor}
\,\big)
\,,
\end{multline*}
whence
$$\displaylines{
\smash{\sum_{
\genfrac{}{}{0pt}{1}
{1\leq s< n}
{s<t\leq s+k }}
}
P\big(
\tde=t,\,S=s
\,|\,
N_{0} =i\,\big)
\,\leq\, 
\hfill\cr
\hfill
n\sum_{1\leq t\leq k}
m
P\big(
N_{t}<\delta m
\,|\,
N_{0} =
{\lfloor 
(\rho^*-\delta)m
\rfloor}
\,\big)
\,.
}$$
Putting together the previous inequalities, we obtain
\begin{multline*}
P\big(\tde\leq n
\,|\,N_0=i
\big)
\,\leq\, 
n 
\sum_{t\geq k}
m\exp\Big(-cm\Big\lfloor
\frac{t-1}{h}
\Big\rfloor\Big)
\cr\hfill
+
n\sum_{1\leq t\leq k}
m
P\big(
N_{t}<\delta m
\,|\,
N_{0} =
{\lfloor 
(\rho^*-\delta)m
\rfloor}
\,\big)
\,.
\end{multline*}
We choose $k$ large enough so that
$$ \limsup_{m\to\infty}\,\frac{1}{m}\ln
\bigg(
\sum_{t\geq k}
m\exp\Big(-cm\Big\lfloor
\frac{t-1}{h}
\Big\rfloor\Big)\bigg)
\,<\,-V(\rho^*-\delta,\delta)
\,,$$
and we use the large deviations upper bound stated in 
corollary~\ref{ppgd}
to estimate the second sum: 
$$\displaylines{
 \limsup_{m\to\infty}\,\frac{1}{m}\ln
\bigg(
\sum_{1\leq t\leq k}
m
P\big(
N_{t}<\delta m
\,|\,
N_{0} =
{\lfloor 
(\rho^*-\delta)m
\rfloor}
\,\big)
\bigg)\cr
\hfill
\,\leq\,-
\min_{1\leq t\leq k}
V_t(\rho^*-\delta,\delta)
\,\leq\,-V(\rho^*-\delta,\delta)
\,.}$$
Therefore there exists $m_0\geq 1$ such that,
$$\forall m\geq m_0\qquad
P\big(\tde\leq n
\,|\,N_0=i
\big)
\,\leq\, 
n\exp\big(-mV(\rho^*-\delta,\delta)+m\delta\big)
\,.$$
This proves the proposition.
\end{proof}
\begin{lemma}\label{appov}
Let $V^*<V(\rho^*,0)$. There exists $\delta>0$ such that
$$V(\rho^*-\delta,\delta)-2\delta
\,\geq\,V^*\,.$$
\end{lemma}
\begin{proof}
Let $V^*<V(\rho^*,0)$. Let $\varepsilon>0$ be such that
$V(\rho^*,0)-4\varepsilon>V^*$.
For $\delta>0$, 
we have
$$V(\rho^*,0)\,\leq\,
V(\rho^*,\rho^*-\delta)+
V(\rho^*-\delta,\delta)+V(\delta,0)\,.$$
We bound next
$V(\rho^*,\rho^*-\delta)$ and $V(\delta,0)$:
$$\displaylines{
V(\rho^*,\rho^*-\delta)\,\leq\,
I\Big(
\frac{
\big(1-F(1-\rho^*)\big)
\pi}{\sigma},
{\rho^*-\delta}\Big)\,=\,
I\big(
{
\rho^*
},
{\rho^*-\delta}\big)
\,,\cr
V(\delta,0)\,\leq\,
I\Big(
\frac{
\big(1-F(1-\delta)\big)
}{\sigma},0\Big)\,\leq\,
-\ln\Big(1-
\frac{
\big(1-F(1-\delta)\big)
}{\sigma}\Big)\,,
}$$
and the righthand terms go to~$0$ when~$\delta$ goes to~$0$.
Thus we can choose $\delta>0$ such that
$$\delta\,<\,\varepsilon\,,\qquad
V(\delta,0)\,<\,\varepsilon\,,\qquad
V(\rho^*,\rho^*-\delta)\,<\,\varepsilon\,.$$
We have then
$$V(\rho^*-\delta,\delta)-2\delta\,\geq\,
V(\rho^*,0)
-2\delta-2\varepsilon\,\geq\,V^*\,$$
and the lemma is proved.
\end{proof}
\begin{corollary}\label{ghy}
For any $V^*<V(\rho^*,0)$, there exist 
$\delta>0$ and
$m_0\geq 1$ such that
$$\displaylines{
\forall m\geq m_0\qquad
P\big(\tde>
\exp(mV^*)
\,|\,N_0=
\lfloor(\rho^*-\delta)m\rfloor
\big)
\,\geq\, 1-\exp(-m\delta)\,.}$$
\end{corollary}
\begin{proof}
Let $\delta>0$ be associated
to~$V^*$ as in lemma~\ref{appov}. 
We apply 
proposition~\ref{cghy} with $\delta$ and
$n=
\exp(mV^*)$: there exists $m_0\geq 1$ such that
$$\displaylines{
\forall m\geq m_0\qquad
P\big(\tde\leq
\exp(mV^*)
\,|\,N_0=
\lfloor(\rho^*-\delta)m\rfloor
\big)
\,\leq\, \exp(-m\delta)\,.}$$
This is the desired inequality.
\end{proof}

\noindent
For $\delta>0$, 
let $T({\rho^*-\delta})$ be the first time the process
$(N_n)_{n\geq 0}$
becomes larger than~$(\rho^*-\delta) m$:
$$T({\rho^*-\delta})\,=\,\inf\,\{\,n\geq 0: N_n\geq
(\rho^*-\delta) m\,\}\,.$$
\begin{proposition}\label{hfaj}
Let $\pi>1$ be fixed.
For any $\delta>0$, there exist $\kappa>0$ and $p_1>0$, 
which depend only on $\pi$ and $\delta$, such that:
for any set of parameters
$\ell,p_C,p_M$ satisfying
$\pi
\,=\,\sigma(1-p_C)(1-p_M)^\ell$,
we have
$$\forall m\geq 1\qquad
P\big(T({\rho^*-\delta})\leq\kappa\ln m\,\big|\,N_0=1\big)\,\geq\,
p_1\,.$$
\end{proposition}
\begin{proof}
Let $\kappa,\delta_0$
be given by proposition~\ref{sfj}.
Let $\delta>0$ 
be associated to $V^*$ as in corollary~\ref{ghy}.
We suppose in addition that $\delta<\delta_0$.
We have 
\begin{multline*}
P\big(
T({\rho^*-\delta})\leq 2\kappa\ln m
\,\big|\,N_0=1\big)
\hfill\cr
\,\geq\,\sum_{k=1}^{\kappa\ln m}
\sum_{i\geq m\delta}
P\big(
T({\rho^*-\delta})\leq 2\kappa\ln m
,\,T(\delta)=k,\,
N_k=i\,\big|\,N_0=1\big)\cr
\,=\,\sum_{k=1}^{\kappa\ln m}
\sum_{i\geq m\delta}
P\big(
T({\rho^*-\delta})\leq 2\kappa\ln m
\,\big|\,T(\delta)=k,\,
N_k=i\big)\hfill\cr
\hfill\times
P\big(T(\delta)=k,\,
N_k=i
\,\big|\,N_0=1\big)\cr
\,=\,\sum_{k=1}^{\kappa\ln m}
\sum_{i\geq m\delta}
P\big(
T({\rho^*-\delta})\leq 2\kappa\ln m-k
\,\big|\,
N_0=i\big)\hfill\cr
\hfill\times
P\big(T(\delta)=k,\,
N_k=i
\,\big|\,N_0=1\big)\cr
\,\geq\,
P\big(
T({\rho^*-\delta})\leq\kappa\ln m
\,\big|\,
N_0=\lfloor m\delta\rfloor \big)
\,
P\big(
T(\delta)\leq\kappa\ln m\,\big|\,N_0=1\big)\,.
\end{multline*}
Let $h\geq 1$, $c>0$ as in
lemma~\ref{lbd}. We suppose that $m$ is large enough so that
$\kappa\ln m\geq h$.
Using again the Markov property, we obtain 
\begin{multline*}
P\big(
T({\rho^*-\delta})\leq\kappa\ln m
\,\big|\,
N_0=\lfloor m\delta\rfloor \big)
\,\geq\,\hfill\cr
P\big(
N_1>0,\dots,N_{h-1}>0,N_h>m(\rho^*-\delta)\,\big|\,
N_0=\lfloor m\delta\rfloor \big)\cr
\,\geq\,
1-\exp(-cm)\,.
\end{multline*}
Putting together the previous inequalities, and using the inequality
of proposition~\ref{sfj}, we conclude that
for $m$ large enough,
$$P\big(
T({\rho^*-\delta})\leq 2\kappa\ln m 
\,\big|\,N_0=1\big)\,\geq\,
\big(1-\exp(-cm)\big)p_0\,.
$$
This implies the result stated in the proposition.
\end{proof}
\begin{corollary}\label{faj}
Let $\pi>1$ be fixed.
For any $V^*<V(\rho^*,0)$, there exists $p^*>0$, 
which
depends only on $V^*$ and $\pi$,
such that:
for any set of parameters
$\ell,p_C,p_M$ satisfying
$\pi
\,=\,\sigma(1-p_C)(1-p_M)^\ell$,
we have
$$\forall m\geq 1\qquad
P\big(\tau_0\geq\exp(mV^*)\,\big|\,N_0=1\big)\,\geq\,
p^*\,.$$
\end{corollary}
\begin{proof}
Let $V^*<V(\rho^*,0)$. 
Let $\delta>0$ and $m_0\geq 1$
be associated to $V^*$ as in 
corollary~\ref{ghy}.
Let $\kappa>0$ and $p_1>0$ 
be associated to $\delta$ as in 
{proposition}~\ref{hfaj}.
We write
$$\displaylines{
P\big(\tau_0\geq\exp(mV^*)\,\big|\,N_0=1\big)\,\geq\,
\hfill\cr
P\big(
\tau_0\geq\exp(mV^*),\,
T({\rho^*-\delta})\leq\kappa \ln m
\,|\,N_0=1\big)
\,\geq\,
\cr
\sum_{i>(\rho^*-\delta)m}
P\big(
\tau_0\geq\exp(mV^*),\,
T({\rho^*-\delta})\leq\kappa \ln m,\,
X_{T({\rho^*-\delta})}=i
\,|\,N_0=1\big)\cr
\,\geq\,
\sum_{i>(\rho^*-\delta)m}
P\big(
\tau_0\geq\exp(mV^*)\,\big|\,
X_{T({\rho^*-\delta})}=i,\,
\tau_0> {T({\rho^*-\delta})}
\big)
\hfill\cr
\hfill
P\big(
X_{T({\rho^*-\delta})}=i,\,
\tau_0> {T({\rho^*-\delta})},\,
T({\rho^*-\delta})\leq\kappa \ln m
\,|\,N_0=1\big)\cr
\,\geq\,
P\big(\tau_0\geq\exp(mV^*)\,\big|\,N_0=\lfloor (\rho^*-\delta)m\rfloor
\big)\hfill\cr
\hfill
\times P\big(
\tau_0> {T({\rho^*-\delta})},\,
T({\rho^*-\delta})\leq\kappa \ln m
\,|\,N_0=1\big)\,.
}$$
Since $\tau_0\geq\tau_\delta$, we have 
by corollary~\ref{ghy} that for 
$m\geq m_0$,
$$\displaylines{
P\big(\tau_0\geq
\exp(mV^*)
\,|\,N_0=
\lfloor(\rho^*-\delta)m\rfloor
\big)
\,\geq\, 1-\exp(-m\delta)\,,}$$
whence
$$P\big(\tau_0\geq\exp(mV^*)\,\big|\,N_0=1\big)\,\geq\,
(1-\exp(-m\delta))\,p_1\,.$$
For $m<m_0$, the above probability is still positive,
and we obtain the desired conclusion.
\end{proof}
\section{Proof of theorems~\ref{thqsr}, \ref{thcsr}
\ref{htopt}, 
\ref{ctopt},
\ref{invpt}}
\label{finproof}
\subsection{Proof of theorem~\ref{thqsr}}
Let $\lambda_0$ be as in
theorem~\ref{thqsr}.
By proposition~\ref{couplnl}, we have 
$$\forall n\geq 0\qquad N(X_n,\lambda_0)\,\geq\,
N_{n}\big(0,1)\,.$$
Let $V^*>0$ be such that
$V^*<V(\rho^*,0)$.
By corollary~\ref{faj},
there exists $p^*>0$, 
which depends on $\pi$ and $V^*$ only, such that
$$\forall m\geq 1\qquad
P\big(\forall n\leq\exp(mV^*)\quad
N_n(0,1)\geq 1
\big)\,\geq\,
p^*\,.$$
This yields the conclusion of 
theorem~\ref{thqsr}.
\subsection{Proof of theorem~\ref{thcsr}}
We apply proposition~\ref{couplnl} with $n=t$, starting time~$s$,
$\lambda=
\Lambda\big(X_s,\lfloor\rho^*m\rfloor\big)$.
By definition of $\Lambda$, 
$$N\big(X_s,\Lambda\big(X_s,\lfloor\rho^*m\rfloor\big)\big)
\,\geq\,
\lfloor\rho^*m\rfloor\,,$$
therefore
$$\forall t\geq s\qquad
N\big(X_t,
\Lambda\big(X_s,\lfloor\rho^*m\rfloor\big)\big)\,\geq\,
N_t\big(s,
\lfloor\rho^*m\rfloor\big)\,.$$
If for some time $t> s$, we have
$$\max_{1\leq i\leq m}f\big(X_t(i)\big)
\,<\,
\Lambda\big(X_s,\lfloor\rho^*m\rfloor\big)
\,,$$
then
$$N\big(X_t,
\Lambda\big(X_s,\lfloor\rho^*m\rfloor\big)\big)\,=\,0\,,$$
and from the previous inequality, we conclude that
$N_t\big(s, \lfloor\rho^*m\rfloor\big)=0$.
Yet
$$P\big(N_t\big(s, \lfloor\rho^*m\rfloor\big)=0\big)
\,\leq\,
P\big(\tau_0\leq t-s\,\big|\,N_0
=\lfloor\rho^*m\rfloor\big)\,.
$$
Let $V^*>0$ be such that
$V^*<V(\rho^*,0)$.
Let $\delta>0$ as in lemma~\ref{appov}.
By proposition~\ref{cghy}, there exists $m_0\geq 1$ such that,
for $m\geq m_0$, 
$$
P\big(\tau_0\leq t-s\,\big|\,N_0
=\lfloor\rho^*m\rfloor\big)
\,\leq\,
(t-s)
\exp(-V^*m)\,.$$
We sum this inequality over $s,t$ such that
$s < t\leq\exp(V^*m/4)$
to obtain
$$\displaylines{
P\Big(\exists\,t\leq \exp(V^*m/4)\quad
\max_{1\leq i\leq m}f\big(X_t(i)\big)
\,<\,
\max_{0\leq s\leq t}\Lambda\big(X_s,\lfloor\rho^*m\rfloor\big)
\Big)\hfill\cr
\,\leq\,
\sum_{0\leq s < t\leq\exp(V^*m/4)}
(t-s) \exp(-V^*m)
\,\leq\,
\exp(-V^*m/4)\,.
}$$
This proves theorem~\ref{thcsr}.
\subsection{Proof of theorems~\ref{htopt}, \ref{ctopt}}
For $E$ a subset of $\zulm$, we define the entrance time
of the genetic algorithm in~$E$ as
$$\tau(E)\,=\,
\inf\,\big\{\,n\geq 0: X_n\in E
\big\}\,.$$
For $\lambda\in\mathbb R$ and $k\geq 1$, we define
$L(\lambda,k)$ as the set of the populations
containing at least $k$ chromosomes with a fitness
larger than or equal to~$\lambda$:
$$L(\lambda,k)\,=\,\Big\{\,x\in\zulm:
\Lambda(x,k)\geq \lambda\,\Big\}\,.$$
Recall that the quantity $\Delta(\lambda,\gamma)$ is defined
just before~theorem~\ref{htopt}.
\begin{lemma}
\label{laibd}
Let $\pi>1$. Let 
$V^*<V(\rho^*,0)$. There exist positive constants
$\delta,\kappa^*,m_0$ which depend only on $\pi$ and $V^*$ such that
$$\displaylines{
\forall\lambda<\gamma\quad
\forall x_0\in L\big(\lambda,
\lfloor(\rho^*-\delta)m\rfloor\big)
\quad\forall m\geq m_0
\hfill\cr
P\Big(
\tau\big(
L\big(\gamma,\lfloor(\rho^*-\delta)m\rfloor\big)
\big)
>
\frac{\kappa^*}{2}(\ln m)m^2
(p_M)^{-\Delta(\lambda,\gamma)}
\,\big|\,X_0=x_0\Big)
\,\leq\,\hfill\cr
\hfill
\kappa^*
(\ln m)
m^2
(p_M)^{-\Delta(\lambda,\gamma)}
\exp\big(-mV^*\big)
\,.}$$
\end{lemma}
\begin{proof}
Let $\pi>1$. Let 
$V^*<V(\rho^*,0)$.
Let $\delta>0$ and $m_0\geq 1$
be associated to $V^*$ as in corollary~\ref{ghy}.
Let $p_1>0$
be associated to $\delta$ as in 
proposition~\ref{hfaj}.
Let us fix $\lambda<\gamma\in\mathbb R$. 
By the very definition of $\Delta(\lambda,\gamma)$, 
any chromosome in $L(\lambda)$ can be transformed with
at most
$\Delta(\lambda,\gamma)$ 
mutations
into 
a chromosome of $L(\gamma)$, and the corresponding probability
is bounded from below by
$$(p_M)^{\Delta(\lambda,\gamma)}
(1-p_M)^{\ell-\Delta(\lambda,\gamma)}\,.$$
Let $x_0$ belong to
$L\big(\lambda,1\big)$.
The population~$x_0$ contains at least one 
chromosome belonging to
$L(\lambda)$. By estimating the probability
that this chromosome is selected, that there is no crossover on it,
and that it
is transformed by mutation
into 
a chromosome in $L(\gamma)$, we obtain that
$$P\big(X_1
\in L(\gamma,1)\,\big|\,X_0=x_0\big)\,\geq\,
F_m(m)
(1-p_C)
(p_M)^{\Delta(\lambda,\gamma)}
(1-p_M)^{\ell-\Delta(\lambda,\gamma)}\,.$$
The hypothesis on $F_m$ (see section~\ref{genhy}) implies that,
for $m$ large enough,
$F_m(m)\geq \sigma/(2m)$. Since we have also $\pi>1$,
then
$$P\big(X_1
\in L(\gamma,1)\,\big|\,X_0=x_0\big)\,\geq\,
\frac{1}{2m}
(p_M)^{\Delta(\lambda,\gamma)}
\,.$$
Suppose next that $X_1$
belong to
$L\big(\gamma,1\big)$.
Then the population~$X_1$ contains at least one 
chromosome in $L(\gamma)$, hence
$N(X_1,\gamma)\geq 1$, and
by
proposition~\ref{couplnl}, we have 
$$\forall n\geq 1\qquad N(X_n,\gamma)\,\geq\,
N_{n}\big(1,1)\,.$$
Thus, by proposition~\ref{hfaj}, 
there exist $\kappa>0$ and $p_1>0$ 
such that,
for any $x_1\in L\big(\gamma,1\big)$, and any $m\geq 1$,
\begin{multline*}
P\big(
\tau\big(
L\big(\gamma,\lfloor(\rho^*-\delta)m\rfloor\big)
\big)
\leq\kappa\ln m+1
\,\big|\,X_1=x_1\big)
\cr
\,\geq\,
P\big(T({\rho^*-\delta})\leq\kappa\ln m+1\,\big|\,N_1=1\big)
\,\geq\,
p_1\,.
\end{multline*}
Combining the previous bounds, we get
$$\displaylines{
\forall x_0\in L\big(\lambda,1\big)\hfill\cr
P\big(
\tau\big(
L\big(\gamma,\lfloor(\rho^*-\delta)m\rfloor\big)
\big)
\leq\kappa\ln m+1
\,\big|\,X_0=x_0\big)
\,\geq\,
p_1\,
\frac{1}{2m}
(p_M)^{\Delta(\lambda,\gamma)}
\,.}
$$
We decompose
$\{\,1,\dots,n\,\}$ into subintervals of length
$\lfloor\kappa\ln m+1\rfloor$ and we use the
previous estimate: we obtain that
for any
$x_0\in L(\lambda,1)$ and $n\geq 1$,
$$\displaylines{
P\big(
\tau\big(
L(\lambda,1)
\big)>n,\,
\tau\big(
L\big(\gamma,\lfloor(\rho^*-\delta)m\rfloor\big)
\big)
>n
\,\big|\,X_0=x_0\big)
\,\leq\,\hfill\cr
\hfill\Big(1-p_1\,
\frac{1}{2m}
(p_M)^{\Delta(\lambda,\gamma)}
\Big)^{
\Big\lfloor
\displaystyle
\frac{n}{
\lfloor\kappa\ln m+1\rfloor
}
\Big\rfloor
}
\,.}
$$
Let next
$x_0\in L\big(\lambda,
\lfloor(\rho^*-\delta)m\rfloor\big)$.
We use
proposition~\ref{cghy} 
and the previous estimate: there exists $m_0\geq 1$ such that,
for $n\geq 1$,
$$\displaylines{
P\big(
\tau\big(
L\big(\gamma,\lfloor(\rho^*-\delta)m\rfloor\big)
\big)
>n
\,\big|\,X_0=x_0\big)
\,\leq\,
P\big(
\tau\big(
L(\lambda,1)
\big)\leq n
\,\big|\,X_0=x_0\big)
\hfill\cr
\hfill
+
P\big(
\tau\big(
L(\lambda,1)
\big)>n,\,
\tau\big(
L\big(\gamma,\lfloor(\rho^*-\delta)m\rfloor\big)
\big)
>n
\,\big|\,X_0=x_0\big)
\cr
\,\leq\,
n\exp\big(-mV^*\big)
+\exp\Big(-p_1\,
\frac{1}{2m}
(p_M)^{\Delta(\lambda,\gamma)}
\Big\lfloor
\displaystyle
\frac{n}{
\lfloor\kappa\ln m+1\rfloor
}
\Big\rfloor
\Big)\,.
}$$
We choose
$$\kappa^*=
\frac{8\kappa V^*}{p_1}\,,\qquad
n= \frac{1}{2}\kappa^*
(\ln m)
m^2
(p_M)^{-\Delta(\lambda,\gamma)}\,,$$
and for $m$ large enough, we obtain the estimate
stated in the lemma.
\end{proof}
Let now 
$\lambda_0<\cdots<\lambda_r$ 
be an increasing sequence such that
$$\lambda_0\,=\,\min\,\big\{\,f(u):u\in\zul\,\big\}\,,\quad
\lambda_r\,=\,\max\,\big\{\,f(u):u\in\zul\,\big\}\,.$$
Let $\pi>1$, 
$V^*<V(\rho^*,0)$ and let
$\delta,\kappa^*,m_0$ as given by lemma~\ref{laibd}.
We write
$$\tau^*\,\leq\,
\sum_{k=1}^r
\Big(
\tau\big( L\big(\lambda_{k},\lfloor(\rho^*-\delta)m\rfloor\big) \big)
-
\tau\big( L\big(\lambda_{k-1},\lfloor(\rho^*-\delta)m\rfloor\big) \big)
\Big)\,.
$$
Thus, for any starting population $x_0$, we have
$$\displaylines{
P\Big(
\tau^*\leq 
\frac{\kappa^*}{2}(\ln m)m^2
\sum_{k=1}^r
(p_M)^{-\Delta(\lambda_{k-1},\lambda_k)}
\,\big|\,X_0=x_0\Big)
\,\geq\,\hfill\cr
P\Big(
\forall k\in\{\,1,\dots,r\,\}\quad
\tau\big( L\big(\lambda_{k},\lfloor(\rho^*-\delta)m\rfloor\big) \big)
-
\tau\big( L\big(\lambda_{k-1},\lfloor(\rho^*-\delta)m\rfloor\big) \big)
\hfill\cr
\hfill
\leq
\frac{\kappa^*}{2}(\ln m)m^2
(p_M)^{-\Delta(\lambda_{k-1},\lambda_k)}
\,\big|\,X_0=x_0\Big)
\,.}$$
To control this last probability, we use repeatedly
the Markov property and the 
estimate of 
lemma~\ref{laibd}.
Finally, we obtain
$$\displaylines{
P\Big(
\tau^*\leq 
\frac{\kappa^*}{2}(\ln m)m^2
\sum_{k=1}^r
(p_M)^{-\Delta(\lambda_{k-1},\lambda_k)}
\,\big|\,X_0=x_0\Big)
\,\geq\,\hfill\cr
\,\geq\,
\prod_{k=1}^r
\Big(1-\kappa^*
(\ln m)
m^2
(p_M)^{-\Delta(\lambda_{k-1},\lambda_k)}
\exp\big(-mV^*\big)\Big)
\,.}$$
We complete now the proof 
of~theorem~\ref{htopt}.
We suppose that $m\geq c^*\ell\ln\ell$ and that $p_M\geq c^*/\ell$.
Since 
$\Delta(\lambda_{k-1},\lambda_k)\leq\ell$ for any $k$, we have
for $m$ large enough
$$\forall k\in\{\,1,\dots,r\,\}\qquad
\kappa^*
(\ln m)
m^2
(p_M)^{-\Delta(\lambda_{k-1},\lambda_k)}
\,\leq\,
\exp\big(2m/c^*\big)
\,.$$
We take $c^*$ such that $2/c^*<V^*/2$ and we obtain,
for $m$ large enough,
$$\displaylines{
P\Big(
\tau^*\leq 
\frac{\kappa^*}{2}(\ln m)m^2
\sum_{k=1}^r
(p_M)^{-\Delta(\lambda_{k-1},\lambda_k)}
\,\big|\,X_0=x_0\Big)
\,\hfill\cr
\,\geq\,
\Big(1-
\exp\big(-mV^*/2\big)\Big)^{\textstyle 2^\ell}
\,\geq\,\frac{1}{2}
\,.}$$
The bound on the expectation of $\tau^*$ is a consequence
of this estimate and lemma~\ref{ebound}.
This completes the proof 
of~theorem~\ref{htopt}.

We complete next the proof of~theorem~\ref{ctopt}. The proof is
a variant of the previous argument. 
We take $c^*$ such that 
$$c^*>\kappa^*\,,\quad
c^*>\frac{4}{V^*}\,,\quad
c^*>\frac{2\gamma}{V^*\Delta}\,.$$
We suppose that $m\geq c^*\Delta\ln\ell$ and that $p_M\geq c^*/\ell$.
Let $k\in\{\,1,\dots,r\,\}$.
Since 
$\Delta(\lambda_{k-1},\lambda_k)\leq\Delta$,
we have
for $m$ large enough,
$$
\kappa^*
(\ln m)
m^2
(p_M)^{-\Delta(\lambda_{k-1},\lambda_k)}
\,\leq\,
(\ln m)
m^2
\ell^{\Delta}
\,\leq\,
\exp\big(2m/c^*\big)
\,.$$
It follows that
$$\displaylines{
P\Big(
\tau^*\leq 
\frac{1}{2}(\ln m)m^2
\ell^{\gamma+\Delta}
\,\big|\,X_0=x_0\Big)
\,\geq\,
\Big(1-
\exp\big(-mV^*/2\big)\Big)^{\textstyle \ell^\gamma}
\,\geq\,\frac{1}{2}
\,.}$$
We conclude as before with the help of
lemma~\ref{ebound}.

\subsection{Proof of theorem~\ref{invpt}}
Let $V^*<V(\rho^*,0)$ and 
let $\delta>0$ as in lemma~\ref{appov}.
Let
$$\lambda^*\,=\,
\max\,\big\{\,f(u):u\in\zul\,\big\}\,.$$
We apply the estimate on the invariant measure given in 
lemma~\ref{bdinv} with the following sets:
$$V\,=\,\Lambda\big(
\lambda^*,
\lfloor(\rho^*-\delta)m\rfloor\big)\,,\qquad
G\,=\,\Lambda\big(
\lambda^*,
1
\big)^c
\,.
$$
We obtain
$$\mu(G)\,\leq\,
\sup_{x\in V}\,
P\big(\tau_G<\tau_V\,\big|\,X_0=x\big)\,
\sup_{y\in G}\,
E\big(\tau_V\,\big|\,X_0=y\big)\,.$$
Let $x\in V$ and $n\geq 1$. We have
$$
P\big(\tau_G<\tau_V\,\big|\,X_0=x\big)\,\leq\,
P\big(\tau_G\leq n\,\big|\,X_0=x\big)\,+\,
P\big(n<\tau_G<\tau_V\,\big|\,X_0=x\big)\,.$$
We estimate separately each term.
Let
$i=N\big(x, \lambda^* \big)$.
Then
$i\geq
\lfloor(\rho^*-\delta)m\rfloor$ and
$$\displaylines{
P\big(\tau_G\leq n\,\big|\,X_0=x\big)\,=\,
P\big(\exists \,k\leq n\quad
N\big(X_k,
\lambda^*
\big)=0
\,\big|\,X_0=x
\Big)\,.}$$
From proposition~\ref{couplnl}, if $X_0=x$,
we have
$$\forall k\geq 0\qquad
N\big(X_k,
\lambda^*
\big)
\,\geq\,
N_{k}(0,i)\,.$$
Applying proposition~\ref{cghy}, we obtain, for $m$ large enough,
$$\displaylines{
P\big(\tau_G\leq n\,\big|\,X_0=x\big)\,\leq\,
P\big(\tau_0\leq n\,\big|\,N_0=i\big)\,\leq\,
n\exp(-mV^*)\,.}$$
For the second term, we remark that, on the event
$\{\,n<\tau_G<\tau_V\,\}$, we have that
$$\forall k\in\{\,1,\dots,n\,\}\qquad
1\,\leq \, N_k(0,1)\,<\,
\lfloor(\rho^*-\delta)m\rfloor\,.$$
We use proposition~\ref{hfaj} and we decompose
$\{\,1,\dots,n\,\}$ into subintervals of length
$\kappa\ln m+1$ to conclude that
$$P\big(n<\tau_G<\tau_V\,\big|\,X_0=x\big)\,\leq\,
(1-p_1)^{\big\lfloor
\textstyle
\frac{n}{\kappa\ln m +1}\big\rfloor}\,.$$
Putting together the previous inequalities, we have,
for $n\geq 1$,
$$
P\big(\tau_G<\tau_V\,\big|\,X_0=x\big)\,\leq\,
n\exp(-mV^*)\,+\,
(1-p_1)^{\big\lfloor
\textstyle
\frac{n}{\kappa\ln m +1}\big\rfloor}\,.$$
We take $n=m^2$ and we conclude that, for $m$ large enough,
$$
P\big(\tau_G<\tau_V\,\big|\,X_0=x\big)\,\leq\,
2m^2\exp(-mV^*)\,.$$
Let next $y\in G$. Using the bounds obtained in theorem~\ref{htopt}
with $r=1$ and $\Delta=\ell$, we have
$$\forall y\in G\qquad
E(
\tau^*
\,\big|\,X_0=y\big)\,\leq\,
2+\kappa^* (\ln m) m^2
\, (p_M)^{-\ell}
\,.$$
Inspecting the proof of theorem~\ref{htopt}, we see that we have in
fact proved this estimate for $\tau_V$ (this is a little stronger
since $\tau_V\geq \tau^*$).
Thus, 
for $p_M\,\geq\,{c^*}/{\ell}$,
$m\geq c^*\ell\ln\ell$ and $m$ large enough,
we have
$$\displaylines{
\mu(G)\,\leq\,
2m^2\exp(-mV^*)\times
\Big(2+\kappa^* (\ln m) m^2
\, (p_M)^{-\ell}\Big)\cr
\,\leq\, m^5\exp\big(-\ell\ln p_M-mV^*\big)\,.}$$
We choose $c^*$ large enough in order to obtain
the conclusion of theorem~\ref{invpt}.

\appendix
\section{Markov chains}
We recall here a few classical facts on Markov chains with 
finite state space.
This material can be found in any reference book on Markov chains,
for instance 
\cite{Breiman}, 
\cite{Fe1},
\cite{SH}. 
In the sequel, we consider
a discrete time Markov chain 
$(X_t)_{t\geq 0}$ 
with
values in a finite state space $\cE$ and
with transition matrix
$(p(x,y))_{x,y\in \cE}$.
\medskip

\noindent
{\bf Invariant probability measure.}
If the Markov chain is
irreducible and aperiodic,
then it admits a unique
invariant probability 
measure 
$\mu$, i.e., the set of equations
$$\mu(y)\,=\,\sum_{x\in\cE}\mu(x)\,p(x,y)\,,
\qquad
y\in\cE\,,$$
admits a unique solution.
\medskip

\noindent
{\bf Representation formula.}
Let us suppose that 
the Markov chain 
$(X_t)_{t\geq 0}$ is irreducible and aperiodic.
Let $\mu$ be the
invariant probability 
measure of
$(X_t)_{t\geq 0}$.
Let $V$ be a non--empty subset of $\cE$.
We define
$$\tau_V\,=\,\min\,\big\{\,n\geq 1: X_n\in V\,\big\}\,.$$
We have then, for any subset $G$ of $\cE$,
$$\mu(G)\,=\,\frac{1}{\mu(V)}\int_V
d\mu(x)\,
E\Big(\sum_{k=0}^{\tau_V-1}1_G(X_k)\,\Big|\,X_0=x\Big)\,.$$
This formula can be found in the book of Freidlin and Wentzell
(see chapter~$6$, section~$4$ of \cite{FW}), where it is attributed
to Khas'minskii, and in the book of Kifer \cite{KI,KID}.
\begin{lemma}\label{bdinv} 
For any subsets $V,G$ of $\cE$, we have
$$\mu(G)\,\leq\,
\sup_{x\in V}\,
P\big(\tau_G<\tau_V\,\big|\,X_0=x\big)\,
\sup_{y\in G}\,
E\big(\tau_V\,\big|\,X_0=y\big)\,.$$
\end{lemma}
\begin{proof}
From the representation formula for the invariant measure,
we obtain that
$$\mu(G)\,\leq\,
\sup_{x\in V}\,
E\Big(\sum_{k=0}^{\tau_V-1}1_G(X_k)\,\Big|\,X_0=x\Big)\,.$$
Let us try to bound the last expectation.
We denote by $E_x$ the expectation for the Markov chain
starting from~$x$.
We have
$$\displaylines{
E_x\Big(\sum_{k=0}^{\tau_V-1}1_G(X_k)\Big)
\,=\,
E_x\Big(
\sum_{y\in G}
1_{\tau_G<\tau_V}
1_{X_{\tau_G}=y}
\sum_{k=0}^{\tau_V-1}1_G(X_k)\Big)
\cr
\,=\,
\sum_{y\in G}
E_y\Big(
\sum_{k=0}^{\tau_V-1}1_G(X_k)\Big)
P_x\big(
\tau_G<\tau_V,\,
{X_{\tau_G}=y}
\big)
\cr
\,\leq\,
\sum_{y\in G}
E_y\big(
{\tau_V}\big)
P_x\big(
\tau_G<\tau_V,\,
{X_{\tau_G}=y}
\big)\cr
\,\leq\,
\sup_{y\in G}
E_y\big(
{\tau_V}\big)
\,
P_x\big(
\tau_G<\tau_V
\big)
\,.}$$
Taking the supremum over $x\in V$, we obtain the inequality stated in the lemma.
\end{proof}

\noindent
For $x\in\cE$, we define
$$T(x)
\,=\,
\inf\,\big\{\,n\geq 0: X_n=x
\big\}\,.$$
\begin{lemma}\label{indepday} 
Let $k\geq 1$ and let $x_1,\dots,x_k$ be $k$ distinct points of~$\cE$.
The random variables
$X_{T(x_1)+1},\dots,
X_{T(x_k)+1}$ are independent.
\end{lemma}
\begin{proof}
We do the proof by induction over~$k$.
For $k=1$, there is nothing to prove.
Let $k\geq 2$ and
suppose that the result has been proved until rank $k-1$.
Let $x_1,\dots,x_k$ be $k$ distinct points of~$\cE$.
Let $y_1,\dots,y_k$ be $k$ points of~$\cE$.
Let us set
$$T\,=\,\min\,\big\{\,T(x_i):1\leq i\leq k\,\big\}\,.$$
Let us compute
$$\displaylines{
P\big(
X_{T(x_1)+1}=y_1,\dots,
X_{T(x_k)+1}=y_k\big)
\,=\,\hfill\cr
\sum_{1\leq i\leq k}
P\big(
X_{T(x_1)+1}=y_1,\dots,
X_{T(x_k)+1}=y_k,T=T(x_i)\big)
\cr
\,=\,
\sum_{1\leq i\leq k}
P\big(
X_{T(x_1)+1}=y_1,\dots,
X_{T(x_k)+1}=y_k\,|\,
T=T(x_i)\big)
P\big(T=T(x_i)\big)
\cr
\,=\,
\sum_{1\leq i\leq k}
P\big(\forall j\neq i
\quad
X_{T(x_j)+1}=y_j,\,X_1=y_i
\,|\,
X_0=x_i\big)
P\big(T=T(x_i)\big)
\cr
\,=\,
\sum_{1\leq i\leq k}
p(x_i,y_i)
P\big(\forall j\neq i
\quad
X_{T(x_j)+1}=y_j
\,|\,
X_0=y_i\big)
P\big(T=T(x_i)\big)\,.
}$$
We use the induction hypothesis:
$$P\big(\forall j\neq i
\quad
X_{T(x_j)+1}=y_j
\,|\,
X_0=y_i\big)
\,=\,\prod_{j\neq i}p(x_j,y_j)\,.$$
Reporting in the sum, we get
$$\displaylines{
P\big(
X_{T(x_1)+1}=y_1,\dots,
X_{T(x_k)+1}=y_k\big)
\,=\,\hfill\cr
\,=\,
\sum_{1\leq i\leq k}
\prod_{1\leq j\leq k}p(x_j,y_j)
P\big(T=T(x_i)\big)
\,=\,
\prod_{1\leq j\leq k}p(x_j,y_j)
\,.
}$$
This completes the induction step and the proof.
\end{proof}
\begin{lemma}\label{ebound} 
Let $\tau$ be a stopping time associated to 
the Markov chain 
$(X_t)_{t\geq 0}$. If there exists an integer $k$ and $\beta$ positive such that
$$\forall x\in\cE\qquad
P\big(\tau\leq k
\,|\,
X_0=x\big)\,\geq\,\beta\,,$$
then we have
$$\forall x\in\cE\qquad
E\big(\tau
\,|\,
X_0=x\big)\,\leq\,
\frac{k}{\beta}
\,.$$
\end{lemma}
\begin{proof}
Reversing the inequality, we have
$$\forall x\in\cE\qquad
P\big(\tau>k
\,|\,
X_0=x\big)\,\leq\,1-\beta\,.$$
Since the bound is uniform with respect to~$x$, we prove by induction
on~$n$ that
$$\forall x\in\cE\quad\forall n\geq 1\qquad
P\big(\tau>nk
\,|\,
X_0=x\big)\,\leq\,(1-\beta)^n\,.$$
We computation next the expectation of $\tau$ as follows:
for $x\in\cE$,
$$\displaylines{
E\big(\tau
\,|\,
X_0=x\big)\,=\,
\sum_{n=0}^\infty
P\big(\tau>n \,|\, X_0=x\big)\hfill\cr
\,\leq\,
\sum_{i=0}^\infty
\sum_{l=0}^{k-1}
P\big(\tau>ik
+l \,|\, X_0=x\big)
\cr
\,\leq\,
\sum_{i=0}^\infty
k
P\big(\tau>ik
\,|\, X_0=x\big)
\cr
\,\leq\,
\sum_{i=0}^\infty
k(1-\beta)^i
\,\leq\,
\frac{k
}{\beta}
}$$
as requested.
\end{proof}

\section{Monotonicity}
We first recall some standard definitions concerning monotonicity 
and coupling
for stochastic processes.
A classical reference
is Liggett's book
\cite{LIG},
especially for applications to particle systems. 
In the next two definitions,
we consider
a discrete time Markov chain 
$(X_n)_{n\geq 0}$ 
with
values in a space $\cE$.
We suppose that the state space $\cE$ is finite and
that it is equipped with a partial order $\leq$.
A function $f:\cE\to\R$ is non--decreasing if
$$\forall x,y\in\cE\qquad
x\leq y\quad\Rightarrow\quad f(x)\leq f(y)\,.$$
\begin{definition}
The Markov chain
$(X_n)_{n\geq 0}$ is said to be monotone if, 
for any non--decreasing function $f$, the function
$$x\in\cE\mapsto E\big(f(X_n)\,|\,X_0=x\big)$$
is non--decreasing.
\end{definition}
A natural way to prove monotonicity is to construct an adequate coupling.
\begin{definition}
A coupling 
for the Markov chain
$(X_n)_{n\geq 0}$ 
is a family of processes
$(X_n^x)_{n\geq 0}$
indexed by 
$x\in\cE$, which are all defined on the same probability space, and such that, 
for $x\in\cE$, the process
$(X_n^x)_{n\geq 0}$ is the Markov chain 
$(X_n)_{n\geq 0}$ 
starting from $X_0=x$.
The coupling is said to be monotone if
$$\forall x,y\in\cE\qquad
x\leq y\quad\Rightarrow\quad \forall n\geq 1\qquad X_n^x\leq X_n^y\,.$$
%
%
%
\end{definition}
If there exists a monotone coupling, 
then the 
Markov chain
is monotone.

\section{Stochastic domination}
\label{apstoc}
Let $\mu,\nu$ be two probability measures on $\mathbb R$.
We say that $\nu$ stochastically dominates $\mu$,
which we denote by $\mu\preceq\nu$, if
for any non--decreasing positive function $f$, we have
$\mu(f)\leq \nu(f)$.
\begin{lemma}\label{binopoi}
Let $n\geq 1$, $p\in [0,1]$, $\lambda>0$ be such that
$(1-p)^n\geq\exp(-\lambda)$.
Then
the binomial law
$\cB(n,p)$ of parameters $n,p$
is stochastically dominated by the Poisson 
law $\cP(\lambda)$ of parameter $\lambda$.
\end{lemma}
\begin{proof}
Let $X_1,\dots,X_n$ be independent random variables
with common law the Poisson law of parameter $-\ln(1-p)$.
Let $Y$ be a further random variable, independent of $X_1,\dots,X_n$,
with law 
the Poisson law of parameter $\lambda-n\ln(1-p)$.
Obviously, we have
$$Y+X_1+\dots+X_n\,\geq\,\min(X_1,1)+\dots+\min(X_n,1)\,.$$
Moreover, the law of the lefthand side is the Poisson
law of parameter~$\lambda$, while the law of the righthand side
is the binomial law $\cB(n,p)$.
\end{proof}
\begin{lemma}\label{poisstail}
Let $\lambda>0$ and let $Y$ be a random variable with law the 
Poisson law $\cP(\lambda)$ of parameter~$\lambda$.
For any $t\geq\lambda$, we have
$$P(Y\geq t)\,\leq\,\Big(\frac{\lambda e}{t}\Big)^t\,.$$
\end{lemma}
\begin{proof}
We write
\begin{multline*}
P(Y\geq t)\,=\,
\sum_{k\geq t}
\frac{\lambda^k}{k!}\exp(-\lambda)
\,=\,
\sum_{k\geq t}
\frac{\lambda^{k-t}}{k!}\exp(-\lambda)\lambda^t
\cr
\,\leq\,
\sum_{k\geq t}
\frac{t^{k-t}}{k!}\exp(-\lambda)\lambda^t
\,\leq\,
\Big(\frac{\lambda e}{t}\Big)^t\,.
\end{multline*}
\end{proof}

\noindent
Let $Y$ be a random variable following the Poisson law $\cP(\lambda)$.
For any $t\in\R$, we have
$$\Lambda_Y(t)\,=\,\ln E\big(\exp(tY)\big)\,=\,
\ln\Big(
\sum_{k=0}^\infty
\frac{\lambda^k}{k!}\exp(-\lambda+kt)\Big)\,=\,
\lambda \big(\exp(t)-1\big)\,.$$
For any $\alpha,t\in\R$, 
$$\Lambda_{\alpha Y}(t)\,=\,
\Lambda_{Y}(\alpha t)
\,=\,\lambda \big(\exp(\alpha t)-1\big)\,.$$
Let us compute
the Fenchel--Legendre transform
$\Lambda_{\alpha Y}^*$.
By definition, for $x\in\R$,
$$\Lambda_{\alpha Y}^*(x)\,=\,
\sup_{t\in\R}\Big(tx-
\lambda \big(\exp(\alpha t)-1\big)\Big)\,.$$
The maximum is attained at $t=(1/\alpha)\ln (x/(\lambda\alpha))$,
hence
$$\Lambda_{\alpha Y}^*(x)\,=\,
\frac{x}{\alpha}\ln\Big(\frac{x}{\lambda\alpha}\Big)
-
\frac{x}{\alpha}+\lambda\,.$$
\begin{lemma}
\label{bnpp}
Let $p\in[0,1]$ and let $n\geq 1$. Let $X$ be a random variable following
the binomial law $\cB(n,p)$. 
Let $Y$ be a random variable following
the Poisson law $\cP(np)$. For any $\alpha\in\R$,
we have
$\Lambda^*_{\alpha X}\geq \Lambda^*_{\alpha Y}$.
\end{lemma}
\begin{proof}
For any $t\in\R$, we have
$$\Lambda_X(t)\,=\,\ln E\big(\exp(tX)\big)\,=\,
n\ln\big(1-p+p\exp(t)\big)\,\leq\,np\big(\exp(t)-1\big)\,.$$
For any $\alpha,t\in\R$, 
$$\Lambda_{\alpha X}(t)\,=\,
\Lambda_{X}(\alpha t)
\,\leq\,np\big(\exp(\alpha t)-1\big)\,.$$
Thus, taking $\lambda=np$, we conclude that
$$\forall t\in\R\qquad
\Lambda_{\alpha X}(t)\,\leq\, \Lambda_{\alpha Y}(t)\,.$$
Taking the Fenchel--Legendre transform, we obtain
$$\forall x\in\R\qquad
\Lambda^*_{\alpha X}(x)\,\geq\, \Lambda^*_{\alpha Y}(x)\,$$
as required.
\end{proof}

\section{Binomial estimate}
We recall a basic estimate for the binomial coefficients. 
\begin{lemma}
\label{cnk}
For any $n\geq 1$, 
any $k\in\{\,0,\dots, n\,\}$, we have
$$
\Big|
\ln
\frac{n!}{ k!(n-k)!}
+
k\ln \frac{k}{ n}
+(n-k)\ln \frac{n-k}{ n}
\Big|
\,\leq\,
2\ln n+3
\,.$$
\end{lemma}
\begin{proof}
The proof is standard (see for instance \cite{EL}).
Setting, for $n\in\mathbb N$, $\phi(n)=\ln n!-n\ln n+n$,
we have 
\begin{multline*}
\ln\frac{n!}{ k!(n-k)!}
\,=\,\ln n!-\ln k!
-\ln (n-k)!
\cr
\,=\,n\ln n-n+\phi(n)
-\big(k\ln k-k+\phi(k)\big)
-\big((n-k)\ln (n-k)-(n-k)+\phi(n-k)\big)
\\
\,=\,
-{k}\ln \frac{k}{n}
-{(n-k)}\ln \frac{n-k}{n}
+\phi(n)-\phi(k)-\phi(n-k)
\,.
\end{multline*}
Comparing the discrete sum 
$\ln n!=\sum_{1\leq k\leq n}\ln k$ 
to the integral 
$\int_1^n\ln x\,dx$,
we see that
$1\leq \phi(n)\leq \ln n +2$ for all $n\geq 1$. On one hand,
$$ \phi(n)-\phi(k)-\phi(n-k)\,\leq\, \ln n\,,$$
on the other hand,
$$ \phi(n)-\phi(k)-\phi(n-k)\,\geq\, 1-
(\ln k+2)
-(\ln (n-k)+2)
\,\geq\,
-3-2\ln n\,$$
and we have the desired inequalities.
\end{proof}

\section{Exponential inequalities}
\label{fixed}
\noindent
{\bf Hoeffding's inequality}.
We state Hoeffding's inequality for Bernoulli random variables
\cite{HOE}.
Suppose that $X$ is a random variable with law the binomial law
$\cB(n,p)$. We have
$$\forall t<np\qquad P(X<t)\,\leq\,\exp\Big(-\frac{2}{n}\big(np-t)^2\Big)\,.
$$
\noindent
{\bf Tchebytcheff exponential inequality}.
Let $X_1,\dots,X_n$ be i.i.d. random variables with common law~$\mu$.
Let $\Lambda$ be the Log--Laplace of~$\mu$, defined by
$$\forall t\in\R\qquad\Lambda(t)\,=
\ln\Big(\int_{\mathbb R}\exp(ts)\,d\mu(s)\Big)\,.$$
Let $\Lambda^*$ be the Cram\'er transform of~$\mu$, defined by
$$\forall x\in\R\qquad
\Lambda^*(x)\,=\,\sup_{t\in\R}\big(tx-\Lambda(t)\big)
\,.$$
We suppose that $\mu$ is integrable and we denote by $m$ its mean,
i.e., $m=\int_\R x\,d\mu(x)$. 
We have then 
$$\forall x\geq m\qquad
P\Big(\frac
{1}{n}\big(
{X_1+\cdots+X_n}\big)\,\geq\,x\Big)\,\leq\,\exp\big(-n\Lambda^*(x)\big)\,.$$

\vfill\eject
\bibliographystyle{plain}
\bibliography{sga}
 \thispagestyle{empty}
\vfill\eject
\tableofcontents

\end{document}